\documentclass{amsart}

\usepackage[
  top=1in,
  bottom=1in,
  left=1.25in,
  right=1.25in,
  heightrounded,
]{geometry}

\author{Kenneth T-R McLaughlin}
\address{Colorado State University}

\author{Patrik V. Nabelek}
\address{Oregon State University}
\email{patrik@alyrica.net}

\date{\today}

\usepackage{amsmath}
\usepackage{amssymb}
\usepackage{amsthm}
\usepackage{graphicx}

\newtheorem{thm}{Theorem}[section]
\newtheorem{lma}[thm]{Lemma}
\newtheorem{prop}[thm]{Proposition}
\newtheorem{defn}[thm]{Definition}
\newtheorem{clm}[thm]{Claim}
\newtheorem{rhp}[thm]{Riemann--Hilbert Problem}
\newtheorem{cor}[thm]{Corollary}

\newtheorem{rmk}[thm]{Remark}

\newcommand{\bmat}{\begin{pmatrix}}
\newcommand{\emat}{\end{pmatrix}}
\newcommand\numberthis{\addtocounter{equation}{1}\tag{\theequation}}

\begin{document}

\title[Infinite Gap Riemann--Hilbert Problem]{A Riemann--Hilbert Problem Approach to Periodic Infinite Gap Hill Operators and the Korteweg--de Vries Equation}

\maketitle

\begin{abstract}
We formulate the inverse spectral theory of infinite gap Hill's operators with bounded periodic potential as a Riemann--Hilbert problem on a typically infinite collection of spectral bands and gaps.  We establish a uniqueness theorem for this Riemann--Hilbert problem, which provides a new route to establishing unique determination of periodic potentials from spectral data. As the potential evolves according to the KdV equation, we use integrability to derive an associated Riemann--Hilbert problem with explicit time dependence.  Basic principles from the theory of Riemann--Hilbert problems yield a new characterization of spectra for periodic potentials in terms of the existence of a solution to a scalar Riemann--Hilbert problem, and we derive a similar condition on the spectrum for the temporal periodicity for an evolution under the KdV equation.
\end{abstract}

\section{Introduction}

\sloppy
A Hill's operator is a one-dimensional Schr\"{o}dinger operator with periodic potential,  \begin{equation}L : H^2(\mathbb{R}) \to L^2(\mathbb{R})\ , \end{equation} \begin{equation}L = -\frac{\partial^2}{\partial x^2} + u(x),\end{equation} where $u \in L^\infty (\mathbb{R})$ is real valued with period $T$.
The spectrum $\sigma(L)$ is a half line with a possibly countably infinite number of real open intervals (gaps) removed, with the endpoints of the gaps lying at interlaced periodic and antiperiodic real eigenvalues \cite{MaWi79}.
To be explicit, we write
\begin{equation} \sigma(L) = \bigcup_{n = 0}^\infty [\lambda_{2n},\lambda_{2n+1}] \end{equation}
where $\{\lambda_n\}$ is the nondecreasing sequence of the periodic/antiperiodic eigenvalues.
Without loss of generality we can assume $\lambda_0 = 0$, because if this is not the case we can always make the transformation to a new potential $\tilde u$ defined by $\tilde u(x) = u(x)-\lambda_0$.
It was proven by Simon that a half line with countably infinite gaps is the generic case \cite{Si76}.  

The 1-D Schr\"{o}dinger equation and various versions of its spectral and inverse spectral theory are elements of solution procedures for the Korteweg--de Vries (KdV) equation, 
\begin{eqnarray}
\label{eq:kdv}
u_{t} - 6 u u_{x} + u_{xxx} = 0 \ .
\end{eqnarray}
The operator $L$ with potential $u$ depending parametrically on $t$ was discovered to have spectra that remained constant as $u$ evolved, and led to the Lax-pair formalism \cite{GGKM, Lax}.  Briefly, the KdV equation (\ref{eq:kdv}) arises as the compatibility condition for a simultaneous solution of two differential equations,
\begin{eqnarray}
&&L \psi = \lambda \psi, \\
& & \psi_{t} = A \psi,
\end{eqnarray}
where the operator $A$ takes the form
\begin{eqnarray}
A  = -4 \frac{\partial^3}{\partial x^3} + 6 u(x,t) \frac{\partial}{\partial x} + 3 u_x(x,t).
\end{eqnarray}

When the potential $u$ is assumed to be rapidly decaying as $|x| \to \infty$, the scattering transformation from quantum mechanics completely linearizes the flow, in the sense that the reflection coefficient and the norming constants (associated to $L^{2}$ eigenvalues) evolve from their initial values in a simple manner, and the solution $u(x,t)$ is determined via the inverse scattering machinery.  This connection has led, over the past 50 years, to many results concerning the detailed analysis and asymptotic behavior of solutions to the KdV equation with decaying boundary conditions, including soliton and multi-soliton solutions, small dispersion behavior, and long-time asymptotic analysis.  All of this work relies implicitly on analytical results concerning the unique determination of the potential of the Schr\"{o}dinger equation from the associated scattering data.

When the potential $u$ is periodic, the spectrum $\sigma(L)$ is independent of $t$ but $\sigma(L)$ alone does not uniquely determine the potential $u$.  A well-posed inverse spectral problem is formulated by considering $\sigma(L)$ together with the spectrum of the operator  \begin{equation}L_0 : H_0^2 ([0,T]) \to L^2([0,T])\end{equation} defined as \begin{equation}L_0 = - \frac{\partial^2}{\partial x^2} + u(x),\end{equation} where $H_0^2 ( [0,T])$ is the Hilbert subspace of $H^2 ([0,T])$ functions with Dirichlet boundary conditions \cite{MckMo75}.
The spectrum $\sigma(L_0)$ consists of one eigenvalue $\mu_n$ in each (possibly degenerate) gap \cite{LeSa75}.  A Dirichlet eigenfunction on $(0,T)$ extends naturally to a solution on $\mathbb{R}$.  When $\mu_{n}$ is in the interior of a gap, the eigenfunction must either decay or grow as $x \to + \infty$.
We set $\sigma_n = 1$ when the eigenfunction decays, and $\sigma_n = -1$ when the eigenfunction grows.
We call $\sigma_n$ the signature of $\mu_n$.
For notational convenience, we set $\sigma_n = 0$ when $\mu_n$ is on a band end.
The potential $u$ is uniquely determined by $\sigma(L)$ together with $\sigma(L_{0})$ and the signature $\sigma_n$ for each $\mu_{n}$ in the interior of the gaps.  As $u$ evolves in time according to the KdV equation, the Dirichlet eigenvalues that reside within non-degenerate gaps are not constant, but evolve in time, crossing back and forth across the gap.  

The case of periodic solutions to the KdV equation is also well developed, and has an enormous literature.  
A complete understanding of the collection of all ``finite gap'' solutions (when the spectrum $\sigma(L)$ has only a finite number of gaps) was established in the 1970s in works of Novikov \cite{No74}, Marchenko \cite{Mar74}, Matveev \cite{Mat75}, Its--Matveev \cite{ItsMatv75a,ItsMatv75b,ItsMatv76}, Dubrovin \cite{Dubrovin75}, Lax \cite{Lax75}, and McKean--van Moerbeke \cite{MckMo75}.
For further details on the history of the finite gap theory see \cite{Beetal94, Mat08}.
A characterization of the spectra of infinite gap periodic Hill's operators was discovered by Mar\v{c}henko--Ostrovski\u{i} \cite{MaOs75}.
Analysis of the infinite-gap potentials continued with the work of McKean--Trubowitz \cite{MckTr76,MckTr78} who studied their geometry, constructed an infinite-gap hyperelliptic function theory, and showed that for periodic $C^\infty$ initial data the finite gap theory actually extends to the infinite gap potentials. Several different global coordinate transformations  have been constructed under which the KdV flow is mapped to periodic flow on an infinite-dimensional torus, and this has led to KAM-type results concerning perturbed KdV equations (see \cite{KaPo03} and the many references contained therein).  Ercolani and McKean \cite{ErMck90} discuss how the inverse scattering theory for Schr\"{o}dinger operators on the whole line with decaying potential and no bound states could be interpreted as a continuum genus limit of the infinite gap theory of McKean and Trubowitz. Fast methods for computing finite gap potentials using a Fourier series representation and fast Fourier transforms have also been developed by Osborne \cite{Osborne2010}, and he has applied these fast methods to time series analysis of ocean waves.

For infinite gap potentials, computation of the map from spectral variables to the infinite genus torus requires solving an infinite genus Jacobi inversion problem.
The solvability of the infinite genus Jacobi inversion problem leads to strong existence, uniqueness and qualitative results for the behavior of periodic solutions to the KdV equation.
However, the infinite dimensional Riemann matrix needs to be explicitly computed to use the Matveev--Its formula in the case of infinite genus hyper-elliptic surfaces.

In this paper we will present a Riemann--Hilbert problem formulation of the Cauchy problem for the KdV equation that applies to infinite genus initial conditions, and establish a uniqueness result for the solution of this Riemann-Hilbert problem.
The Riemann--Hilbert problem avoids the requirement that the Riemann matrix be computed by inverting a period matrix for a basis of Abelian differentials on an infinite genus hyper-elliptic curve.
The complete set of spectral data $\Sigma(u)$ can be computed if the fundamental matrix solution $Y(x,\lambda)$ is known.
The building blocks for this Riemann-Hilbert problem are the Bloch--Floquet solutions to Hill's equation; these are constructed in Section 2 as solutions $\psi^{+}$ and $\psi^{-}$ that satisfy the initial condition
\begin{eqnarray}
\psi^{\pm}(0, \lambda) = 1  , 
\end{eqnarray}
and the multiplicative condition 
\begin{eqnarray}
&&\psi^{+}(x+T, \lambda) = \rho(\lambda)\psi^{+}(x,\lambda), \\ 
&& \psi^{-} (x+T, \lambda) = \rho(\lambda)^{-1} \psi^{-}(x,\lambda)   , 
\end{eqnarray}
with $\rho(\lambda)$ being the root of the characteristic polynomial (\ref{eq:charpoly}) that satisfies $| \rho(\lambda) | < 1$.  We let $\Psi$ denote the matrix fundamental solution built out of these two independent solutions.  If it is assumed $\mu_n \in (\lambda_{2n-1},\lambda_{2n})$, then it follows from the representation (\ref{eq:bloflo}) appearing in section 2 and analyticity properties of the constituent function of this representation that the matrix $\Psi(x,\lambda)$ has a pole at each Dirichlet eigenvalue $\mu_{n}$; the poles could occur in either the first or second column, depending on whether $\psi^{+}$ or $\psi^{-}$ is singular at $\mu_{n}$.
This singular behavior is elaborated on in section 5.

In section 3 we establish the relevant asymptotic properties of the various solutions of Hill's differential equation that are used in subsequent sections.  The eventual uniqueness theorem established in Section 5 requires the use of the  Phr\'{a}gmen--Lindel\"{o}f Theorem.  To apply the Phr\'{a}gmen--Lindel\"{o}f theorem, one must establish a bound of the form $|f(\lambda)| \le M e^{c|\lambda|^p}$ in the complex plane.
In section 4 we prove some propositions that will then be used to establish bounds of this form in sections 5 and 6.

We show in Section 5 that the Bloch--Floquet solutions can be recovered from the entries of the first row of a solution to a Riemann-Hilbert problem.
For brevity of presentation, we will consider the generic case with no degenerate spectral gaps when we present the Riemann--Hilbert problem here in the introduction.
This Riemann--Hilbert problem will be presented here in the introduction with a slightly different notation than in section 5.
The reason for the more complicated notation used in section 5 is to allow application to the general case where degenerate gaps are allowed.

To state the Riemann--Hilbert problem we introduce some notation.
We will use $\mathbb{R}^+$ to refer to the closed positive real axis.
Consider the functions $f^\pm$ defined on $\mathbb{C}$ determined by the Dirichlet eigenvalues $\mu_n$ and signatures $\sigma_n$ by
\begin{equation} f^+ (\lambda) := \prod_{\substack{n = 1 \\ \sigma_n = 1}}^\infty \frac{T^2}{n^2 \pi^2} (\mu_n - \lambda), \end{equation}
\begin{equation} f^- (\lambda) := \prod_{\substack{n = 1 \\ \sigma_n = -1}}^\infty \frac{T^2}{n^2 \pi^2} (\mu_n - \lambda). \end{equation}
We use $f^\pm$ to define a matrix function $B$ defined on $\mathbb{C} \setminus \mathbb{R}^+$ determined by the periodic/antiperiodic eigenvalues $\lambda_n$, Dirichlet eigenvalues $\mu_n$, and signatures $\sigma_n$ by
\begin{equation} \label{eq:defBintro} B(\lambda) = \frac{i \sqrt{\displaystyle{\prod_{\substack{n = 1 \\ \sigma_n = 0}}^\infty} \frac{T^2}{n^2 \pi^2} (\mu_n - \lambda)}}{\sqrt[4]{4(\lambda_0 - \lambda)\displaystyle{\prod_{n = 1}^\infty} \frac{T^4}{n^4 \pi^4} (\lambda_{2n}-\lambda)(\lambda_{2n-1}-\lambda)}} \begin{pmatrix} f^-(\lambda) & 0 \\ 0 & f^+(\lambda)\end{pmatrix}. \end{equation}
The branches of the square root and the quartic root of the infinite products appearing in $B$ are defined explicitly in remark \ref{eq:rmkrootRHP}.
The jump function $V$ is defined on $\mathbb{R}^+ \setminus \{\lambda_n\}_{n = 0}^\infty$ and determined by the same spectral data as $B$ by
\begin{equation} V (\lambda) := \begin{cases} (-1)^{k + m(\lambda) - 1} \begin{pmatrix} 0 & i \frac{f^+(\lambda)}{f^-(\lambda)} \\ i \frac{f^-(\lambda)}{f^+(\lambda)} & 0 \end{pmatrix} & \lambda \in (\lambda_{2n-2},\lambda_{2n-1})  \\ (-1)^{k + m(\lambda)} e^{2 i \sigma_3 \sqrt{\lambda} x} & \lambda \in (\lambda_{2n-1},\lambda_{2n}) \end{cases}. \end{equation}
The jump function is defined in terms of the counting function $m:\mathbb{R} \to \mathbb{Z}$ for the Dirichlet eigenvalues on the edges of non-degenerate gaps given by
\begin{equation} m(\lambda) := |\{ n \in \mathbb{N} : \mu_{n} \le \lambda, \sigma_{n} = 0 \}| . \end{equation}
The matrix $\sigma_3 = \begin{pmatrix} 1 & 0 \\ 0 & -1 \end{pmatrix}$ appearing in the definition of the jump matrix is the third Pauli matrix, and should not be confused with the signatures $\sigma_n$.
Finally, we define some discs $D_n$ with radius ${\max\{\lambda_{2n}-\lambda_{2n-1}\}_{n=1}^\infty}$ centered at $(\lambda_{2n}+\lambda_{2n-1})/2$ for $n \ge 0$ and centered at $\lambda_0 = 0$ for $n = 0$.
We use the discs $D_n$ to define the domain
\begin{equation} \mathcal{D} = \mathbb{R}^+ \cup \bigcup_{n = 0}^\infty D_n.\end{equation}
We also must define a subset $\Omega_s \subset \mathbb{C}$ as the complex plane with a sector of angle $2s$ around the positive real axis removed. This set is defined explicitly in (\ref{eq:defsector}). 
The Riemann--Hilbert problem is then

\begin{rhp}
For some $x \in \mathbb{R}$ find a $2 \times 2$ matrix valued function $\Phi(x,\lambda)$ such that:
\begin{enumerate}

\item $\Phi(x,\lambda)$ is a holomorphic function of $\lambda$ for $\lambda \in \mathbb{C} \setminus \mathbb{R}^+$.

\item $\Phi_{\pm} (x,\lambda)$ are continuous functions of $\lambda \in \mathbb{R}^+ \setminus \{\lambda_n\}_{n = 0}^{\infty}$ that have at worst quartic root singularities on $\{\lambda_n\}_{n = 0}^{\infty}$.

\item $\Phi_{\pm} (x,\lambda)$ satisfy the jump relation $\Phi_+(x,\lambda) = \Phi_-(x, \lambda)  V (x,\lambda)$.

\item $\Phi (x,\lambda)$ has an asymptotic description of the form
\begin{equation} \Phi(x,\lambda) =  \begin{pmatrix} 1 & 1 \\ -i \sqrt{\lambda} & i \sqrt{\lambda} \end{pmatrix} \left( I + O\left(\sqrt{\lambda}^{-1} \right) \right) B(\lambda) \end{equation}
as $\lambda \to \infty$ with $\lambda \in \Omega_s$ for some {$0 < s < \tfrac{\pi}{8}$}.

\item There exist positive constants $c$, and $M$ such that $|\phi_{ij} (x,\lambda)| \le M e^{c|\lambda|^2} $ for all $\lambda \in \mathcal{D}$.
\end{enumerate}
\end{rhp}
We subsequently prove Theorem \ref{thm:rhpPhi}, which asserts that the first row of the solution $\Phi$ is unique.
Solutions $\psi^\pm$ to the Schr\"{o}dinger equation $-{\psi^{\pm}}''(x,\lambda) + u(x) \psi^\pm(x,\lambda) = \lambda \psi^\pm(x,\lambda)$ and the potential $u(x)$ can be computed from any solution to the Riemann--Hilbert problem by
\begin{equation} \label{eq:intropsi} \psi^-(x,\lambda) = b_{11}(\lambda)^{-1} e^{-i\sqrt{\lambda}x} \phi_{11}(x,\lambda), \psi^+(x,\lambda) = b_{22}(\lambda)^{-1} e^{i\sqrt{\lambda}x} \phi_{12}(x,\lambda), \end{equation}
and
\begin{equation} \label{eq:intropot} u(x) = 2 i \sqrt{\lambda} \frac{\partial}{\partial x} \lim_{\lambda \to \infty} \left(1 - b_{11}(\lambda)^{-1} \phi_{11}(x,\lambda)\right). \end{equation}
In this paper we rely on Bloch--Floquet theory to guarantee existence of solutions to the Riemann--Hilbert problem.
However, we briefly discuss extension of the uniqueness theory of the first row of solutions to the Riemann--Hilbert problem in which $\lambda_{k}$ and $\mu_j$ correspond to a non-periodic potential, but the Riemann--Hilbert problem is assumed to have a solution.
If the Riemann--Hilbert has a solution, formulas (\ref{eq:intropsi},\ref{eq:intropot}) still give pairs of solutions $\psi^\pm$ to the Schr\"{o}dinger equation and corresponding potential $u(x)$.
However, the solutions have no interpretation in terms of Floquet theory.
It seems likely that the potential $u$ recovered from the non-periodic spectral data has spectrum
\begin{equation} \sigma(L) = \bigcup_{n = 0}^\infty [\lambda_{2n},\lambda_{2n+1}]\end{equation}
although $\lambda_j$ would have no interpretation in terms of periodic/antiperiodic eigenvalues.

The bound in condition 5 is not common in the theory of Riemann--Hilbert problems.  The point of this bound is to allow application of the Phragm\'{e}n--Lindelof theorem in the proof of uniqueness for the Riemann--Hilbert problem.
This is necessary because the asymptotic condition 4 is only uniform in $\lambda$ as $\lambda \to \infty$ within $\Omega_{s}$ (it has not been proven to be uniform in $x$), which excludes a sector around the positive real axis.

In Section 6 we establish a Riemann-Hilbert characterization for the KdV equation with periodic, smooth, infinite gap initial conditions - an existence theorem as well as a uniqueness theorem.  More precisely, we consider the evolution of $u$ under the KdV equation, and study the evolution of the normalized Bloch--Floquet solutions to Hill's equation.  As mentioned above, the quantities $\psi^{\pm}(x, t,\lambda)$ possess poles at the Dirichlet eigenvalues, and these poles move about in the gaps between spectral bands.  While the quantities $\psi^{\pm}(x, t, \lambda)$ do not simultaneously solve both differential equations in the Lax pair, we use the Lax pair to deform them in such a way that the new quantities, $\breve{\psi}^{\pm}(x,t,\lambda)$ do solve both differential equations.  Moreover, the poles of $\breve{\psi}^{\pm}(x,t,\lambda)$ are also independent of time.  And, finally, these are explicitly related to Riemann-Hilbert problem \ref{rhp:KdV}, in which the dependence on $x$ and $t$ is completely explicit.  The existence theorem relies on well-known existence theory for the KdV equation, and the uniqueness theorem is a slight extension of the uniqueness theorem of Section 5.

In the case of finitely many gaps, Riemann-Hilbert formulations of the inverse problem have been considered before.  For example, in \cite{TrDe13,TrDe13b} Deconinck and Trogdon used a Riemann-Hilbert problem satisfied by Baker--Akhiezer functions to numerically compute finite gap solutions of the KdV equation.
Its also introduced a Riemann--Hilbert problem formulation for the inverse spectral theory of finite gap Dirac operators and corresponding solutions to the nonlinear Schrödinger equation \cite{Its84}.

To reduce technical complications, in Section 6 we restrict our periodic potential to be $C^\infty$ to allow the use of known existence theorems.  However, since we used very general bounds on the potentials to derive the Riemann--Hilbert problem for the inverse spectral theory, we believe the Riemann--Hilbert problem for solutions to the KdV equation should apply to weak solutions with periodic initial conditions in $L^\infty(\mathbb{R})$.

In section 7 we use the results of sections 5 and 6 to sketch proofs of conditions for spatial and temporal periodicity of solutions to the KdV equation.
In section 8 we connect the uniqueness results for the first row of solutions to Riemann--Hilbert problems \ref{rhp:Phi} and \ref{rhp:KdV} to a uniqueness results for the corresponding spatially periodic infinite gap Baker--Akhiezer functions.
Finally, in section 9 we discuss the relation of the infinite gap theory discussed in this paper to the well known finite gap theory.

\section{Floquet Theory}

In this section we review some aspects of the forward spectral theory of Hill's operators that are relevant to our discussion of the inverse spectral theory in subsequent sections.
This review combines a combination of results from \cite{MaWi79} and \cite{LeSa75}, and all but one of the relevant proofs can be found in PN's dissertation \cite{MyDiss}.
For details on the spectral theory of Hill?s operators also see \cite[chapter 21]{Ti58}.
The forward spectral problem for Hill's operators can be solved by considering the behavior of solutions to Hill's equation
\begin{equation} -y''(x) + u(x) y(x) = \lambda y(x) \end{equation}
{ with $\lambda \in \mathbb{C}$.}
We will call the parameter $\lambda$ appearing in Hill's equation the spectral parameter.
To analyze Hill's equation, we construct the matrix fundamental solution
\begin{equation} Y(x,\lambda) = \begin{pmatrix} y_1(x,\lambda) & y_2(x,\lambda) \\ y_1'(x,\lambda) & y_2'(x,\lambda) \end{pmatrix} ,\end{equation}
where $y_i$ solve Hill's equation and $Y$ is normalized so that $Y(0,\lambda) = I$ (the $'$ indicates a partial derivative with respect to $x$).
The large $|x|$ behavior of solutions to Hill's equation is intimately connected to the eigenvalue problem
\begin{equation} \label{eq:eigvalpro} Y(T,\lambda) v = \rho v, \end{equation}
because if $\rho$ and $v$ are an eigenvalue/eigenvector pair then
\begin{equation} y(x) = v_1 y_1(x,\lambda) + v_2 y_2 (x,\lambda) \end{equation}
is a solution to Hill's equation that factors as $y(x) = \rho^{x T^{-1}} p(x)$ for some function $p$ with period $T$.
This fact is commonly known as the Bloch or Floquet theorem, and such a solution $y$ we will call a Bloch--Floquet solution.

The characteristic polynomial for the eigenvalue problem (\ref{eq:eigvalpro}) is
\begin{equation} \label{eq:charpoly} p_c(\rho) = \rho^2 - \Delta(\lambda) \rho + 1, \end{equation}
which has been written in terms of the Floquet discriminant $\Delta(\lambda) = \text{Tr}(Y(T,\lambda))$.
Since the characteristic polynomial $p_c$ is of degree two with constant coefficient one, it has two roots which are multiplicative inverse of each other.
We define $\rho(\lambda)$ to be the root of $p_c$ analytic for $\lambda \in \mathbb{C} \setminus \sigma(L)$ such that $|\rho(\lambda)|<1$ for $\lambda \in \mathbb{C} \setminus \sigma(L)$.  This function has boundary values (from above or below) on any subinterval of the spectrum of $L$, and the boundary values (denoted $\rho_{\pm}$) satisfy $| \rho_{\pm}(\lambda)| = 1$ for $\lambda \in \sigma(L)$.
It is well known \cite{MaWi79} that the condition $\lambda \in \sigma(L)$ is satisfied if and only if $\Delta(\lambda) \in [-2,2]$.
The second root of $p_c$ for $\lambda \in \mathbb{C} \setminus \sigma(L)$ is then necessarily $\rho(\lambda)^{-1}$ which satisfies $|\rho(\lambda)|^{-1} > 1$.
The roots $\rho(\lambda)$ and $\rho(\lambda)^{-1}$ can be computed as the fixed points of the maps
\begin{equation} \rho \to \frac{1 + \rho^2}{\Delta(\lambda)}, \text{ and } \rho \to \Delta(\lambda) - \frac{1}{\rho} \end{equation}
respectively.

The quadratic formula tells us that
\begin{equation} \label{eq:rhodelta} \rho(\lambda) = \frac{\Delta(\lambda) - \sqrt{\Delta(\lambda)^2 - 4}}{2} \end{equation}
for some choice of $\sqrt{\Delta(\lambda)^2 - 4}$, and the reader may verify that $\sqrt{\Delta(\lambda)^2-4}$ may be taken to be holomorphic on $\mathbb{C} \setminus \sigma(L)$ and positive for $\lambda < \lambda_0$.   With this definition of $\sqrt{\Delta(\lambda)^2-4}$, we also have
\begin{equation} \label{eq:rhoinvdelta} \rho(\lambda)^{-1} = \frac{\Delta(\lambda) + \sqrt{\Delta(\lambda)^2 - 4}}{2} .\end{equation}

We will consider the family of Bloch--Floquet solutions $\psi^\pm(x, \lambda)$ to the Hill's equation  uniquely determined by the condition $\psi^\pm(x+T,\lambda) = \rho(\lambda)^{\pm 1} \psi^\pm(x,\lambda)$ and the normalization $\psi^\pm (0, \lambda) = 1$.
The functions $\psi^\pm$ can be expressed in terms of $y_i$ and $\rho$ as
\begin{equation} \label{eq:bloflo} \psi^\pm (x,\lambda) = y_1(x,\lambda) + \frac{\rho(\lambda)^{\pm 1} - y_1(T,\lambda)}{y_2(T,\lambda)} y_2(x,\lambda), \end{equation}
because the vector valued functions
\begin{equation}v^\pm(\lambda) = \left( 1,  \frac{\rho(\lambda)^{\pm 1} - y_1(T, \lambda)}{y_2(T, \lambda)} \right)^T\end{equation}
solve $(Y(T,\lambda) - \rho(\lambda)^{\pm 1})v^\pm(\lambda) = 0$ for all $\lambda \in \mathbb{C} \setminus {(\sigma(L) \cup \sigma(L_0))}$.

Analyticity properties of $\psi^{\pm}$ in $\lambda$ follow from the fact that the fundamental matrix solution $Y$ is entire, and the analyticity properties of the function $\rho$.  This will be explored further and used in Section 5.

The asymptotic behavior of the above Bloch--Floquet solutions is related to the potential $u$ by the following lemma, which is proved in section 3:
\begin{lma} \label{lma:BFass}
For all $s > 0$ and fixed $x$, $\psi^\pm (x,\lambda)$ and ${\psi^\pm}'(x,\lambda)$ have asymptotic descriptions of the form
\begin{equation} \label{eq:asspsi} \psi^+ (x,\lambda) = e^{i \sqrt{\lambda} x} \left(1 + \frac{1}{2 i \sqrt{\lambda}} \int_0^x u(t) dt + O(\lambda^{-1}) \right), \end{equation}
\begin{equation*} \psi^- (x,\lambda) = e^{-i \sqrt{\lambda} x} \left(1 - \frac{1}{2 i \sqrt{\lambda}} \int_0^x u(t) dt + O(\lambda^{-1}) \right), \end{equation*}
and
\begin{equation} {\psi^+}' (x,\lambda) =e^{i \sqrt{\lambda} x} \left(  i \sqrt{\lambda} + O(1) \right), \end{equation}
\begin{equation*} {\psi^-}' (x,\lambda) =e^{-i \sqrt{\lambda} x} \left(  -i \sqrt{\lambda} + O(1) \right), \end{equation*}
as $\lambda \to \infty$, valid for $\lambda \in \Omega_s$ where
\begin{equation}\label{eq:defsector} { \Omega_s := \{\lambda \in \mathbb{C} : |s< \arg(\lambda) < 2 \pi - s \}}. \end{equation}
\end{lma}

\begin{rmk} \label{rmk:defbranch}
In (\ref{eq:defsector}) we use the branch of the argument taking values in $[0,2\pi)$. We also use this branch of the argument to define $\sqrt{\lambda}$ in the previous lemma, so  $\sqrt{\lambda}$ takes values in the upper half plane for $\lambda \in \mathbb{C} \setminus \mathbb{R}^+$ and positive values for $\lambda \in \mathbb{R}^+ \setminus \{0\}$.
\end{rmk}

Hill's equation has at least one bounded solution if and only if $\Delta(\lambda) \in [-2,2]$.
In fact, the existence of bounded solutions to Hill's equation for a given $\lambda$ is equivalent to saying $\lambda \in \sigma(L)$ by an argument based on the Weyl criterion (see \cite{HiSi96} for example).
The spectrum $\sigma(L_0)$ is composed of the zeros of $y_2(T,\lambda)$.
The results we will need on forward spectral theory are summarized in the following theorem:

\begin{thm} \label{thm:forspec}
Consider a Hill's operator $-\frac{\partial^2}{\partial x^2} + u(x)$ where $u \in L^\infty$ is periodic with period $T$.
The following are necessary conditions on the spectra $\sigma(L)$ and $\sigma(L_0)$:
\begin{enumerate}
\item There are infinitely many periodic eigenvalues $\{\lambda_0\} \cup \{ \lambda_{4n-1},\lambda_{4n}\}_{n = 1}^\infty$ and infinitely many antiperiodic eigenvalues $\{\lambda_{4n-3},\lambda_{4n-2}\}_{n = 1}^\infty$ of $L$ that are all real and can be ordered such that
\begin{equation} \label{eq:perord}\lambda_0 < \cdots < \lambda_{2n-1} \le \lambda_{2n} < \lambda_{2n+1} \le \lambda_{2n+2} < \cdots . \end{equation}

\item The spectrum $\sigma (L)$ for the whole line problem is the union
\begin{equation} \label{eq:specc} \sigma(L) = \bigcup_{n = 0}^\infty [\lambda_{2n},\lambda_{2n+1}]. \end{equation}

\item The spectrum $\sigma (L_0)$ for the Dirichlet problem on $[0,T]$ is an increasing sequence $\sigma(L_0) = \{\mu_n\}_{n = 1}^\infty$ where $\mu_n \in [\lambda_{2n-1},\lambda_{2n}]$.

\item $\lambda_{2n-1}$, $\lambda_{2n}$ and $\mu_n$ have the large $n$ asymptotic behavior
\begin{equation} \label{eq:assspec123} \lambda_{2n-1}, \lambda_{2n}, \mu_n = \frac{n^2 \pi^2}{T^2} + \frac{Q}{T} + O(n^{-1})\end{equation}
where
\begin{equation} \label{eq:defQ} Q = \int_0^T u(x) dx. \end{equation}
\end{enumerate}
\end{thm}

Parts 1 and 2 are classical results that can be found in \cite[chapter 2]{MaWi79}, part 3 is basic Sturm--Liouville theory which is reviewed in \cite[chapter 1]{LeSa75} for example, and part 4 is a corollary to a theorem by Borg stated in \cite[page 39]{MaWi79}.

The Dirichlet eigenfunctions $y_2(x,\mu_n)$ must also be Bloch--Floquet solutions.
In fact we have the following claim:
\begin{clm} \label{clm:dirflo} The Dirichlet eigenfunction $y_2(x,\mu_n)$ is a Bloch--Floquet solution with shift eigenvalue ${{y_2}'(T,\mu_n) = y_1(T,\mu_n)^{-1} = \rho(\mu_n)}$ or $\rho(\mu_n)^{-1}$. \end{clm}

\begin{proof}[Proof of Claim \ref{clm:dirflo}]
The Dirichlet eigenfunction $y_2(x,\mu_n)$ solves Hill's equation with spectral parameter $\mu_n$, so $y_2(x+T,\mu_n)$ also solves Hill's equation with spectral parameter $\mu_n$.
The space of solutions of Hill's equation with spectral parameter $\mu_n$ is spanned by $y_1(x,\mu_n)$ and $y_2(x,\mu_n)$, so $y_2(x+T, \mu_n) = \alpha y_1(x, \mu_n) + \beta y_2(x, \mu_n)$ for some constants $\alpha, \beta$ .
Evaluation of $y_2(x+T, \mu_n)$ at $x = 0$ gives $y_2(T, \mu_n) = \alpha = 0$.
Differentiating both sides of $y_2(x+T, \mu_n) = \beta y_2 (x,\mu_n)$ by $x$ and evaluating the result at $x = 0$ gives $\beta = {y_2}'(T,\mu_n)$.
Therefore, $y_2(x,\mu_n)$ is a Bloch--Floquet solution with shift eigenvalue ${y_2}'(T,\mu_n)$, so ${y_2}' (T,\mu_n) = \rho(\mu_n)$ or ${y_2}' (T,\mu_n) = \rho(\mu_n)^{-1}$ by the Floquet theorem.
Then $y_1(T,\mu_n) = \rho(\mu_n)^{-1}$ or $y_1(T,\mu_n) = \rho(\mu_n)$ respectively because $\det(Y(T,\mu_n)) = y_1(T,\mu_n) {y_2}'(T,\mu_n) = 1$. \end{proof}

Recall that the order $r$ of an entire function $f$ is defined \cite[page 248]{Ti39} as
\begin{equation} \label{eq:deforder} r := \inf \{ r_0 > 0 : f(\lambda) = O(e^{|\lambda|^{r_0}}) \text{ as } \lambda \to \infty \}. \end{equation}
In \cite{MaWi79} they prove that $y_i(x,\lambda)$ and $y_i'(x,\lambda)$ are entire functions with orders at most $\tfrac{1}{2}$ for all $x$, and in fact $y_2(T,\lambda)$ and $\Delta(\lambda)$ are entire functions with orders $\tfrac{1}{2}$.
By standard results from complex analysis, one deduces from this that $L$ has infinitely many periodic and antiperiodic eigenvalues.

In \cite{MaWi79} they also prove that: all zeros of $y_2(T,\lambda)$ are simple; $\Delta(\lambda) - 2$ has simple zeros at the periodic eigenvalues unless $\lambda_{4n-1} = \lambda_{4n}$, in which case $\Delta(\lambda) - 2$ has a multiplicity 2 zero at $\lambda_{4n}$; $\Delta(\lambda) + 2$ has simple zeros at the antiperiodic eigenvalues unless $\lambda_{4n-3} = \lambda_{4n-2}$, in which case $\Delta(\lambda) + 2$ has a multiplicity 2 zero at $\lambda_{4n-2}$.
These facts, along with known asymptotic behaviors as $\lambda \to \infty$, allows us to conclude the following expansions by the Hadamard factorization theorem \cite{MckTr78}:
\begin{equation} \label{eq:hady2} y_2(T,\lambda) = \prod_{n \ge 1} \frac{T^2}{n^2 \pi^2} (\mu_n - \lambda), \end{equation}
\begin{equation} \label{eq:hadfloper} \Delta(\lambda) - 2 = 2 (\lambda_0 - \lambda) \prod_{n \text{ even}} \frac{T^4}{n^4 \pi^4} (\lambda_{2n} - \lambda) (\lambda_{2n-1} - \lambda), \end{equation}
\begin{equation} \label{eq:hadfloaper} \Delta(\lambda) + 2 = 2 \prod_{n \text{ odd}} \frac{T^4}{n^4 \pi^4} (\lambda_{2n} - \lambda) (\lambda_{2n-1} - \lambda).  \end{equation}
\begin{rmk}
We observe that the product expansion
\begin{equation} \label{eq:expansionsqrtflo}
\sqrt{ \Delta(\lambda)^{2} - 4} = - 2 i \sqrt{\lambda - \lambda_{0}} \prod_{n=1}^{\infty} \left( \frac{ - T^{2}}{n^{2} \pi^{2}} \right) \sqrt{\lambda - \lambda_{2n-1}}\sqrt{\lambda - \lambda_{2n}}
\end{equation}
yields an explicit representation of the function $\sqrt{\Delta(\lambda)^{2} - 4}$ that is consistent with the definition from formula (\ref{eq:rhodelta}).
The square roots of the elementary factors are defined using the branch of the square root discussed in remark \ref{rmk:defbranch}.
\end{rmk}

We conclude this section by explicitly defining the spectral data that uniquely define the potential $u$.
\begin{defn}
\label{def:SpecData}
Let $\{E_n\}_{n = 0}^{2 \mathcal{G}}$ be the subsequence of $\{\lambda_n\}_{n = 0}^\infty$ such that the spectrum $\sigma(L)$ is the disjoint union
\begin{equation} \sigma(L) = \bigcup_{n = 0}^{\mathcal{G}-1} [E_{2n}, E_{2n+1}] \cup [E_{2 \mathcal{G}}, \infty) \end{equation}
where the number of gaps $\mathcal{G}$ is either finite or countably infinite, and in the latter case the interval $[E_{2 \mathcal{G}}, \infty)$ is not included.
Let $\{\mu_{n_k} \}_{k = 1}^{\mathcal{G}}$ be the subsequence of $\mu_{n}$ such that $\mu_{n_k} \in [E_{2k-1}, E_{2k}]$, and let $\sigma_n = -\textrm{sgn}(\log(|{y_2}'(T,\mu_n)|))$.
We define the spectral data associated to the potential $u$ as
\begin{equation} \Sigma(u) :=  \label{eq:specdata} \{E_0 \} \cup \{ E_{2k-1}, E_{2k}, \mu_{n_k}, \sigma_{n_k}\}_{k = 1}^{\mathcal{G}}.\end{equation}
\end{defn}
The reason for considering the new parameters $E_k$ rather than $\lambda_n$ in the spectral data is that when a spectral gap closes the doubly degenerate periodic/antiperiodic eigenvalue corresponding to the degenerate gap - which is also a Dirichlet eigenvalue - is set purely by the band ends $E_k$.
From the point of view of our Riemann--Hilbert problem, this is essentially due to the cancellation that occurs in (\ref{eq:defBintro}) in the degenerate cases.
In other words, we need to consider $E_{2k-1},E_{2k}$ and $\mu_{n_k}$, $\sigma_{n_k}$ instead of the full sequence $\lambda_{2n-1},\lambda_{2n}$ and $\mu_n$, $\sigma_n$ to end up with an inverse problem that is not overdetermined in the degenerate cases.

In the above definition we use the convention that the sign function $\text{sgn}$ takes values 1 on the positive numbers, $-1$ on the negative numbers and 0 otherwise.
The spectral interpretation of the signature $\sigma_{n_k}$ when $\sigma_{n_k} = \pm 1$ is that $y_2(x,\mu_{n_k})$ is then a Bloch--Floquet solution with shift eigenvalue $\rho(\mu_{n_k})^{\sigma_{n_k}}$.
The signature $\sigma_n$ vanishes if and only if $\mu_n = \lambda_{2n-1}$ or $\mu_n = \lambda_{2n}$ since in those cases $\rho(\mu_n) = \rho(\mu_n)^{-1} = \pm 1$.
Although knowledge that $\sigma_{n_k} = 0$ is redundant information, we include the signature in these cases for notational convenience.
If $\sigma_{n_k} = 1$ or $\sigma_{n_k} = -1$, then $\psi^+(x,\mu_{n_k})$ or $\psi^-(x,\mu_{n_k})$ is undefined respectively because it is then impossible to make the normalization $\psi^+ (0,\mu_{n_k}) = 1$ or $\psi^- (0,\mu_{n_k}) = 1$ respectively.
Similarly, if $\sigma_{n_k} = 0$ then both $\psi^\pm(x,\mu_n)$ are undefined because on the band ends there is only one linearly independent Bloch--Floquet solution \cite[page 4]{MaWi79}.

A primary objective of the next few sections is to prove that the Bloch--Floquet solutions of Hill's equation yield a solution to a matrix Riemann-Hilbert problem.  Toward that end, we will first establish some asymptotic estimates concerning solutions of Hill's equation, valid for $\lambda \to \infty$, away from the positive real axis.

For simplicity of presentation we will assume that $E_0 = 0$ for the remainder of the paper.  If $E_0 \ne 0$ one can always apply the arguments as presented in this paper to the shifted potential $\tilde u(x) = u(x) - E_0$.

\section{Asymptotic Expansions of Bloch--Floquet Solutions}

In this section we prove lemma \ref{lma:BFass}, and some other useful asymptotic results.
Let us introduce the solutions $y^\pm$ to Hill's equation with spectral parameter $\lambda$ normalized by the initial condition $y^\pm (0,\lambda) = 1$ and $(y^\pm) '(0,\lambda) = \pm i \sqrt{\lambda}$.
Then
\begin{equation}y^\pm (x,\lambda) = y_1(x,\lambda) \pm i \sqrt{\lambda} y_2(x,\lambda)\end{equation}
so $y^\pm (x,\lambda)$ exist, are unique, are holomorphic for $\lambda \in \mathbb{C} \setminus \mathbb{R}^+$.
Let $m(x, \lambda) = e^{i \sqrt{\lambda} x} y^-(x,\lambda)$ and $n(x, \lambda) = e^{i \sqrt{\lambda} x} y^+(x,\lambda)$.
Recall that we are using the non-principal branch of the square root discussed in remark \ref{rmk:defbranch} that maps $\mathbb{C} \setminus \mathbb{R}^+$ onto the upper half plane so that $e^{i \sqrt{\lambda} x}$ exponentially decays and $e^{-i\sqrt{\lambda} x}$ exponentially grows as $\lambda \to \infty$ for $x > 0$. { This is the branch of the square root determined by the branch of the argument mapping $\mathbb{C}$ onto $[0,2\pi)$}.

\begin{prop} \label{prop:assmn}
Suppose that $u \in L^\infty(\mathbb{R})$, then $m(T,\lambda)$ and $n(T,\lambda)$ are holomorphic functions of $\lambda$ for $\mathbb{C} \setminus \mathbb{R}^+$ such that for each $s > 0$ as $\lambda \to \infty$, $\lambda \in \Omega_{s}$, we have
\begin{equation} m(T,\lambda) = 1 - \frac{1}{2 i \sqrt{\lambda}} Q + O(\lambda^{-1}), 
\end{equation}
$n(T,\lambda) = O(\lambda^{-1})$, and $m'(T,\lambda),n'(T,\lambda) = O(\sqrt{\lambda}^{-1})$.
If we additionally assume that $u$ has a bounded second derivative, then
\begin{equation} n(T,\lambda) = -\frac{1}{4 \lambda} u(0) + \frac{1}{8 i \sqrt{\lambda}^3} \left( u'(0) + u(0) Q \right) + O(\lambda^{-2}),  \end{equation}
\begin{equation} m'(T,\lambda) = - \frac{1}{2 i \sqrt{\lambda}} u(0) + \frac{1}{4 \lambda} \left(u'(0) - u(0) Q\right) + O(\sqrt{\lambda}^{-3}), \end{equation}
\begin{equation} n'(T,\lambda) = O(\sqrt{\lambda}^{-3})\end{equation}
as $\lambda \to \infty$, $\lambda \in \Omega_s$.
\end{prop}

\begin{proof}
The functions $m$ and $n$ solve the integral equations
\begin{equation}m(x, \lambda) = 1 + G(m)(x, \lambda), n(x,\lambda) = e^{2i \sqrt{\lambda} x} + G(n)(x,\lambda),\end{equation}
where
\begin{equation}G(f) (x)  = \int_0^x \frac{e^{2i\sqrt{\lambda}(x - t)}-1}{2i \sqrt{\lambda}} u(t) f(t) dt.\end{equation}
The bound $|e^{2i\sqrt{\lambda} (T-t)}-1| \le 2$ for $t \in [0,T]$ implies
\begin{align*} \numberthis \label{eq:boundGn}  |G^n(g)(x)| & \le \frac{2^n}{|\sqrt{\lambda}|^n} \int_0^T |u(t_1)| \cdots \int_0^{t_{n-1}} |u(t_n)| |g(t_n)| dt_n \cdots dt_1 \\ 
& \le \frac{\|u\|_{L^\infty (\mathbb{R})}^n 2^n T^n}{|\sqrt{\lambda}|^n n!} \|g\|_{L^\infty ([0,T])}. \end{align*}
The bound (\ref{eq:boundGn}) implies that the Neumann series
\begin{equation}m(x,\lambda) = \sum_{j = 0}^\infty G^j(1)(x,\lambda), \end{equation}
\begin{equation} n(x,\lambda) = \sum_{j = 0}^\infty G^j (e^{2 i \sqrt{\lambda} t})(x, \lambda)\end{equation}
converges uniformly to holomorphic functions that uniquely solve the integral equation for $m,n$.
The bounds
\begin{equation}\sum_{n = 2}^\infty |G^n(1)(x,\lambda)| \le \frac{C'}{|\lambda|} \exp\left( 2 \|u\|_{L^\infty(\mathbb{R})} T |\sqrt{\lambda}|^{-1}\right)\end{equation}
following from (\ref{eq:boundGn}) and
\begin{align*} \numberthis \label{eq:gasscorection} \left| G(1)(x,z) + \frac{1}{2 i \sqrt{\lambda}} \int_0^x u(t) dt \right|& = \left| \int_0^x \frac{e^{2i\sqrt{\lambda}(x - t)}}{2i \sqrt{\lambda}} u(t)dt \right| \\
& \le \frac{e^{-2 |\text{Im}(\sqrt{\lambda} )| x}}{2 |\sqrt{\lambda}|} \|u\|_{L^\infty (\mathbb{R})} \int_0^x e^{2 |\text{Im}(\sqrt{\lambda})| t} dt \\
 & \le \frac{\|u\|_{L^\infty (\mathbb{R})}}{\sin(s) |\lambda|}\end{align*}
imply that
\begin{equation}m (x,\lambda) = 1 - \frac{1}{2 i \sqrt{\lambda}} \int_0^x u(t) dt + O(\lambda^{-1}).\end{equation}
The $\sin(s)$ appears in (\ref{eq:gasscorection}) because $|\text{Im}(\sqrt{\lambda})| \ge \sin\left(\frac{s}{2} \right)|\sqrt{\lambda}| \ge \frac{\sin(s)}{2} |\sqrt{\lambda}|$ for $\lambda \in \Omega_s$. This is the first place we have used the condition $\lambda \in \Omega_s$ in this proof. \\

Since $e^{2i\sqrt{\lambda}x} \le e^{-\sin(s) |\text{Im}(\sqrt{\lambda})|x} < 1$, the function $n(x,\lambda)$ is bounded for large $\lambda$.
To show that ${n(x,\lambda) = O(\lambda^{-1})}$ for $\lambda \in \Omega_s$ we note that from the exponential decay in $e^{2 i \sqrt{\lambda} x}$ and (\ref{eq:boundGn}) implies that the only term in the Neumann series that could prevent $n(x,\lambda) = O(\lambda^{-1})$ is $G(e^{2 i \sqrt{\lambda} t})(x,\lambda)$.
\begin{equation} \label{eq:assbound123} G(e^{2 i \sqrt{\lambda} t}) (x,\lambda)  = \frac{e^{2i\sqrt{\lambda}x}}{2i \sqrt{\lambda}} \int_0^x  u(t) dt - \frac{1}{2i \sqrt{\lambda}} \int_0^x e^{2i\sqrt{\lambda}t}  u(t) dt.\end{equation}
The first term in (\ref{eq:assbound123}) decays exponentially in $\Omega_s$ while the argument behind (\ref{eq:gasscorection}) implies that the second term in (\ref{eq:assbound123}) is bounded as $O(\lambda^{-1})$. \\

We also need control on $m'(T,\lambda)$ and $n'(T,\lambda)$.
By differentiating both sides of the integral equations for $m$ and $n$ we can bound $m'(T, \lambda)$ and $n'(T, \lambda)$ in terms of $m(T, \lambda)$ and $n(T, \lambda)$ as
\begin{equation}\label{eq:intdm} m'(T,\lambda) = \int_0^T e^{2i \sqrt{\lambda} (T-t)} u(t) m(t) dt, \end{equation}
and
\begin{equation}\label{eq:intdn} n'(T,\lambda) = 2i\sqrt{\lambda} e^{2i\sqrt{\lambda}T} + \int_0^T e^{2i\sqrt{\lambda}(T-t)} u(t) n(t) dt , \end{equation}
so
\begin{equation}|m'(T,\lambda)| \le \frac{2\|u\|_{L^\infty (\mathbb{R})} \| m(\lambda) \|_{L^\infty ([0,T])}}{\sin(s) |\sqrt{\lambda}|},\end{equation}
and
\begin{equation}|n'(T,\lambda)| \le 2 |\sqrt{\lambda} | e^{-\sin(s) |\text{Im}(\sqrt{\lambda})|T} + \frac{2 \|u\|_{L^\infty (\mathbb{R})} \| n(\lambda) \|_{L^\infty ([0,T])}}{\sin(s) |\sqrt{\lambda}|}.\end{equation}
Therefore, $m'(T,\lambda)$ and $n'(T,\lambda)$ have the large $\lambda$ asymptotic behaviors $m'(T, \lambda) = O(\sqrt{\lambda}^{-1})$ and $n'(T,\lambda) = O(\sqrt{\lambda}^{-1})$.

The proof of the remaining estimates comes from applying integration by parts.
For the expansion for $n$ we consider terms in the Neumann series.
For the expansion of $m'$ and the bound on $n'$ we plug the expansions for $n$ and $m$ into (\ref{eq:intdm}) and (\ref{eq:intdn}).
\end{proof}

\begin{prop}\label{prop:assdelta}
For all $s > 0$ the asymptotic expansion
\begin{equation}\label{eq:assdelta} \Delta(\lambda) = e^{-i \sqrt{\lambda} T} \left(1 - \frac{1}{2i\sqrt{\lambda}} Q + O(\lambda^{-1}) \right)\end{equation}
is valid as $\lambda \to \infty$ for $\lambda \in \Omega_s$.
\end{prop}

\begin{proof}[Proof of Proposition \ref{prop:assdelta}]
The Floquet discriminant $\Delta(\lambda) = y_1(T,\lambda) + y_2'(T,\lambda)$ is expressed in terms of $m,n$ as
\begin{equation}
\Delta(\lambda) = e^{-i \sqrt{\lambda} T} \left( m(T,\lambda) + \frac{n'(T,\lambda) - m'(T,\lambda)}{2 i \sqrt{\lambda}} \right).
\end{equation}
Therefore, $\Delta(\lambda)$ has the large $\lambda$ asymptotic description
\begin{equation}\Delta(\lambda) = e^{-i \sqrt{\lambda} T} \left(1 - \frac{1}{2 i \sqrt{\lambda}} Q + O(\lambda^{-1}) \right)\end{equation}
for $\lambda \in \Omega_s$, where the $O(\sqrt{\lambda}^{-1})$ correction term comes from (\ref{eq:gasscorection}).
In particular, notice that this asymptotic description implies that $\Delta(\lambda)$ has order at least $\frac{1}{2}$ \cite[page  248]{Ti39}.
However, $\Delta(\lambda)$ also has order at most $\frac{1}{2}$, so $\Delta(\lambda)$ has order exactly $\frac{1}{2}$.
\end{proof}

\begin{prop}
If $u(x) \in L^\infty (\mathbb{R})$, then for all $s > 0$
\begin{equation} \label{eq:propx01} \psi_x^\pm(0,\lambda) = \frac{\rho(\lambda)^{\pm 1} - y_1(T,\lambda)}{y_2(T,\lambda)} = \pm i \sqrt{\lambda} + O(\sqrt{\lambda}^{-1}) \end{equation}
as $\lambda \to \infty$ for $\lambda \in \Omega_s$.
If we additionally assume the $u$ has a bounded second derivative then for all $s > 0$
\begin{equation} \label{eq:propx02} \psi_x^\pm(0,\lambda) = \frac{\rho(\lambda)^{\pm 1} - y_1(T,\lambda)}{y_2(T,\lambda)} = \pm i \sqrt{\lambda} \pm \frac{u(0)}{2 i \sqrt{\lambda}}  + \frac{u_x(0)}{4 \lambda} +  O(\sqrt{\lambda}^{-3})\end{equation}
as $\lambda \to \infty$ for $\lambda \in \Omega_s$.
\end{prop}

\begin{proof}
From the definitions of $\rho(\lambda)$ and $\rho(\lambda)^{-1}$ in terms of contraction mappings, it is clear that
\begin{equation} \rho(\lambda)^{-1} = \Delta(\lambda) + O(\Delta(\lambda)^{-1}),\end{equation}
\begin{equation} \rho(\lambda) = O(\Delta(\lambda)^{-1}). \end{equation}
As $\lambda \to \infty$ for $\lambda \in \Omega_s$, the $O(\Delta(\lambda)^{-1})$ decay exponentially, and so does $1/y_2(T,\lambda)$.
The asymptotic behavior of $\psi_x^+$ is therefore determined by
\begin{equation} \psi_x^+(0,\lambda) = -\frac{y_1(T,\lambda)}{y_2(T,\lambda)} + O(\sqrt{\lambda}^{-3}), \end{equation}
and the asymptotic behavior of $\psi_x^-$ is determined by
\begin{equation} \psi_x^-(0,\lambda) = \frac{\Delta(\lambda) - y_1(T,\lambda)}{y_2(T,\lambda)} + O(\sqrt{\lambda}^{-3}) = \frac{y_{2x}(T,\lambda)}{y_2(T,\lambda)} + O(\sqrt{\lambda}^{-3}). \end{equation}
In terms of $m$ and $n$ we have
\begin{equation} \label{eq:plusmn} -\frac{y_1(T,\lambda)}{y_2(T,\lambda)} = i \sqrt{\lambda} \frac{m(T,\lambda) + n(T,\lambda)}{m(T,\lambda) - n(T,\lambda)} = i \sqrt{\lambda} + 2 i \sqrt{\lambda} \sum_{j = 1}^\infty \left(\frac{n(T,\lambda)}{m(T,\lambda)} \right)^j.\end{equation}
Applying the expansions for $n(T,\lambda)$ and $m(T,\lambda)$ from proposition \ref{prop:assmn} into (\ref{eq:plusmn}) give (\ref{eq:propx01},\ref{eq:propx02}) for the case of $\psi_x^+(0,\lambda)$.
In terms of $m$, $m'$, $n$ and $n'$ we have
\begin{equation} \label{eq:minusmn} \frac{y_{2x}(T,\lambda)}{y_2(T,\lambda)} = -i \sqrt{\lambda} + \frac{m'(T,\lambda) - n'(T,\lambda)}{m(T,\lambda) - n(T,\lambda)}.\end{equation}
Applying the expansions for $n(T,\lambda)$, $n'(T,\lambda)$, $m(T,\lambda)$ and $m'(T,\lambda)$ from proposition \ref{prop:assmn} into (\ref{eq:minusmn}) give (\ref{eq:propx01},\ref{eq:propx02}) for the case of $\psi_x^-(0,\lambda)$.
\end{proof}

\begin{proof}[Proof of Lemma \ref{lma:BFass}]

The main difficulty in proving the lemma is dealing with the exponentially decaying solutions in $\lambda$, because in these cases all the exponentially growing parts of the $y_1(x,\lambda)$ and $y_2(x,\lambda)$ terms must cancel.
We therefore must instead factor the solutions into exponentially decaying and periodic parts.
This proof also works for the exponentially growing solutions, so we apply it for all cases of the expansions.
To handle expansions for both $\psi^+$ and $\psi^-$ let us consider the $z$ plane defined by $z^2 = \lambda$ and use the normalized Bloch--Floquet solution $\tilde \psi$ and $\tilde \rho$ on the $z$ plane defined in (\ref{eq:deftpsi}) and (\ref{eq:deftrho}) but for $z \in \mathbb{C}$.
We will justify an expansion for $\tilde \psi(x,z)$ and then recover expansions for $\psi^+(x,\lambda)$ and $\psi^-(x,\lambda)$ using $\psi^+(x,\lambda) = \tilde \psi(x,\sqrt{\lambda})$ and $\psi^-(x,\lambda) = \tilde \psi(x,-\sqrt{\lambda})$, where we should recall that our choice of square root has a branch cut on the positive real numbers and takes values in the upper half plane.
It follows from $|\rho(\lambda)|< 1$ and $|\rho(\lambda)^{-1}| > 1$ for $\lambda \in \mathbb{C} \setminus \sigma(L)$ that $|\tilde \rho(z)|<1$ for $z \in \mathbb{C}^+$ and $|\tilde \rho(z)|>1$ for $z \in \mathbb{C}^-$.
Let us fix some $0 < s < \tfrac{\pi}{4}$ and define the domain $\tilde \Omega_s = \{z \in \mathbb{C} : z^2 \in \Omega_{s} \}$.
The domain $\tilde \Omega_s \subset \mathbb{C}$ with sectors around the positive and negative axes removed is used to get tight bounds on exponential integrals, and to make sure exponentials of the form  $e^{\pm i c z}$ for $c > 0$ decay uniformly in $\mathbb{C}^\pm \cap \tilde \Omega_s$ as $z \to \infty$.

We must first derive some simple large $z$ asymptotic descriptions following from (\ref{eq:assdelta}).
By combining the fact that $\tilde \rho(z)$ is a fixed point of the map $\rho \to \Delta(z^2) + 1/\rho$ with $|\tilde \rho(z)| > 1$ for $z \in \mathbb{C}^-$ we find that
\begin{equation} \tilde \rho(z) = \Delta (z^2)(1 + O(\Delta (z^2)^{-1})) = e^{i z T}\left(1 + \frac{1}{2 i z} Q + O(z^{-2}) \right)\end{equation}
for $z \in \tilde \Omega_s \cap \mathbb{C}^-$.
It can be verified from the definition of $\tilde \rho(z)$ in terms of $\rho(\lambda)$ and $\rho(\lambda)^{-1}$ that $\tilde \rho(z) = \tilde \rho(-z)^{-1}$, which implies that the large $z$ asymptotic description 
\begin{equation} \label{eq:rhoass} \tilde \rho(z) = e^{i z T}\left(1 + \frac{1}{2 i z} Q + O(z^{-2}) \right)\end{equation}
as $z \to \infty$ valid for $z \in \tilde \Omega_s$.
Applying a logarithm on both sides of (\ref{eq:rhoass}) gives the asymptotic description
\begin{equation} \label{eq:logass} \log(\tilde \rho(z)) = i z T + \log \left(1 + \frac{1}{2 i z} Q + O(z^{-2}) \right) = i z T + \frac{1}{2 i z} Q + O(z^{-2})\end{equation}
as $z \to \infty$ valid for $z \in \tilde \Omega_s$.

A bit of care must be taken in defining $\log(\tilde \rho(z))$ and establishing (\ref{eq:logass}), and the argument requires that we consider the cases of $z \in \mathbb{C}^+ \cap \Omega_s$ and $z \in \mathbb{C}^- \cap \Omega_s$ separately.
We will give the argument for $z \in \mathbb{C}^+ \cap \Omega_s$, and leave it to the reader to make the analogous argument for $z \in \mathbb{C}^- \cap \Omega_s$.
The function $\tilde \rho(z)$ is nonzero in the simply connected domain $\mathbb{C}^+ \cap \Omega_s$ and $\tilde \rho(z) \to 1$ and $z \to 0$ in $\mathbb{C}^+ \cap \Omega_s$.
We can therefore define the logarithm by
\begin{equation} \label{eq:deftlogrho} \log(\tilde \rho(z)) = \int_0^z \frac{1}{\tilde \rho(z')} \frac{d \tilde \rho}{dz'}(z') dz'\end{equation}
where the contour of integration from $0$ to $z$ lies in $\mathbb{C}^+ \cap \Omega_s$.
If follows from (\ref{eq:rhodelta}) and the reality of $\Delta(\lambda)$ for $\lambda \in \mathbb{R}$ that $\tilde \rho(z)$ is real for $z \in i \mathbb{R}^+$, and so $\log(\tilde \rho(z))$ must also be real by (\ref{eq:deftlogrho}).
Combining the fact that $\log(\tilde \rho(z))$ must be real for $z \in i \mathbb{R}^+$ with (\ref{eq:rhoass}), we see that the dominant asymptotic behavior of $\log(\tilde \rho(z))$ is given by $i z T$, and that for large $z$ we should expand $\log \left(1 + \frac{1}{2 i z} Q + O(z^{-2}) \right)$ in (\ref{eq:logass}) using the principle branch.

From (\ref{eq:logass}) we derive the further asymptotic descriptions
\begin{equation} \label{eq:logsqass} T^{-2} \log(\tilde \rho(z))^2 + z^2 = Q T^{-1} + O(z^{-1}),\end{equation}
\begin{equation} \label{eq:logreass} \frac{1}{2 \log(\tilde \rho(z))} = \frac{1}{2 i z} T^{-1} + O(z^{-2}),\end{equation}
and
\begin{align*} \numberthis
\label{eq:assrhoexp} \tilde \rho(z)^{xT^{-1}} & = e^{i z x} e^{ x T^{-1} \log \left(1 + \frac{1}{2 i z} Q + O(z^{-2}) \right) }  \\
& = e^{i z x} \left(1 + \frac{1}{2 i z} x T^{-1} Q + O(z^{-2}) \right)
\end{align*}
as $z \to \infty$ valid for $z \in \tilde \Omega_s$.

The function $\tilde \psi(x,z)$ is a eigenfunction of the right shift operator $\tilde \psi(x,z) \to \tilde \psi(x+T,z)$ with eigenvalue $\tilde \rho(z)$, therefore
\begin{equation} \label{eq:tpsip} \tilde \psi(x,z) = \tilde \rho(z)^{xT^{-1}} p(x,z)\end{equation}
where $p(x,z)$ has period $T$.
Moreover, plugging $\tilde \rho(z)^{xT^{-1}} p(x,z)$ into Hill's equation gives the Sturm--Liouville differential equation
\begin{equation} \label{eq:periodic} \frac{d}{dx} \left( \tilde \rho(z)^{2 x T^{-1}} \frac{d p(x,z)}{dx} \right) = \tilde \rho(z)^{2 x T^{-1}} \left(u(x) - T^{-2} \log(\tilde \rho(z))^2 - z^2\right) p(x,z) \end{equation}
for $p(x,z)$.
We can rewrite equation (\ref{eq:periodic}) in integral form as
\begin{align*} \numberthis \label{eq:inteqp}
p(x,z) = & p(0,z) + \frac{Tp'(0,z)}{2 \log \tilde \rho(z)} \left(1 - \tilde \rho(z)^{-2 x T^{-1}} \right) \\
& + \frac{T}{2 \log \tilde \rho(z)} \int_0^x \left(1 - \tilde \rho(z)^{2 (t-x) T^{-1}} \right) (u(t) - Q T^{-1} + \epsilon(z)) p(t,z) dt
\end{align*}
where $\epsilon$ is the $O(z^{-1})$ error term in (\ref{eq:logsqass}).
From (\ref{eq:tpsip}) we find $p(0,z) = 1$ and
\begin{equation} p'(0,z) = \tilde \psi'(0,z) - \frac{\log(\tilde \rho(z))}{T}.\end{equation}
Using (\ref{eq:propx01}) and (\ref{eq:logass}) in the above formula for $p'(0,z)$ implies the asymptotic behavior $p'(0,z) = O(z^{-1})$ as $z \to \infty$ for $z \in \tilde \Omega_s$.

The existence and uniqueness of the solution to equation (\ref{eq:inteqp}) when $z \in \tilde \Omega_s \cap \mathbb{C}$ is equivalent to showing invertibility of
\begin{equation} (1-F) : L^\infty [0,T] \to L^\infty [0,T] \text{ or } (1-F) : L^\infty [-T,0] \to L^\infty [-T,0], \end{equation}
where
\begin{align*} \numberthis
F f (x) = & \frac{T}{2 \log \tilde \rho(z)} \int_0^x \left(1 - \tilde \rho(z)^{2 (t-x) T^{-1}} \right)  (u(t) - QT^{-1} + \epsilon(z)) f(t) dt 
\end{align*}
because $p$ is periodic with period $T$.
This just requires that we show that $\sum_{n=0}^\infty F^n g$ converges for ${g \in L^\infty [0,T]}$ or $g \in L^\infty [-T,0]$ respectively.

Since $|\tilde \rho(z)| < 1$ for $z \in \mathbb{C}^+$ and $|\tilde \rho(z)| > 1$ for $z \in \mathbb{C}^-$, the bound
\begin{equation} \label{eq:nextgen}|F^n g(x)| \le \frac{  (T^2 \|u\|_{L^\infty}(\mathbb{R}) + T|Q| + T^2 |\epsilon(z)|)^n}{ |\log \tilde \rho(z)|^n n!} \|g\|_{L^\infty (I)} \end{equation}
holds for $z \in \tilde \Omega_s \cap \mathbb{C}^-$ and $x \in I = [0,T]$ or $z \in \tilde \Omega_s \cap \mathbb{C}^+$ and $x \in I = [-T,0]$.
We never have to worry about $\log(\tilde \rho(z)) = 0$ in the above bound because $\tilde \rho(z) = 1$ happens only at the periodic eigenvalues which are real.
Therefore, the Neumann series
\begin{equation} \label{eq:nextgenp} p(x,z) = \sum_{n = 0}^\infty F^n (f) (x,z)\end{equation}
is the unique solution to (\ref{eq:periodic}) for $z \in \tilde \Omega_s \cap\mathbb{C}$, where from (\ref{eq:inteqp})
\begin{equation} \label{eq:fboundK}
f(x,z) = 1 + \frac{Tp'(0,z)}{2 \log \tilde \rho(z)} \left(1 - \tilde \rho(z)^{-2 x T^{-1}} \right) = 1 + O(z^{-2}) .\end{equation}
The estimate in (\ref{eq:fboundK}) holds for $z \in \Omega_{s} \cap \mathbb{C}_{-}$ and $x \in (0,T)$. and for $z \in \Omega_{s} \cap \mathbb{C}_{+}$ and $x \in (-T,0)$.

There exists some radius $R$ such that $|\epsilon(z)| \le M_\epsilon$ and $|\log \tilde \rho(z)| \ge M_\rho$ for $z \in \tilde \Omega_s \cap \{|z| \ge R\}$ because of the asymptotic descriptions (\ref{eq:logass}) and $\epsilon(z) = O(z^{-1})$.
Therefore, by (\ref{eq:nextgen}) and (\ref{eq:nextgenp}) the bound
\begin{equation} \label{eq:punif} |p(x,z)| \le \exp\left( \frac{T^2 \|u\|_{L^\infty(\mathbb{R})} + T|Q| + T^2 M_\epsilon}{M_\rho} \right) \|f(\cdot,z)\|_{L^\infty(I)} \end{equation}
where $I = [0,T]$ for $z \in \mathbb{C}^-$ and $I = [-T,0]$ for $z \in \mathbb{C}^+$ holds for $|z| \in \tilde \Omega_s\cap \{z \ge R\}$.
It follows that $p(x,z)$ is bounded at $z \to \infty$ in the desired domains.

Combing (\ref{eq:fboundK}) with (\ref{eq:nextgen}) and using (\ref{eq:logreass}) lets us conclude that $F^n(f)(x,z) = O(z^{-2})$ as $z \to \infty$ for $n \ge 2$, where $z \in \tilde \Omega_s$ and $x \in (0,T]$ for $z \in \mathbb{C}^-$ and $x \in [-T,0)$ for $z \in \mathbb{C}^+$.
From (\ref{eq:logreass}), (\ref{eq:fboundK}) and $\epsilon(z) = O(z^{-1})$ we find
\begin{align}
\label{eq:Fboundf1} F(f)(x,\lambda) & = \frac{1}{2iz} \int_0^x \left(1 - \tilde \rho(z)^{2 (t-x) T^{-1}} \right)  (u(t) - QT^{-1}) dt + O(z^{-2}) \\
\label{eq:Fboundf2} & = \frac{1}{2iz} \int_0^x (u(t) - QT^{-1}) dt + O(z^{-2}).
\end{align}

Therefore, the Neumann series solves for $p$ and up to $O(z^{-2})$ as
\begin{equation} \label{eq:pexp} p(x,z) = 1 + \frac{1}{2 i z} \int_0^x u(t) dt - \frac{x T^{-1} Q}{2 i z} + O(z^{-2})\end{equation}
as $z \to \infty$.
By multiplying the asymptotic description (\ref{eq:pexp}) for $p(x,z)$ by the asymptotic description (\ref{eq:assrhoexp}) for $\tilde \rho(z)^{xT^{-1}}$ we find that $\tilde \psi$ has the  asymptotic description
\begin{equation} \label{eq:assnew123} \tilde \psi (x,z) = \tilde \rho(z)^{x T^{-1}} p (x,z) = e^{i z x} \left( 1 + \frac{1}{2 i z} \int_0^x u(t) dt + O(z^{-2}) \right)\end{equation}
as $z \to \infty$ valid for $z \in \tilde \Omega_s$. 

There are no obstacles to differentiating (\ref{eq:assnew123}) term-by-term, so $\tilde \psi_x$ has the asymptotic description
\begin{equation}
\tilde \psi_x (x,z) = \tilde \rho(z)^{xT^{-1}}(\log(\tilde \rho(z)) T^{-1} p(x,z) + p_x(x,z)) = (iz + O(1))e^{izx}\end{equation}
as $z \to \infty$ valid for $z \in \tilde \Omega_s$.
 \end{proof}

\section{Phr\'{a}gmen--Lindel\"{o}f Theorem} \label{sec:PL}
One of the difficulties in dealing with the Riemann--Hilbert problem for infinite gap potentials is that we will not be able to have uniform asymptotic bounds that hold as the complex parameter approaches infinity in any direction. Instead our asymptotic descriptions will only be valid in regions which exclude some sector containing the positive semiaxis.
To be able to apply Liouville's theorem, we establish bounds on our Riemann--Hilbert problem in these bad sectors as in the following version of the Phragm\'{e}n--Lindel\"{o}f theorem from \cite{Ti39}. 

\begin{thm}[Phragm\'{e}n--Lindel\"{o}f] \cite[page  177]{Ti39}
\label{thm:PL}
Let $\alpha,\beta \in \mathbb{R}$ with $\alpha<\beta$ be such that $\beta - \alpha < 2 \pi$ and let $p > 0$ satisfy $p < \frac{\pi}{|\beta - \alpha|}$.
Choose a branch of $\arg$ so that $(\alpha,\beta) \subset \text{Im}(\arg)$.
Let $f(\lambda)$ be a holomorphic function for
\begin{equation}\lambda \in \Upsilon_{\alpha,\beta} := \{ \lambda \in \mathbb{C} : \alpha < \arg(\lambda) < \beta \}\end{equation}
that extends to a continuous function on $\overline{\Upsilon}_{\alpha,\beta}$.
If $f(\lambda)$ is bounded for $\lambda \in \partial \Upsilon_{\alpha,\beta}$ and $|f(\lambda)| \le M e^{c |\lambda|^p}$ for $\lambda \in \Upsilon_{\alpha,\beta}$ then $f(\lambda)$ is bounded for $\lambda \in \Upsilon_{\alpha,\beta}$.
\end{thm}

The following algebras of holomorphic functions will aid in the establishment of bounds of the form $|f(\lambda)| \le M e^{c |\lambda|^p}$.
\begin{defn} \label{def:algp}
Let $\Omega \subset \mathbb{C}$ be some domain, and let $\mathcal{A}(\Omega)$ be the algebra of holomorphic functions on $\Omega$.
For $p > 0$ we define the Phragm\'{e}n--Lindel\"{o}f subalgebras
\begin{equation} \mathcal{A}_p (\Omega) = \{ f \in \mathcal{A}(\Omega) : \text{There exist } c,M > 0 \text{ such that } |f(\lambda)| \le M e^{c |\lambda|^p} \}. \end{equation}
\end{defn}

We now list some propositions about functions in $\mathcal{A}_p (\Omega)$ which will be useful in the next section. Since these are standard, we omit their proofs.

\begin{prop} \label{prop:pl0} If $\Omega' \subset \Omega$ and $p' \ge p$, then $\mathcal{A}_p (\Omega) \subset \mathcal{A}_{p'}(\Omega')$. \end{prop}

\begin{prop} \label{prop:pl1} If $f$ is an entire function of order $r$ and $p > r$, then $f \in \mathcal{A}_p (\mathbb{C})$.  \end{prop}

\begin{prop} \label{prop:pl2} $\mathcal{A}_p(\Omega)$ is a subalgebra of $\mathcal{A}(\Omega)$. If $f \in \mathcal{A}_p (\Omega)$ and $0 < q < 1$, then $f^q \in \mathcal{A}_p (\Omega')$ where { $\Omega' \subset \Omega$ is some domain on which $f^q$ can be defined as a single valued function}.  \end{prop}

\begin{clm} \label{clm:opennessproof}
Suppose that $\overline{D}_n$ for $n = 0, \dots, \infty$ are a family of closed discs of radius $R > 0$ for which there exists some $N$ such that $\overline{D}_n$ are disjoint for $n > N$.
Then the domain
\begin{equation} \Omega = \mathbb{C} \setminus \bigcup_{n = 0}^\infty \overline{D}_n \end{equation}
is an open set.
\end{clm}

\begin{proof}
Since there are only a finite number of overlaps of $\overline{D}_n$ and a finite intersection of open sets is open, it suffices to prove that the domain $\Omega = \mathbb{C} \setminus \bigcup_{n = 0}^\infty \overline{D}_n$ is open when $\overline{D}_n$ are disjoint.

We suppose to the contrary that $\Omega$ is not open.  Then there is a $\lambda \in \Omega$ such that there is no neighborhood of $\lambda$ contained in $\Omega$.  This means that for each small $\epsilon_{0} > 0$, there would exist a disc $D_{n(\epsilon_{0})}$ such that $\lambda \notin \overline{D_{n(\epsilon_{0})}}$, and $\mbox{dist}(\lambda, \overline{D_{n(\epsilon_{0})}}) < \epsilon$.  Choosing $\epsilon_{1} < \mbox{dist}(\lambda, \overline{D_{n(\epsilon_{0})}})$, we obtain a second disc ${D_{n(\epsilon_{1})}}$, such that $\lambda \notin \overline{D_{n(\epsilon_{1})}}$, and $\mbox{dist}(\lambda, \overline{D_{n(\epsilon_{1})}}) < \epsilon_{1}$.  Therefore, there exists an infinite sequence of disjoint discs with fixed radius, all within a compact set, since the centers of all these discs are within distance $R + \epsilon_{0}$ of $\lambda$.  This is a contradiction.  Therefore, the set $\Omega$ must be open.  

\end{proof}

While the above properties of these algebras of holomorphic functions are standard, the following two propositions are (to the best of our knowledge) new, and useful for dealing with Bloch--Floquet solutions of Hill's equation and associated infinite-product expansions.

\begin{prop} \label{prop:pl3}
Suppose that $\Omega \subset \mathbb{C}$ is some domain and let $\overline{D}_n \subset \Omega$ for $n = 1, 2, \dots , \mathcal{N}$ ($\mathcal{N}$ is either finite or infinite) be a family of discs with radii $R > 0$, and when $\mathcal{N}$ is infinite we assume there exists $N$ such that $\overline{D}_n$ are disjoint for $n > N$. Then
\begin{equation} \mathcal{A}(\Omega) \cap \mathcal{A}_p \left( \Omega \setminus \bigcup_{n = 1}^\mathcal{N} \overline{D}_n\right) \subset \mathcal{A}_p (\Omega). \end{equation}
\end{prop}

\begin{proof}[Proof of Proposition \ref{prop:pl3}]
Suppose that
\begin{equation} f \in \mathcal{A}(\Omega) \cap \mathcal{A}_p \left( \Omega \setminus \bigcup_{n = 1}^\mathcal{N} \overline{D}_n \right). \end{equation}
Then there exist constants $c,M > 0$ such that $|f(\lambda)| \le M e^{c |\lambda|^p}$ for $\lambda \in \Omega \setminus \bigcup \overline{D}_n$.
We will prove that $f$ must then satisfy the slightly weaker bound $|f(\lambda)| \le M' e^{c' |\lambda|^p}$ for all $\lambda \in \Omega$ in terms of
\begin{equation}
c' = (1 + R)^p c, \text{ and } M' = \max(\{M\} \cup \{|f(\lambda)| : \lambda \in \bigcup_{n = 1}^\mathcal{N} \overline{D}_n, |\lambda| \le K \})
\end{equation}
where $K \ge 2$ is a cut off such that
\begin{equation} \bigcup_{n = 0}^N \overline{D}_n \subset \{ \lambda \in \mathbb{C}, |\lambda| \le K\}. \end{equation}
This implies $f \in \mathcal{A}_p (\Omega)$.
The bound $|f(\lambda)| \le M' e^{c' |\lambda|^p}$ holds outside of the discs $\overline{D}_n$ so we just need to extend the bound into the discs $\overline{D}_n$.

The domain $\Omega \subset \mathbb{C}$ is arbitrary so it can potentially contain both large and small values of $\lambda$.
The argument used to obtain the bound $|f(\lambda)| \le M' e^{c' |\lambda|^p}$ for large $\lambda$ requires an upper bound on $(1 + 2 R|\lambda|^{-1})$ and require $\overline{D}_n$ to be disjoint and so does no work for small values of $\lambda$.
We therefore impose the cut off $K$ and handle the cases $|\lambda| \ge K$ and $|\lambda| < K$ separately.
Suppose that $\{\lambda \in \bigcup_{n = 1}^{\mathcal{N}} \overline{D}_n, |\lambda| \le K\}$ is nonempty.
Then $\{\lambda \in \bigcup_{n = 1}^{\mathcal{N}} \overline{D}_n, |\lambda| \le K\}$ is compact so
\begin{equation} 0 < \max\{|f(\lambda)| : \lambda \in \bigcup_{n = 1}^{\mathcal{N}} \overline{D}_n, |\lambda| \le K \} < \infty. \end{equation}
Therefore, $M' < \infty$ and $|f(\lambda)| \le M' \le M' e^{c |\lambda|^p}$ for $\lambda \in \{\lambda \in \bigcup_{n = 1}^{\mathcal{N}} D_n,\lambda \le K\}$. 

Now suppose that $\{\lambda \in D_n, |\lambda| \ge K\}$ is nonempty and consider $\lambda \in D_n$ with $|\lambda| \ge K$.
$f(\lambda)$ is holomorphic for $\lambda \in D_n$, so $|f(\lambda)|$ is bounded for $\lambda \in D_n$ by the maximum of $|f(\lambda)|$ for $\lambda \in \partial D_n$ by the maximum principal.
Since $D_n$ is disjoint from $D_{n'}$ for any $n' \ne n$, the bound $f(\lambda) \le M e^{c|\lambda|^p}$ extends to $\partial D_n$ by continuity.
Let $c_n$ be the center of $D_n$, then there exists a unit complex number $\nu$ such that
\begin{equation}|f(\lambda)| \le M' e^{c |c_n + \nu R|^p}\end{equation}
for $\lambda \in D_n$ since $M' e^{cr^p}$ is an increasing function of $r$.
The triangle inequality gives
\begin{equation}|f(\lambda)| \le M' e^{c |c_n + R|^p} \le M' e^{c (|c_n| + R)^p}\end{equation}
for $\lambda \in D_n$.
Moreover,
\begin{equation}|c_n| = |z + c_n - z| \le |z| + |c_n - z| \le |z| + R \end{equation}
for $\lambda \in D_n$ by the triangle inequality, and so
\begin{equation}|f(\lambda)| \le M' e^{c (|\lambda| + 2R)^p} = M' e^{c (1 + 2 R |\lambda|^{-1})^p |\lambda|^p} \le M' e^{c (1 + 2R/K)^p |\lambda|^p}\end{equation}
for $\lambda \in \{\lambda \in D_n : |\lambda| \ge K\}$.
\\ \end{proof}

\begin{prop} \label{prop:pl4}
Suppose that $\{\eta_n \}_{n = 1}^\infty$ is a sequence of nonzero complex numbers with $|\eta_n|$ nondecreasing for which there exists $N,C > 0$ such that $|\eta_n| \ge C n^2$ for $n \ge N$.
Let $D_n$ for $n = 0, 1, \cdots , \infty$ be a family of open discs of radii $R > 0$ centered at $c_n$ with $c_0 = 0$,  such that $\eta_n \in \{ \lambda \in \mathbb{C} : |\lambda - c_n| \le R/2 \}$ for $n > 0$.
Moreover, assume that there exists some $\tilde N$ such that $\overline{D}_n$ are disjoint for $n > \tilde N$.
Consider the canonical product
\begin{equation} P(\lambda) = P(\lambda) = \lambda^m \prod_{n = 1}^\infty \left(1 - \frac{\lambda}{\eta_n}\right)\end{equation}
for $m \ge 0$. Then
\begin{equation} P^{-1} \in \mathcal{A}_1 \left( \mathbb{C} \setminus \bigcup_{n = 0}^\infty \overline{D}_n\right) \end{equation}
($P^{-1}$ denotes the multiplicative inverse of $P$ rather than its inverse function).
\end{prop}

The condition that $\overline{D}_n$ are disjoint for large $n$ is only used to guarantee that
$\mathbb{C} \setminus \bigcup_{n = 0}^\infty \overline{D}_n$ is an open set.
This proposition is true so long as this set is open.
The proof of this proposition relies on the following claim.

\begin{clm} \label{claim:pl4} Suppose that $\{\eta_n \}_{n = 1}^\infty$ satisfies the hypotheses on the sequence in proposition \ref{prop:pl4}.
Let $\mathfrak{N}(\lambda) \ge 0$ be the smallest integer such that $|\eta_n| > 2 |\lambda|$ for $n > \mathfrak{N}(\lambda)$.
Then $\mathfrak{N}(\lambda) \le D \sqrt{|\lambda|}$ where
\begin{equation} D = \max \left\{ N \sqrt{\frac{2}{|\eta_1|}} , \sqrt{\frac{2}{C}} \right\}. \end{equation} \end{clm}

\begin{proof}[Proof of Claim \ref{claim:pl4}] 
The argument for this claim is different for $2 |\lambda| < |\eta_1|$, $|\eta_1| \le 2 |\lambda| \le |\eta_N|$ and $|\eta_N| < 2 |\lambda|$.
The case of $2 |\lambda| < |\eta_1|$ is the easiest because then $|\eta_n| > 2 |\lambda|$ for all $n > 0$ and so $\mathfrak{N}(\lambda) = 0 \le D \sqrt{|\lambda|}$.
We now consider the case of $|\eta_1| \le 2 |\lambda| \le |\eta_N|$.
$\mathfrak{N}(\lambda) \le N$ for all $\lambda$ such that $2 |\lambda| < |\eta_N|$.
Therefore
\begin{equation} \mathfrak{N}(\lambda) \le \frac{N}{\sqrt{|\eta_1|}} \sqrt{|\eta_1|} \le N \sqrt{\frac{2}{|\eta_1|}}  \sqrt{|\lambda|} \end{equation}
for $|\eta_1| \le 2 |\lambda| \le |\eta_N|$.

We now consider the case when $2 |\lambda| \ge |\eta_N|$.
We will use two inequalities: $C \mathfrak{N}(\lambda)^2 \le |\eta_{\mathfrak{N}(\lambda)}|$ and $|\eta_{\mathfrak{N}(\lambda)}| \le 2 |\lambda|$. 
The first is verified by noting that $\mathfrak{N}(\lambda) \ge N$ and therefore $C \mathfrak{N}(\lambda)^2 \le |\eta_{\mathfrak{N}(\lambda)}|$ by the hypotheses on $\{\eta_n\}_{n = 1}^\infty$ in proposition \ref{prop:pl4}.
If the second inequality were not true then $M = \mathfrak{N}(\lambda)-1 < \mathfrak{N}(\lambda)$ would be an integer such that $|\eta_n| > 2 |\lambda|$ for $n > M$, however this contradicts the fact that $\mathfrak{N}(\lambda)$ was defined to be the smallest such integer.
Therefore the second inequality is true.
Combining these two inequalities yields $\mathfrak{N}(\lambda) \le \sqrt{\tfrac{2}{C}} \sqrt{|\lambda|}$. \end{proof}

\begin{proof}[Proof of Proposition \ref{prop:pl4}]
Suppose that $\lambda \in \mathbb{C} \setminus \bigcup D_n$.
The canonical product breaks up as ${P(\lambda) = \lambda^m P_1(\lambda) P_2(\lambda)}$ where
\begin{equation}P_1(\lambda) = \prod_{n \le \mathfrak{N}(\lambda)} \left(1 - \frac{\lambda}{\eta_n}\right)\end{equation}
and
\begin{equation}P_2(\lambda) = \prod_{n > \mathfrak{N}(\lambda)} \left(1 - \frac{\lambda}{\eta_n} \right).\end{equation}
$\lambda^{-m}$ clearly satisfies $|\lambda^{-m}| \le R^{-m} e^{\sqrt{|\lambda|}}$ for $\lambda \in \mathbb{C} \setminus D_0$ so $\lambda^{-m} \in \mathcal{A}_1(\mathbb{C} \setminus D_0)$.
By the reverse triangle inequality
\begin{equation} |\lambda - \eta_n| = |\lambda - c_n - \eta_n  + c_n| \ge ||\lambda - c_n| - |\eta_n - c_n||.\end{equation}
Moreover, combining $|\lambda - c_n| \ge R$ and $|\eta_n - c_n| \le R/2$ implies $|\lambda - c_n| - |\eta_n - c_n| > 0$, and so $|\eta_n - \lambda| \ge R/2$.
We can bound $P_1(\lambda)$ from below by
\begin{equation}|P_1(\lambda)| = \prod_{n = 1}^{\mathfrak{N}(\lambda)} \frac{1}{|\eta_n|}{|\eta_n - \lambda|} \ge \prod_{n = 1}^{\mathfrak{N}(\lambda)} \frac{R}{2 |\eta_n|} \ge R^{\mathfrak{N}(\lambda)} (2 |\eta_{\mathfrak{N}(\lambda)}|)^{-\mathfrak{N}(\lambda)}.\end{equation}
We can then bound $|P_1(\lambda)|^{-1}$ from above by
\begin{align} |P_1(\lambda)|^{-1} & \le (2 \eta_{\mathfrak{N}(\lambda)} R^{-1})^{\mathfrak{N}(\lambda)} \le (4 R^{-1} |\lambda|)^{\mathfrak{N}(\lambda)} \\ & \le e^{D \sqrt{|\lambda|} (\log(4 R^{-1}) + \log(|\lambda|)) } \le M e^{c |\lambda|}\end{align}
for some positive constants $c$ and $M$.

The set $A := \{\lambda \in \mathbb{C} : \lambda \le |\eta_N|\} \setminus \bigcup  D_n$ is a compact set on which $P_2(\lambda)^{-1}$ is piecewise polynomial with a finite number of domains of definition.
Therefore, there exists some $M \ge 1$ such that $|P_2(\lambda)|^{-1} \le M$ for all $\lambda \in A$.
Let us pick $\lambda \in \mathbb{C} \setminus \bigcup  D_n$ such that $\lambda > |\eta_N|$ so that $\mathfrak{N}(\lambda) \ge N$.
Taking the principal branch of the logarithm gives
\begin{equation} \log(P_2(\lambda)) = \sum_{n > \mathfrak{N}(\lambda)} \log \left( 1 - \frac{\lambda}{\eta_n} \right) = \sum_{n > \mathfrak{N}(\lambda)} \sum_{m = 1}^\infty -\frac{\lambda^m}{m \eta_n^{m}},\end{equation}
where the series expansion of the logarithm converges since $|\eta_n| > 2 |\lambda|$ implies $\frac{|\lambda|}{|\eta_n|} \le \frac{1}{2}$.
The trivial bound $\frac{1}{m} < 1$ and the geometric series gives
\begin{align*} \numberthis |\log(P_2(\lambda))| & \le \sum_{n > \mathfrak{N}(\lambda)} \sum_{m = 1}^\infty \frac{|\lambda|^{m}}{|\eta_n|^{m}} = \sum_{n > \mathfrak{N}(\lambda)} \frac{|\lambda|}{|\eta_n|} \left( 1 - \frac{|\lambda|}{|\eta_n|} \right)^{-1} \\ & \le 2 |\lambda| \sum_{n = \mathfrak{N}(\lambda)}^\infty \frac{1}{|\eta_n|} \le \frac{2}{C}|\lambda| \sum_{m = 1}^\infty n^{-2} = \frac{\pi^2}{3 C} |\lambda|. \end{align*}
This gives the lower bound
\begin{equation}|P_2(\lambda)| = e^{\text{Re}(\log(P_2(\lambda)))} \ge e^{-\frac{\pi^2}{3 C} |\lambda|}\end{equation}
which implies $|P_2(\lambda)|^{-1} \le e^{\frac{\pi^2}{3 C} |\lambda|}$ holds for all $\lambda \in \mathbb{C} \setminus \bigcup  D_n$ such that $\mathfrak{N}(\lambda) \ge N$.
Therefore, the bound $|P_2(\lambda)|^{-1} \le M e^{\frac{\pi^2}{3C} |\lambda|}$ holds for all $z \in \mathbb{C} \setminus \bigcup  D_n$.
Finally, $\lambda^{-m},P_1^{-1}, P_2^{-1}$ all satisfy bounds of the form $|\lambda^{-m}|\le M e^{c|\lambda|}$, $|P_1(\lambda)^{-1}| \le M e^{c|\lambda|}$, $|P_2(\lambda)| \le M e^{c|\lambda|}$, and therefore so does $P(\lambda)^{-1} = \lambda^{-m}P_1(\lambda)^{-1}P_2(\lambda)^{-1}$. Moreover, $P(\lambda)^{-1}$ is holomorphic in $\mathbb{C} \setminus \bigcup \overline{D}_n$, so
\begin{equation} P^{-1} \in \mathcal{A}_1 \left( \mathbb{C} \setminus \bigcup \overline{D}_n \right). \end{equation} \end{proof}

\section{Inverse Spectral Theory for Periodic Potentials}

In this section we will produce a Riemann--Hilbert problem from which the periodic potential $u \in L^\infty (\mathbb{R})$ can be recovered from its solutions.
We will characterize the solutions as unique up to left multiplication by $2 \times 2$ lower triangular matrices with ones on the diagonal.

To describe our Riemann--Hilbert problem we need a notation for boundary values.
\begin{defn}
Let $\Gamma \subset \mathbb{R}$ be some contour with a possibly infinite number of components and let ${f : \mathbb{C} \setminus \Gamma \to \mathbb{C}}$ be a holomorphic function.
We define the boundary values $f_\pm(\lambda)$ for some $\lambda \in \Gamma$ by
\begin{equation} f_\pm (\lambda) = \lim_{\epsilon \to 0^+} f(\lambda \pm i \epsilon) \end{equation}
provided the limits exist.
\end{defn}

In defining our Riemann--Hilbert problems we include a subexponential bound that does not extend into some closed discs $\overline{D}_k$ that contain the spectral gaps.

\begin{defn} \label{def:discs}
Suppose $\lambda_n$ are defined as in Theorem \ref{thm:forspec} and $E_k$ are defined as in  definition \ref{def:SpecData}.
Let $R = \max \{E_{2k} - E_{2k-1}\}_{k = 1}^\mathcal{G}$.
We define $D_n$ to be the open disc of radius $R$ centered at $c_n := (\lambda_{2n}+\lambda_{2n-1})/2$ for $n > 0$ and $c_0 := 0$,
and
\begin{equation} \label{eq:defD}
\mathcal{D} := \mathbb{C} \setminus \left(\mathbb{R}^+ \cup \bigcup_{k = 0}^{\mathcal{G}} \overline{D}_{n_k} \right)
\end{equation}
where in (\ref{eq:defD}) we have added the subindex $n_0 = 0$ to the subindices $n_k$.

\end{defn}

The upshot of definition \ref{def:discs} is that: the subsequence of discs $D_{n_k}$ are defined in terms of data included in $\Sigma (q)$ alone, and the full sequence of discs $D_n$ have radii $R \ge (\lambda_{2n} - \lambda_{2n-1})$ and so
\begin{equation} \mu_n \in [\lambda_{2n-1},\lambda_{2n}] \subset \{\lambda \in \mathbb{C} : |\lambda - c_n| \le R/2\}. \end{equation}
Moreover, there exists $C$ and $N$ such that $\mu_n, \lambda_{2n-1}, \lambda_{2n} \ge C n^2$ by part 4 of Theorem \ref{thm:forspec}.
Therefore, the following claim is valid:
\begin{clm} \label{prop:pl4hyp} The discs $D_n$ together with $\{\eta_n\}_{n = 1}^\infty$ satisfy the hypotheses of proposition \ref{prop:pl4} for each of the following choices for $\eta_n$: $\eta_n = \lambda_{2n-1}$, $\eta_n = \lambda_{2n}$ or $\eta_n = \mu_n$. \end{clm}

The set $\mathcal{D}$ can be seen to be an open set by combining claim \ref{clm:opennessproof} with the fact that $\mathbb{C} \setminus \mathbb{R}^+$ is also an open, and finite intersections of open sets are open.
The set
\begin{equation} \mathbb{C} \setminus \left(\mathbb{R}^+ \cup \bigcup_{n = 0}^{\infty} \overline{D}_{n} \right)\end{equation}
is open for the same reason.
The sets
\begin{equation} \mathbb{C} \setminus \bigcup_{n = 0}^\infty \overline{D}_n, \quad \mathbb{C} \setminus \bigcup_{k = 0}^{\mathcal{G}} \overline{D}_{n_k}\end{equation}
are also open because of claim \ref{clm:opennessproof}.

\begin{defn}
\label{prop:rhpPsi} The matrix Bloch--Floquet solution $\Psi$ is given by
\begin{equation} \label{eq:defPsi}
\Psi(x,\lambda) = \begin{pmatrix} \psi^-(x,\lambda) & \psi^+(x,\lambda) \\ {\psi^-}'(x,\lambda) & {\psi^+}'(x,\lambda) \end{pmatrix}.
\end{equation}
\end{defn}
We will use two notations for entries in this matrix.
One is $\psi^\pm$ notation introduced in the above definition, the other notation is the usual matrix notation $\psi_{ij}$ for the $i,j$th component of $\Psi$.

\begin{defn} \label{defn:regularize}
Let {$B$} be the $2 \times 2$ matrix valued function
\begin{equation} \label{eq:deformA} B(\lambda) :=  \frac{i\sqrt{f^0(\lambda)}}{\sqrt[4]{\Delta(\lambda)^2-4}} \begin{pmatrix} f^-(\lambda)  & 0 \\ 0 & f^+(\lambda) \end{pmatrix} \end{equation}
which is defined in terms of the following potentially infinite products
\begin{equation} f^+ (\lambda) := \prod_{\substack{n = 1 \\ \sigma_n = 1}}^\infty \frac{T^2}{n^2 \pi^2} (\mu_n - \lambda), \end{equation}
\begin{equation} f^- (\lambda) := \prod_{\substack{n = 1 \\ \sigma_n = -1}}^\infty \frac{T^2}{n^2 \pi^2} (\mu_n - \lambda), \end{equation}
and
\begin{equation} f^0 (\lambda) := \prod_{\substack{n = 1 \\ \sigma_n = 0}}^\infty \frac{T^2}{n^2 \pi^2} (\mu_n - \lambda). \end{equation}
\end{defn}

\begin{rmk} \label{eq:rmkrootRHP}
The functions $\sqrt{f^0(\lambda)}$ and $\sqrt[4]{\Delta(\lambda)^2-4}$ can be defined explicitly via the product expansions
\begin{equation} \sqrt{f^0(\lambda)} = \prod_{\substack{n = 1 \\ \sigma_n = 0}}^\infty (-i) \frac{T}{n \pi} \sqrt{\lambda- \mu_n}, \end{equation}
\begin{equation} \sqrt[4]{\Delta(\lambda)^2-4} = \sqrt{2} e^{-i \frac{\pi}{4}} \sqrt[4]{\lambda-\lambda_0} \prod_{n = 1}^\infty (-i)\frac{T}{n \pi} \sqrt[4]{\lambda-\lambda_{2n}} \sqrt[4]{\lambda-\lambda_{2n-1}}.\end{equation}
The quartic roots of the elementary factors in the second expansion are defined using the non-principal branch of the argument taking values in $[0,2\pi)$, in the same manner that we defined the square root (see remark \ref{rmk:defbranch}).
The function $\sqrt{f^0(\lambda)}$ has been defined so that it is holomorphic for ${\lambda \in \mathbb{C} \setminus \overline{ \{\lambda \in \mathbb{R} : |\{\lambda_n : \sigma_n = 0, \lambda_n \le \lambda\}| \text{ is odd} \}}}$ and is positive for real $\lambda < \displaystyle{\min_{\sigma_n = 0}}(\lambda_n)$.
The function $\sqrt[4]{\Delta(\lambda)^2-4}$ has been defined so that it is holomorphic for $\lambda \in \mathbb{C} \setminus \bigcup_{n = 1}^\infty [\lambda_{4n-4},\lambda_{4n-1}]$ and positive for real $\lambda < \lambda_0 = 0$.
\end{rmk}
Important properties of these functions will be discussed in what follows, but we should observe at this juncture that the asymptotic behavior of the $\mu_{n}$'s given by (\ref{eq:assspec123}) implies that the (possibly infinite) products are convergent, and the functions so constructed are analytic.
Moreover, the formal equalities
\begin{equation} \frac{\sqrt{f^0(\lambda)}}{\sqrt[4]{\Delta(\lambda)^2-4}} = \frac{\sqrt{2}}{2} \left( \frac{1}{E_0 - \lambda} \prod_{\substack{k = 1 \\ \sigma_{n_k} = 0}}^\mathcal{G} \frac{(\mu_{n_k}-\lambda)^2}{(E_{2k-1}-\lambda)(E_{2k}-\lambda)} \prod_{\substack{k = 1 \\ \sigma_{n_k} \ne 0}}^\mathcal{G} \frac{{n_k}^4 \pi^4}{T^4} \frac{1}{( E_{2k-1} - \lambda )(E_{2k}-\lambda)} \right)^{\frac{1}{4}}, \end{equation}
and
\begin{equation} f^+(\lambda) = \prod_{\substack{k = 1 \\ \sigma_{n_k} = 1}}^{\mathcal{G}} \frac{T^2}{n_k^2 \pi^2} (\mu_{n_k} - \lambda), \quad  f^-(\lambda) = \prod_{\substack{k = 1 \\ \sigma_{n_k} = -1}}^{\mathcal{G}} \frac{T^2}{n_k^2 \pi^2} (\mu_{n_k} - \lambda),\end{equation} 
show explicitly that $B(\lambda)$ depends only on $\Sigma(q)$.

\begin{defn}
Let $V : \mathbb{R}^+ \setminus \{E_j\}_{j = 0}^{2 \mathcal{G}} \to SL(2,\mathbb{C})$ be given by
\begin{equation} \label{eq:jump} V (\lambda) := \begin{cases} (-1)^{k + m(\lambda) - 1} \begin{pmatrix} 0 & i \frac{f^+(\lambda)}{f^-(\lambda)} \\ i \frac{f^-(\lambda)}{f^+(\lambda)} & 0 \end{pmatrix} & \lambda \in (E_{2k-2},E_{2k-1})  \\ (-1)^{k + m(\lambda)} e^{2 i \sigma_3 \sqrt{\lambda} x} & \lambda \in (E_{2k-1},E_{2k}) \end{cases}, \end{equation}
which is defined in terms of the counting function $m:\mathbb{R} \to \mathbb{Z}$ for the Dirichlet eigenvalues on the edges of non-degenerate gaps given by
\begin{equation} \label{eq:mcounting} m(\lambda) := |\{ k \in \mathbb{N} : \mu_{n_k} \le \lambda, \sigma_{n_k} = 0 \}| .\end{equation}
The matrix $\sigma_3 = \begin{pmatrix} 1 & 0 \\ 0 & -1 \end{pmatrix}$ is the third Pauli matrix which should not be confused with the signatures $\sigma_n$.
\end{defn}

Finally, we define  $\Phi$ via
\begin{equation} \label{eq:defPhi} \Phi (x,\lambda) := \Psi(x,\lambda)  B(\lambda) e^{i \sigma_3 \sqrt{\lambda} x }. \end{equation}

\begin{prop} \label{prop:rhpPhi}
The matrix-valued function $\Phi$ solves the following Riemann--Hilbert problem: \end{prop}
\begin{rhp} \label{rhp:Phi}
For some $x \in \mathbb{R}$ find a $2 \times 2$ matrix valued function $\Phi(x,\lambda)$ such that:
\begin{enumerate}

\item $\Phi(x,\lambda)$ is a holomorphic function of $\lambda$ for $\lambda \in \mathbb{C} \setminus \mathbb{R}^+$.

\item $\Phi_{\pm} (x,\lambda)$ are continuous functions of $\lambda \in \mathbb{R}^+ \setminus \{E_j\}_{j = 0}^{2 \mathcal{G}}$ that have at worst quartic root singularities on $\{E_k\}_{k = 0}^{2 \mathcal{G}}$.

\item $\Phi_{\pm} (x,\lambda)$ satisfy the jump relation $\Phi_+(x,\lambda) = \Phi_-(x, \lambda)  V (x,\lambda)$.

\item $\Phi (x,\lambda)$ has an asymptotic description of the form
\begin{equation} \label{eq:assphi} \Phi(x,\lambda) =  \begin{pmatrix} 1 & 1 \\ -i \sqrt{\lambda} & i \sqrt{\lambda} \end{pmatrix} \left( I + O\left(\sqrt{\lambda}^{-1} \right) \right) B(\lambda) \end{equation}
as $\lambda \to \infty$ with $\lambda \in \Omega_s$ for some $0 < s < \tfrac{\pi}{8}$.

\item There exist positive constants $c$, and $M$ such that the entries $\phi_{ij}$ of $\Phi$ satisfy ${|\phi_{ij} (x,\lambda)| \le M e^{c|\lambda|^2} }$ for all $\lambda \in \mathcal{D}$.
\end{enumerate}
\end{rhp}

Before giving the proof, we have one comment on condition 5.
We will actually prove a stronger bound $|\phi_{ij}| \le M e^{c |\lambda|}$ and $0 < s < \frac{\pi}{4}$.
We choose to present the weaker bound in the theorem statement so that condition 5 of Riemann--Hilbert problem \ref{rhp:Phi} for the inverse spectral theory of Hill's operators matches condition 5 of Riemann--Hilbert problem \ref{rhp:KdV} for the KdV equation.

\begin{proof}[Proof of condition 1]
As observed in Section 2, the quanitities $\psi^{\pm}$ are meromorphic functions of $\lambda$ for $\lambda \in \mathbb{C} \setminus \ \sigma(L)$.  Differentiating (\ref{eq:bloflo}) with respect to $x$ we establish
\begin{equation} \label{eq:blofloprime} {\psi^\pm}' (x,\lambda) = y_{1}' (x,\lambda) + \frac{\rho(\lambda)^{\pm 1} - y_1(T, \lambda)}{y_2(T,\lambda)} y_{2}' (x,\lambda). \end{equation}
It follows from (\ref{eq:bloflo}) and (\ref{eq:blofloprime}) that the entries of $\Psi$ can be expressed as rational combinations of functions that are holomorphic in $\mathbb{C} \setminus \sigma(L)$, and therefore $\Psi$ is meromorphic in $\mathbb{C} \setminus \sigma(L)$.
Moreover, by looking at (\ref{eq:bloflo}), (\ref{eq:blofloprime}), (\ref{eq:deformA}) and (\ref{eq:defPhi}) we see that $\Phi$ could only be be singular on $\{\mu_n\}_{n = 1}^\infty \subset \mathbb{R}^+$ and $\{ \lambda_n \}_{n = 0}^\infty \subset \mathbb{R}^+$.
Therefore $\Phi$ is a holomorphic function of $\lambda$ for $\lambda \in \mathbb{C} \setminus \mathbb{R}^+$. \end{proof}

\begin{proof}[Proof of condition 2]
It follows from (\ref{eq:bloflo}) and (\ref{eq:blofloprime}) that $\Psi$ and $\Psi_{\pm}$ can only be singular on the zeros $\{\mu_n\}_{n = 1}^\infty$ of $y_2(T,\lambda)$.
It then follows from (\ref{eq:deformA}) and (\ref{eq:defPhi}) that $\Phi_\pm$ can only be singular on $\{\lambda_n\}_{n = 1}^\infty$ and $\{\mu_n\}_{n = 1}^\infty$.
Our first task is therefore to show that $\Phi_\pm(x,\lambda)$ are nonsingular at $\lambda = \mu_n$ unless $\mu_n = E_k$ for some $k$.
There are 3 cases to consider: $\sigma_n = 1$, $\sigma_n = -1$ and $\lambda_{2n} = \lambda_{2n-1} = \mu_n$.

Suppose that $\sigma_n = 1$, which is only possibly when $n = n_k$ for some $k$.
By recalling the definition of $\sigma_n$ and the fact that $|\rho(\lambda)| < 1$ we see that $y_1(T;\mu_n) = \rho(\mu_n)^{-1}$ using claim \ref{clm:dirflo}. 
Therefore $\rho(\lambda)^{-1} - y_1(T;\lambda)$ is a holomorphic function of $\lambda$ in some neighborhood of $\mu_n$ with a zero $\mu_n$.
The zero of $\rho(\lambda)^{-1} - y_1(T;\lambda)$ at $\mu_n$ cancels the simple zero of $y_2(T;\lambda)$ appearing in the denominator of (\ref{eq:bloflo}) and (\ref{eq:blofloprime}) when forming $\phi_{11\pm}(x,\lambda)$ and $\phi_{21\pm}(x,\lambda)$.
Here we are taking $\phi_{ij}(x,\lambda)$ to indicate the $i,j$th element of the matrix $\Phi(x,\lambda)$.
The simple zero appearing in the denominators of (\ref{eq:bloflo}) and (\ref{eq:blofloprime}) are canceled by the simple zero $f^+$ in forming $\phi_{12\pm}$ and $\phi_{22\pm}$.
Therefore $\Phi_{\pm}$ is nonsingular at $\mu_n$.

Suppose that $\sigma_n = -1$, which is only possibly when $n = n_k$ for some $k$.
By recalling the definition of $\sigma_n$ and the fact that $|\rho(\lambda)| < 1$ we see that $y_1(T;\mu_n) = \rho(\mu_n)$.
Therefore $\rho(\lambda) - y_1(T;\lambda)$ is a holomorphic function of $\lambda$ in some neighborhood of $\mu_n$ with a zero $\mu_n$.
The zero of $\rho(\lambda) - y_1(T;\lambda)$ at $\mu_n$ cancels the simple zero of $y_2(T;\lambda)$ appearing in the denominator of (\ref{eq:bloflo}) and (\ref{eq:blofloprime}) when forming $\phi_{12\pm}(x,\lambda)$ and $\phi_{22\pm}(x,\lambda)$.
The simple zero appearing in the denominators of (\ref{eq:bloflo}) and (\ref{eq:blofloprime}) are canceled by the simple zero $f^-$ in forming $\phi_{11\pm}$ and $\phi_{21\pm}$.
Therefore $\Phi_{\pm}$ is nonsingular at $\mu_n$.

Suppose that $\lambda_{2n-1} = \lambda_{2n}$.
Let us form
\begin{equation} \label{eq:deftpsi} \tilde \psi (x; z) := \begin{cases} \psi^+ (x,z^2) & z \in \{\text{Im}(z) > 0\} \\ \psi^- (x,z^2) & z \in \{\text{Im}(z) < 0\} \end{cases} \end{equation}
and
\begin{equation} \label{eq:deftrho} \tilde \rho (z) := \begin{cases} \rho (z^2) & z \in \{\text{Im}(z) > 0\} \\ \rho^{-1} (z^2) & z \in \{\text{Im}(z) < 0\} \end{cases}, \end{equation}
which are analytic in $\mathbb{C} \setminus \mathbb{R}$ and continuous on $\{z \in \mathbb{C} : z^2 \in \sigma(L) \setminus \bigcup_{n' = 1}^\infty \mu_{n'}\}$.
From (\ref{eq:bloflo}), we see that
\begin{equation} \label{eq:zbloflo} \tilde \psi(x,z) = y_1(x,z^2) + \frac{\tilde \rho(z)-y_1(T,z^2)}{y_2(T,z^2)} y_2(x,z^2).  \end{equation}
The claims of analyticity and continuity can be verified directly by considering the formulas  (\ref{eq:rhodelta},\ref{eq:rhoinvdelta}) for $\rho^\pm(\lambda)$ with $\sqrt{\Delta(\lambda)^2-4}$ given by (\ref{eq:expansionsqrtflo}), and transferring these formulas from the $\lambda$-plane to the $z$-plane. 
Therefore, $\tilde \psi$ and $\tilde \rho$  analytically continue to holomorphic functions on ${(\mathbb{C} \setminus \mathbb{R}) \cup \{z \in \mathbb{C} : z^2 \in \sigma(L) \setminus \bigcup_{n' = 1}^\infty \mu_{n'} \}}$, 
which contains punctured neighborhoods around, but not including, $\pm\sqrt{\mu_n}$.
Therefore, the calculation $\tilde \rho(\pm\sqrt{\mu_n}) - y_1(T;(\pm\sqrt{\mu_n})^2) = 0$ implies that $\tilde \rho(z) - y_1(T;z^2) = 0$ has zeros at $\pm \sqrt{\mu_n}$ that are at least simple.
These simple zeros at $\pm \sqrt{\mu_n}$ cancel the simple zeros of $y_2(T;z^2)$ at $\pm \sqrt{\mu_n}$ in (\ref{eq:zbloflo}), so $\tilde \psi (x;z)$ and $\tilde \psi' (x;z)$ are regular at $\pm \sqrt{\mu_n}$.
Upon returning to the $\lambda$ plane we see that the boundary values $\psi_\pm^{\pm}(x,\mu_n)$ exist.
Moreover, the square root zero of $\sqrt{f^0(\lambda)}$ at $\mu_n$ cancels the square root zero of $\sqrt[4]{\Delta(\lambda)^2 - 4}$ at $\lambda_{2n-1} = \lambda_{2n} = \mu_n$.
It then follows from (\ref{eq:defPhi}) that $\Phi_{\pm}(x,\lambda)$ is nonsingular at $\mu_n$. 

We will now show that $\Phi_{\pm}(\lambda,x)$ have at worst quartic root singularities at $E_j$ for $j = 0,1,\cdots,\mathcal{G}$.
Suppose that $\sigma_{n_k} = \pm 1$, then the only singular contribution to $\Phi_{\pm}(x,\lambda)$ at $\lambda \to E_{2k}$ or $, E_{2k-1}$ is from the boundary values of $1/\sqrt[4]{\Delta(\lambda)^2 - 4}$ for $\lambda \in \mathbb{R}^+$.
Suppose that $\sigma_{n_k} = 0$, then either $\mu_n = E_{2n}$ or $\mu_n = E_{2n-1}$.
If $\mu_n = E_{2n}$, then the only singular contribution to $\Phi_{\pm}(x,\lambda)$ near $E_{2n-1}$ is from the boundary values of $1/\sqrt[4]{\Delta(\lambda)^2 - 4}$ for $\lambda \in \mathbb{R}^+$.
If $\mu_n = E_{2n-1}$ then the only singular contribution to $\Phi_{\pm}(x,\lambda)$ near $E_{2n}$ is from the boundary values of $1/\sqrt[4]{\Delta(\lambda)^2 - 4}$ for $\lambda \in \mathbb{R}^+$.

If either $\mu_n = E_{2n}$ or $\mu_n = E_{2n-1}$ then $\psi_\pm^\pm (x,\lambda)$ have at worst square root singularities at $\mu_n$.
This is verified by observing that the singular behaviors of $\psi^\pm (x,\lambda)$ and its boundary values $\psi_\pm^\pm(x,\lambda)$ are determined by
\begin{equation} \label{eq:dontwantrepeatlabel123} \frac{\rho(\lambda)^{\pm 1} - y_1(T,\lambda)}{y_2(T,\lambda)} = \frac{\Delta(\lambda) \mp \sqrt{\Delta(\lambda)^2-4}}{y_2(T,\lambda)}. \end{equation}
It follows from claim \ref{clm:dirflo} that $\rho(\mu_n)^{\pm 1} - y_1(T,\mu_n) = 0$ because $\rho(\mu_n) = \pm 1$ at $\mu_n$ (in this case, $\mu_n$ lies at a periodic/antiperiodic eigenvalue).
Since the numerator of (\ref{eq:dontwantrepeatlabel123}) involves only holomorphic functions and the square root of a holomorphic function, the zeros of $\Delta(\lambda) \mp \sqrt{\Delta(\lambda)^2-4}$ at $\mu_n$ must have order at least $\tfrac{1}{2}$.
Since $y_2(T,\lambda)$ has a simple zero at $\mu_n$, $\psi_\pm^\pm (x,\lambda)$ must then have at worst square root singularities at $\mu_n$.
The square root singularities can also be verified by combining $\psi^\pm(x,\lambda)$ into a single function on a hyper-elliptic curve with a pole at the branch point corresponding to $\mu_n$, and the pole behavior on the curve manifests itself in the $\lambda$-plane as square root singular behavior.
The latter approach involving the hyper-elliptic curve is discussed in detail in \cite{TrDe13b} for example.
The boundary values of $1/\sqrt[4]{\Delta(\lambda)^2 - 4}$ for $\lambda \in \mathbb{R}$ have a quartic root singularity at $\mu_n$, and the boundary values of $\sqrt{f^0(\lambda)}$ have a square root zero at $\mu_n$; and therefore, $\Phi$ has at worst a quartic root singularity at $\mu_n$.
\end{proof}

\begin{proof}[Proof of condition 3]
By considering the discrepancy in the boundary values of $\sqrt{\Delta(\lambda)^2-4}$ on the branch cuts $\sigma(L)$ we find that the boundary values $\rho_{\pm}$ satisfy
\begin{equation} \label{eq:jumpshifteval} \rho_+ (\lambda) = \rho_- (\lambda)^{-1}, \rho_+(\lambda)^{-1} = \rho_-(\lambda) \text{ for } \lambda \in \sigma(L). \end{equation}
The only difference in (\ref{eq:bloflo}) between $\psi^+$ and $\psi^-$ is the use of $\rho(\lambda)$ versus $\rho^{-1}(\lambda)$, which establishes $\Psi_+(x,\lambda) = \Psi_-(x,\lambda) \sigma_1$ for $\lambda \in \sigma(L)$.
The matrix $\sigma_1 = \begin{pmatrix} 0 & 1 \\ 1 & 0 \end{pmatrix}$ is the first Pauli matrix, which should not be confused with the signature $\sigma_n$.
The jump condition for $\Phi$ is then calculated by conjugating the jump condition for $\Psi$ as follows
\begin{equation} V(\lambda) = \begin{cases} B_-(\lambda)^{-1} e^{-i \sigma_3 \sqrt{\lambda}_- x} \sigma_1 e^{i \sigma_3 \sqrt{\lambda}_+ x} B_+(\lambda) & \lambda \in \sigma(L) \\ B_-(\lambda)^{-1} e^{i \sigma_3 (\sqrt{\lambda}_+ - \sqrt{\lambda}_-) x} B_+(\lambda) & \lambda \in [E_0,\infty) \setminus \sigma(L) \end{cases}. \end{equation}
The $i$ and $(-1)^{k+m(\lambda)}$ appearing in the definition (\ref{eq:jump}) of the jump matrix $V$ appear due to the branch cuts of $\sqrt{f^0(\lambda)}$ and $\sqrt[4]{\Delta(\lambda)^2 - 4}$ in the definition (\ref{eq:deformA}) of $B$.
\end{proof}

\begin{proof}[Proof of condition 4]
Fix $0 < s < \frac{\pi}{8}$.
Plugging the asymptotic descriptions in lemma \ref{lma:BFass} into (\ref{eq:defPsi}) gives the asymptotic description
\begin{equation}
\Psi(x,\lambda) = \begin{pmatrix} 1 + O(\sqrt{\lambda}^{-1}) & 1 + O(\sqrt{\lambda}^{-1}) \\ -i\sqrt{\lambda} + O(1) & i \sqrt{\lambda} + O(1) \end{pmatrix} e^{-i \sigma_3 \sqrt{\lambda} x}
\end{equation}
which is valid as $\lambda \to \infty$ for $\lambda \in \Omega_s$.
The computation
\begin{align} 
\frac{1}{2 i \sqrt{\lambda}}& \begin{pmatrix} i \sqrt{\lambda} & -1 \\ i \sqrt{\lambda} & 1 \end{pmatrix} \Psi(x,\lambda) e^{i \sigma_3 \sqrt{\lambda} x} \\
& = \frac{1}{2 i \sqrt{\lambda}} \begin{pmatrix} i \sqrt{\lambda} & -1 \\ i \sqrt{\lambda} & 1 \end{pmatrix} \begin{pmatrix} 1 + O(\sqrt{\lambda}^{-1}) & 1 + O(\sqrt{\lambda}^{-1}) \\ -i\sqrt{\lambda} + O(1) & i \sqrt{\lambda} + O(1) \end{pmatrix} \\
& = I + O(\sqrt{\lambda}^{-1})
\end{align}
is valid as $\lambda \to \infty$ for $\lambda \in \Omega_s$ implying (\ref{eq:assphi}).
\end{proof}

\begin{proof}[Proof of condition 5]
In terms of the notation introduced in section \ref{sec:PL} we need to show that
\begin{equation} \psi_{ij} \in \mathcal{A}_2 \left(\mathbb{C} \setminus \left( \mathbb{R}^+ \cup \bigcup_{k = 0}^\mathcal{G} \overline{D}_{n_k} \right) \right), \quad \phi_{ij} \in \mathcal{A}_2 \left(\mathbb{C} \setminus \left( \mathbb{R}^+ \cup \bigcup_{k = 0}^\mathcal{G} \overline{D}_{n_k} \right)\right). \end{equation}
However, it is possible to prove the stronger condition
\begin{equation} \label{eq:boundphi} \psi_{ij} \in \mathcal{A}_1 \left(\mathbb{C} \setminus \left( \mathbb{R}^+ \cup \bigcup_{k = 0}^\mathcal{G} \overline{D}_{n_k} \right) \right), \quad \phi_{ij} \in \mathcal{A}_1 \left(\mathbb{C} \setminus \left( \mathbb{R}^+ \cup \bigcup_{k = 0}^\mathcal{G} \overline{D}_{n_k} \right)\right). \end{equation}
We prove the stronger condition (\ref{eq:boundphi}) by applying the propositions from section \ref{sec:PL}.

The entire function $y_i(x,\lambda)$ and $y_i' (x,\lambda)$ of $\lambda$ have order at most $\tfrac{1}{2}$ so
\begin{equation} \label{eq:ybound} y_i(x,\cdot), y_i' (x,\cdot) \in \mathcal{A}_1 (\mathbb{C}) \ , \end{equation}
by proposition \ref{prop:pl1}.  The fact that 
\begin{equation} \label{eq:rhobound} \rho^\pm(\lambda) \in \mathcal{A}_1 (\mathbb{C} \setminus \mathbb{R}^+) \end{equation}
is established by combining the bound
\begin{equation} \rho(\lambda)^\pm \le \frac{|\Delta(\lambda)| + \sqrt{|\Delta(\lambda)^2 - 4|}}{2} \le 1 + |\Delta(\lambda)| \end{equation}
with the fact $\Delta(\lambda)$ is an entire function of order $\tfrac{1}{2}$.  This yields a bound that is valid in the entire plane, but since $\rho$ has jumps across the bands of spectrum, we conclude that $ \rho^\pm(\lambda) \in \mathcal{A}_1 (\mathbb{C} \setminus \mathbb{R}^+)$.

The inclusion
\begin{equation}  \label{eq:y2bound} y_2(T, \lambda)^{-1} \in \mathcal{A}_1 \left(\mathbb{C} \setminus \bigcup_{n = 0}^\infty \overline{D}_n \right) \end{equation}
is established by applying proposition \ref{prop:pl4} to the canonical product representation (\ref{eq:hady2}) of $y_2(T,\lambda)$.
We have intentionally taken this over the union of the sequence of discs $D_n$, rather than the subsequence of discs $D_{n_k}$.
This is because (\ref{eq:y2bound}) would not be satisfied if the union were to be taken over $D_{n_k}$ because $y_2(T,\lambda)^{-1}$ would be singular at $\{\mu_n\}_{n = 1}^\infty \setminus \{\mu_{n_k}\}_{n = 1}^{\mathcal{G}}$.

We establish
\begin{equation} \label{eq:psibound} \psi^\pm(x,\cdot), {\psi^\pm}' (x,\cdot) \in \mathcal{A}_1 \left(  \mathcal{D} \right)\end{equation}
by combining (\ref{eq:ybound}),(\ref{eq:rhobound}) and (\ref{eq:y2bound}) with proposition \ref{prop:pl1}.

The smallest numbers $\alpha_{\pm}$ and $\alpha_0$ so that
\begin{equation} \sum_{\sigma_n = \pm 1} \mu_n^{- \alpha_\pm}, \sum_{\sigma_n = 0} \mu_n^{- \alpha_0} \end{equation}
converge are such that $\alpha_\pm, \alpha_0 \le \tfrac{1}{2}$ so $f^\pm$ and $f^0$ are entire functions with orders at most $\tfrac{1}{2}$ \cite[page 251]{Ti39}.
Therefore
\begin{equation} \label{eq:bbound} f^\pm, f^0 \in \mathcal{A}_1 (\mathbb{C}) \end{equation}
by propositions \ref{prop:pl1} and \ref{prop:pl2}.
The inclusion 
\begin{equation}  \label{eq:flobound} \frac{1}{\Delta(\lambda)^2 - 4} \in \mathcal{A}_1 \left(\mathbb{C} \setminus \bigcup_{n = 0}^\infty \overline{D}_n\right) \end{equation}
is established by applying proposition \ref{prop:pl4} to the canonical product representations (\ref{eq:hadfloper}) and (\ref{eq:hadfloaper}) of $\Delta (\lambda) - 2$ and $\Delta (\lambda) + 2$, and then applying proposition \ref{prop:pl2}.
The inclusions
\begin{equation} \label{eq:PLboundhpm} h^\pm(\lambda) := \frac{f^{\pm} (\lambda)^2 f^0 (\lambda) }{\Delta(\lambda)^2 - 4} \in \mathcal{A}_1 \left(\mathbb{C} \setminus \bigcup_{n = 0}^\infty \overline{D}_n \right) \end{equation}
are established by noting that the above functions all satisfy the same inclusion.
Now notice if $n$ is distinct from the collection of  $n_k$s, then  $h^\pm (\lambda)$ is in fact analytic at $\mu_n$ which coincides with the center of $D_n$.
The functions $h^\pm$ are therefore analytic in such a $D_n$, and a bound of the form $|h^\pm(\lambda)| \le Me^{c|\lambda|}$ on $h^\pm$ guaranteed by the inclusion (\ref{eq:PLboundhpm}) extends into such a $D_n$ by Proposition \ref{prop:pl3}.
This allows us to conclude,
\begin{equation} h^\pm(\lambda)  \in \mathcal{A}_1 \left(\mathbb{C} \setminus \bigcup_{k = 0}^\mathcal{G} \overline{D}_{n_k} \right). \end{equation}
Finally, the desired result (\ref{eq:boundphi}) is established by applying proposition \ref{prop:pl2} to ${\Phi (x,\lambda) = \Psi(x,\lambda) B(\lambda) e^{i \sigma_3 \sqrt{\lambda} x}}$ and using (\ref{eq:psibound}), (\ref{eq:bbound}) and (\ref{eq:flobound}).\end{proof}

\begin{lma}
If $\Phi$ solves Riemann--Hilbert problem \ref{rhp:Phi} then $\det(\Phi) \equiv 1$.
\end{lma}

\begin{proof}
We begin by proving that the particular solution to Riemann--Hilbert problem \ref{rhp:Phi} given by (\ref{eq:defPhi}) satisfied $\det(\Phi) \equiv 1$.
Recalling the fact that the Wronskian $[\psi^-(x,\lambda), \psi^+(x,\lambda)] = \det(\Psi(x,\lambda))$ for any two solutions to a Sturm--Liouville equation is constant in $x$, we calculate 
\begin{equation} \label{eq:detPhi} \det(\Phi(x,\lambda)) = -\frac{y_2(T,\lambda)}{\sqrt{\Delta(\lambda)^2 - 4}} \det \Psi(x,\lambda).\end{equation}
Since 
\begin{equation} \det \Psi(x,\lambda) = [\psi^-(0,\lambda), \psi^+(0,\lambda)] = \frac{\rho(\lambda) - \rho(\lambda)^{-1}}{y_2(T,\lambda)} = -\frac{\sqrt{\Delta(\lambda)^2-4}}{y_2(T,\lambda)}, \end{equation}
we conclude that $\det(\Phi(x,\lambda)) = 1$ for all $x,\lambda$ by (\ref{eq:detPhi}).
Comparison of $\det(\Phi) \equiv 1$ with the asymptotic behavior (\ref{eq:assphi}) implies that
\begin{equation} \label{eq:assdetA} \det\left( \begin{pmatrix}1 & 1 \\ -i\sqrt{\lambda} & i \sqrt{\lambda}  \end{pmatrix} B(\lambda) \right) = 1 + O(\sqrt{\lambda}^{-1}) \end{equation}
as $\lambda \to \infty$ for $\lambda \in \Omega_s$ for some $0 < s < \tfrac{\pi}{8}$.

Let $\tilde \Phi$ be an arbitrary solution to Riemann--Hilbert problem \ref{rhp:Phi} and let $f(x,\lambda) = \det(\tilde \Phi(x,\lambda))$ which is holomorphic in $\mathbb{C} \setminus \mathbb{R}$.
Using $\det(V(x,\lambda)) = 1$ we find that $f$ satisfies
\begin{equation}
f_+(x,\lambda) = \det(\tilde \Phi_+(x,\lambda)) = \det(V(x,\lambda) \tilde \Phi_-(x,\lambda)) = \det(\tilde \Phi_-(x,\lambda)) = f_-(x,\lambda)
\end{equation}
for $\lambda \in \mathbb{R} \setminus \{E_j\}$.
Therefore $f(x,\lambda)$ extends to a holomorphic function of $\lambda \in \mathbb{C} \setminus \{E_j\}$.
Since $\tilde \Phi$ has at worst quartic singularities at $E_j$, $f$ has at worst square root singularities on $E_j$.
However, isolated singularities of a holomorphic function have at least order one, therefore $f$ extends to an entire function.

To allow us to use the Phragm\'{e}n--Lindel\"{o}f theorem we will establish a bound of the form $|f (x, \lambda)| \le M e^{c |\lambda|^2}$ for $\lambda \in \mathbb{C}$.
In terms of the notation introduced in section \ref{sec:PL} we need to show $f \in \mathcal{A}_2(\mathbb{C})$.
From the Phr\'{a}gmen--Lindel\"{o}f bound of Riemann--Hilbert problem \ref{rhp:Phi} we know that
\begin{equation}
\phi_{ij}, \tilde \phi_{ij} \in \mathcal{A}_2 \left(\mathbb{C} \setminus \left( \mathbb{R}^+ \cup \bigcup_{k = 0}^\mathcal{G} \overline{D}_{n_k} \right) \right)
\end{equation}
and therefore by proposition (\ref{prop:pl2}) we know that
\begin{equation}
f \in \mathcal{A}_2 \left(\mathbb{C} \setminus \left( \mathbb{R}^+ \cup \bigcup_{k = 0}^\mathcal{G} \overline{D}_{n_k} \right)\right).
\end{equation}
However, since $f$ is entire we must have
\begin{equation}
f \in \mathcal{A}_2 \left(\mathbb{C} \setminus  \bigcup_{k = 0}^\mathcal{G} \overline{D}_{n_k} \right).
\end{equation}
by continuity.
An application of proposition \ref{prop:pl3} then implies $f \in \mathcal{A}_2 (\mathbb{C})$ as desired.

Comparing (\ref{eq:assphi}) with (\ref{eq:assdetA}) we find the asymptotic description
\begin{equation}
\label{eq:assf} f(x,\lambda) = 1 + O(\sqrt{\lambda}^{-1})
\end{equation}
as $\lambda \to \infty$ for $\lambda \in \Omega_s$.
Since $f(x,\lambda)$ is entire, the asymptotic description (\ref{eq:assf}) implies that $f(x,\lambda)$ is bounded for $\lambda \in \Omega_{s}$.
Let us pick $s'$ such that $0 < s < s'< \pi/8$, then $f(x,\lambda)$ extends continuously to $\overline{\Upsilon}_{-s',s'}$, is bounded on $\partial \Upsilon_{-s',s'}$ and satisfy a bound of the form $f (x, \lambda) \le M e^{c |\lambda|^2}$ for $\lambda \in \Upsilon_{-s',s'}$.
Therefore, $f (x,\lambda)$ is bounded in $\Upsilon_{-s',s'}$ by the Phragm\'{e}n--Lindel\"{o}f theorem \ref{thm:PL}.
Recall that the sectors $\Upsilon_{\alpha,\beta}$ were defined in Theorem \ref{thm:PL}. 
$f(x,\lambda)$ is entire and bounded for $\lambda \in \mathbb{C} = \Omega_{s'} \cup \Upsilon_{-s',s'}$, so $f(x,\lambda)$ is constant by Liouville's theorem.
Since $f(x,\lambda)$ is constant for all $x$, (\ref{eq:assf}) implies $f \equiv 1$.
\end{proof}

\begin{thm} \label{thm:rhpPhi}
Let $x \in \mathbb{R}$ be fixed, then $\tilde \Phi(x,\lambda)$ solves Riemann--Hilbert problem \ref{rhp:Phi} if and only if
\begin{equation} \label{eq:solspacePhi}  \Phi (x,\lambda) = \begin{pmatrix} 1 & 0 \\ \alpha & 1 \end{pmatrix} \tilde \Phi (x, \lambda) \end{equation}
for some $\alpha$ that is constant in $\lambda$.
Moreover,
\begin{equation} \label{eq:solsfromRHP} \psi^-(x,\lambda) = b_{11}(\lambda)^{-1} e^{-i\sqrt{\lambda} x} \phi_{11}(x,\lambda), \quad \psi^+(x,\lambda) = b_{22} (\lambda)^{-1} e^{i\sqrt{\lambda} x} \phi_{12} (x,\lambda)  \end{equation} solve Hill's equation
\begin{equation} - {\psi^\pm}''(x,\lambda) + u(x) \psi^\pm(x,\lambda) = \lambda \psi^\pm(x,\lambda),\end{equation} with potential $u(x)$ given in terms any solution $\Phi(x,\lambda)$ of Riemann--Hilbert problem \ref{rhp:Phi} by
\begin{equation} \label{eq:reconstructionformulapotential} u(x) = 2 i \frac{\partial}{\partial x} \lim_{\lambda \to \infty} \sqrt{\lambda} \left( 1 - b_{11}(\lambda)^{-1} \phi_{11} (x,\lambda)  \right).\end{equation}
\end{thm}

\begin{proof} Fix $x \in \mathbb{R}$ and let $\Phi$ and $\tilde \Phi$ solve Riemann--Hilbert problem \ref{rhp:Phi} and consider the ratio
\begin{equation} R(x, \lambda) = \Phi(x,\lambda) \tilde \Phi(x,\lambda)^{-1}. \end{equation}
Since $\det(\tilde \Phi (x,\lambda)) = 1$, Cramer's rule allows us to write $\tilde \Phi (x,\lambda)^{-1}$ in terms of the entries of $\tilde \Phi(x,\lambda)$ as
\begin{equation} \tilde \Phi (x,\lambda)^{-1} = \begin{pmatrix} \tilde \phi_{22} (x,\lambda) & -\tilde \phi_{12} (x,\lambda) \\ - \tilde \phi_{21} (x,\lambda) & \tilde \phi_{11} (x,\lambda) \end{pmatrix}. \end{equation}
$R(x, \lambda)$ has no jumps, and thus extends to a holomorphic function on $\mathbb{C} \setminus \{E_n\}_{n = 0}^{2 \mathcal{G}}$.
Moreover, $R(x,\lambda)$ has at worst square root singularities at $E_n$, because the entries in $\Phi(x,\lambda)$ and $\tilde \Phi(x,\lambda)$ have at worst quartic root singularities at $E_n$.
However, isolated singularities of a holomorphic function can only be poles or essential singularities.
Therefore $R(x,\lambda)$ is in fact an entire function of $\lambda$.

To allow us to use the Phragm\'{e}n--Lindel\"{o}f theorem we will establish a bound of the form $|r_{ij} (x, \lambda)| \le M e^{c |\lambda|^2}$ for $\lambda \in \mathbb{C}$.
In terms of the notation introduced in section \ref{sec:PL} we need to show $r_{ij} \in \mathcal{A}_2(\mathbb{C})$.
From the Phr\'{a}gmen--Lindel\"{o}f bound of Riemann--Hilbert problem \ref{rhp:Phi} we know that
\begin{equation}
\phi_{ij}, \tilde \phi_{ij} \in \mathcal{A}_2 \left(\mathbb{C} \setminus \left( \mathbb{R}^+ \cup \bigcup_{k = 0}^\mathcal{G} \overline{D}_{n_k} \right) \right),
\end{equation}
and therefore by proposition (\ref{prop:pl2}) we know that
\begin{equation}
r_{ij} \in \mathcal{A}_2 \left(\mathbb{C} \setminus \left( \mathbb{R}^+ \cup \bigcup_{k = 0}^\mathcal{G} \overline{D}_{n_k} \right)\right).
\end{equation}
However, since $R$ is entire we must have
\begin{equation}
r_{ij} \in \mathcal{A}_2 \left(\mathbb{C} \setminus  \bigcup_{k = 0}^\mathcal{G} \overline{D}_{n_k} \right).
\end{equation}
by continuity.
An application of proposition \ref{prop:pl3} then implies $r_{ij} \in \mathcal{A}_2 (\mathbb{C})$ as desired.

The solutions $\Phi(x,\lambda)$ and $\tilde \Phi(x,\lambda)$ have asymptotic descriptions of the form (\ref{eq:assphi}) valid in $\Omega_s$ and $\Omega_{\tilde s}$ respectively for some $0 < s,\tilde s < \tfrac{\pi}{8}$.
Inverting the asymptotic description (\ref{eq:assphi}) for $\tilde \Phi$ gives
\begin{equation} \label{eq:asstphi} \tilde \Phi(x,\lambda)^{-1} = B(\lambda)^{-1} (I + O(\sqrt{\lambda}^{-1})) \frac{1}{2i \sqrt{\lambda}} \begin{pmatrix} i \sqrt{\lambda} & -1 \\ i \sqrt{\lambda} & 1 \end{pmatrix}  \end{equation}
as $\lambda \to \infty$ valid for $\lambda \in \Omega_{\tilde s}$. 
Multiplying (\ref{eq:assphi}) by (\ref{eq:asstphi}) gives the asymptotic desctiption
\begin{align}  R(x,\lambda) & = \begin{pmatrix} 1 & 1 \\ -i \sqrt{\lambda} & i \sqrt{\lambda} \end{pmatrix}   (I + O(\sqrt{\lambda}^{-1})) \frac{1}{2i \sqrt{\lambda}} \begin{pmatrix} i \sqrt{\lambda} & -1 \\ i \sqrt{\lambda} & 1 \end{pmatrix} \\
\label{eq:Rass} & = \begin{pmatrix} 1 + O(\sqrt{\lambda}^{-1}) & O(\sqrt{\lambda}^{-1}) \\ O(1) & 1 + O(\sqrt{\lambda}^{-1}) \end{pmatrix} \end{align}
as $\lambda \to \infty$ valid for $\lambda \in \Omega_{s'}$ where $0 < s' = \min\{s,\tilde s\} < \tfrac{\pi}{8}$.
Since $R(x,\lambda)$ is entire, the asymptotic description (\ref{eq:Rass}) implies that the entries in $R(x,\lambda)$ are bounded for $\lambda \in \Omega_{s'}$.
Let us choose ${s' < s'' < \tfrac{\pi}{8}}$, then the entries of $R(x,\lambda)$ extend continuously to $\overline{\Upsilon}_{-s'',s''}$, are bounded on $\partial \Upsilon_{-s'',s''}$ and satisfy a bound of the form $|r_{ij} (x, \lambda)| \le M e^{c |\lambda|^2}$ for $\lambda \in \Upsilon_{-s'',s''}$.
Therefore, the entries of $R (x,\lambda)$ are bounded in $\Upsilon_{-s'',s''}$ by the Phragm\'{e}n--Lindel\"{o}f theorem \ref{thm:PL}.
Recall that the sectors $\Upsilon_{\alpha,\beta}$ were defined in Theorem \ref{thm:PL}. 
$R(x,\lambda)$ is entire with bounded entries for $\lambda \in \mathbb{C} = \Omega_{s'} \cup \Upsilon_{-s'',s''}$, so $R(x,\lambda)$ is constant by Liouville's theorem.
The asymptotic description (\ref{eq:Rass}) then implies 
\begin{equation} R(x,\lambda) \equiv \begin{pmatrix} 1 & 0 \\ \alpha & 1 \end{pmatrix} \implies \Phi(x,\lambda) = \begin{pmatrix} 1 & 0 \\ \alpha & 1 \end{pmatrix} \tilde \Phi (x,\lambda).  \end{equation}

The the reconstruction of $u$ in Theorem \ref{thm:rhpPhi} from the particular solution given by (\ref{eq:defPhi}) follows from lemma {\ref{lma:BFass}} because $b_{11}(\lambda)^{-1} \phi_{11}(x,\lambda) e^{i \sqrt{\lambda} x} = \psi^-(x,\lambda)$ {and $\psi^-$ has the asymptotic description (\ref{eq:asspsi})}.
The fact that the reconstruction of $u(x)$ does not depend on the choice of solution to Riemann--Hilbert problem \ref{rhp:Phi} follows from the fact that the mappings of the space of $2 \times 2$ matrices to itself given by
\begin{equation} M \to \tilde M = \begin{pmatrix} 1 & 0 \\ \alpha & 1 \end{pmatrix} M \end{equation}
leaves the first row of $M$ invariant while $M \to \tilde M B(\lambda)^{-1}$ only scales the upper left most entry (i.e. $\tilde m_{11} = b_{11}(\lambda)^{-1} m_{11}$).
The fact that $\psi^-(x,\lambda)$ and $\psi^+(x,\lambda)$ defined by (\ref{eq:solsfromRHP}) solve Hill's equation can be verified by observing how the first row of the particular solution $\Phi$ to Riemann--Hilbert problem \ref{rhp:Phi} defined by (\ref{eq:defPhi}) is built in terms of the Bloch--Floquet solutions. \end{proof}

 \begin{rmk} \label{rmk:1}
To define $\Phi$ using $B$ and $\Psi$ as in (\ref{eq:defPhi}), as well as to define $\breve \Phi$ by (\ref{eq:defbPhi}) in the next section, we only need to assume: (1) $E_{0} < E_{1} < \cdots $; $R = \max_{j} |E_{2j}-E_{2j-1}| < \infty$; (2) if $\mathcal{G}$ is infinite then there exists $N, C$ such that  $E_{n} > C n^{2}$ for all $n > N$;  and (3) if $\mathcal{G}$ is infinite then there exists $K$ such that the discs $D_{n_k}$ of radius $R$ centered at $(E_{2k}+E_{2k-1})/2$ are disjoint for $k \ge K$.
Extra regularity assumptions on $u$ add additional constraints on the decay rate of $|E_{2j}-E_{2j-1}|$ as $j \to \infty$.
\end{rmk}

\begin{rmk}
\label{rmk:quasiperiod}
Consider Riemann--Hilbert problem \ref{rhp:Phi}, however instead of assuming the spectral data is determined by a periodic potential we now assume the spectral data is arbitrarily chosen to consist of, a sequence $\{E_{k}\}_{k=0}^{2 \mathcal{G}+1}$, a sequence $\{\mu_{n_k} \}_{k= 1}^{\mathcal{G}}$ and a sequence $\{\sigma_{n_k}\}_{k=1}^{\mathcal{G}}$ ($\mathcal{G}$ could be finite or infinite), satisfying
\begin{enumerate}
\item $E_{0} < E_{1} < \cdots $,
\item $\mu_{n_k} \in [E_{2k-1},E_{2k}]$,
\item $R = \max_{j} |E_{2j}-E_{2j-1}| < \infty$, 
\item for each $k$, $\sigma_{n_k} \in \{ -1,0,1\}$.
\item In case $\mathcal{G}$ is infinite:  there exists $N, C$ such that  $E_{n} > C n^{2}$ for all $n > N$.
\item In case $\mathcal{G}$ is infinite: there exists $K$ such that the discs $D_{n_k}$ of radius $R$ centered at ${(E_{2k}+E_{2k-1})/2}$ are disjoint for $k \ge K$.
\end{enumerate}
Since in this case we do not know that Riemann-Hilbert problem \ref{rhp:Phi} comes from a periodic potential, we do not know a-priori that a solution exists.  However, if a solution to Riemann--Hilbert problem \ref{rhp:Phi} does exists, then the first row of the solution is unique.
Moreover, if we also know that the solution is differentiable with respect to $x$ (with asymptotic expansion as $\lambda \to \infty$ obtained by differentiating the expansion of the solution with respect to $x$), then $\psi^-(x,\lambda) = b_{11}(\lambda)^{-1} e^{-i\sqrt{\lambda} x} \phi_{11}(x,\lambda) $ and ${\psi^+(x,\lambda) = b_{22} (\lambda)^{-1} e^{i\sqrt{\lambda} x} \phi_{12}(x,\lambda)}$ solve Hill's equation
\begin{equation} - {\psi^\pm}''(x,\lambda) + u(x) \psi^\pm(x,\lambda) = \lambda \psi^\pm(x,\lambda),\end{equation} with potential $u(x)$ given in terms the first row of any solution $\Phi(x,\lambda)$ of Riemann--Hilbert problem \ref{rhp:Phi} by
\begin{equation} \label{eq:reconstructionformulapotential} u(x) = 2 i \frac{\partial}{\partial x} \lim_{\lambda \to \infty} \sqrt{\lambda} \left( 1 -   b_{11}(\lambda)^{-1}\phi_{11} (x,\lambda) \right).\end{equation}
The potential $u$ so determined is most likely not periodic. By introducing time dependence in the Riemann-Hilbert problem (as discussed in Section 6), we therefore obtain a characterization of solutions of the KdV equation that are outside the class of periodic solutions.
\end{rmk}

\begin{rmk} If instead of considering a Riemann--Hilbert problem for $\Phi$ defined by (\ref{eq:defPhi}) we considered a Riemann--Hilbert problem for the function
\begin{equation} \Xi(x,\lambda) = \Psi(x,\lambda) B(\lambda) \rho(\lambda)^{\frac{x}{T} \sigma_3} \end{equation}
then the only change to Riemann--Hilbert problem \ref{rhp:Phi} would be a new jump matrix
\begin{equation} \tilde V (\lambda) := \begin{cases} (-1)^{k + m(\lambda) - 1} \begin{pmatrix} 0 & i \frac{f^+(\lambda)}{f^-(\lambda)} \\ i \frac{f^-(\lambda)}{f^+(\lambda)} & 0 \end{pmatrix} & \lambda \in (E_{2k-2},E_{2k-1})  \\ (-1)^{k + m(\lambda)} e^{\frac{2 i n_k \pi}{T} \sigma_3 x}& \lambda \in (E_{2k-1},E_{2k}) \end{cases}. \end{equation}

Analogues of the existence and uniqueness results of this section for $\Xi$ can be proven by minor modification to the proofs.
Advantages of considering the Riemann--Hilbert problem for $\Xi$ include: (1) the jump matrices are constant (but different) on each gap, (2) if all of the Dirichlet eigenvalues lie on the band edge, then $\tilde V$ is piecewise constant for all $\lambda \in [0,\infty)$, and (3) $\Xi$ has the same asymptotic behavior as $\Phi$ for $\lambda \to \infty$ with $\lambda \in \Omega_{s}$.  However, the definition of $\tilde V$ requires knowledge of degenerated gaps (to determine the $n_{k}$s), while the definition of $V$ does not.
\end{rmk}

We will now provide a little more detail on the derivation of the jump matrix for $\Xi(x,\lambda)$.
We consider the quantity $\log(\rho(\lambda))$, and provide a precise definition as follows.  
Since $\rho(\lambda)$ is nonzero in the simply connected domain $\mathbb{C} \setminus \mathbb{R}^+$ and $\rho(\lambda) \to 1$ as $\lambda \to 0$ in $\mathbb{C} \setminus \mathbb{R}^+$, we can define $\log(\rho(\lambda))$ by 
\begin{equation} \label{eq:logrhointegraldef} \log(\rho(\lambda)) = \int_0^\lambda \frac{\rho_\lambda(\lambda)}{\rho(\lambda)}  d \lambda \end{equation}
where the contour of integration from $0$ to $\lambda$ is contained in $\mathbb{C} \setminus \mathbb{R}^+$ (except of course at the endpoint $0$).
The boundary values $\log(\rho(\lambda))_\pm$ are obtained by computing the integral (\ref{eq:logrhointegraldef}) along either the top or bottom of $\mathbb{R}^+$.

The logarithmic derivative of $\rho$ is computed by differentiating $\rho(\lambda)^2 - \Delta(\lambda) \rho(\lambda) + 1 = 0$ by $\lambda$ to get
\begin{equation} \label{eq:logderrhoform} \rho_{\lambda} (\lambda) (2 \rho(\lambda) - \Delta(\lambda)) = -\Delta_{\lambda} (\lambda) \rho(\lambda) \implies \frac{\rho_\lambda (\lambda)}{\rho(\lambda)} = -\frac{\Delta_\lambda(\lambda)}{\rho(\lambda) - \rho(\lambda)^{-1}} = \frac{\Delta_\lambda(\lambda)}{\sqrt{\Delta(\lambda)^2-4}} \end{equation}
where we have used the fact that $\Delta(\lambda) = \rho(\lambda) + \rho(\lambda)^{-1}$ and $\sqrt{\Delta(\lambda)^2-4} = \rho(\lambda)^{-1} - \rho(\lambda)$.
It is shown in \cite{MaWi79} that $\Delta_\lambda(\lambda) < 0$ for $\lambda \in (\lambda_{4n},\lambda_{4n+1})$ and $\Delta_\lambda(\lambda)>0$ for $\lambda \in (\lambda_{4n+2},\lambda_{4n+3})$, $n = 0,1,2,\dots$. From (\ref{eq:expansionsqrtflo}) we see that $(\Delta_\lambda(\lambda)/\sqrt{\Delta(\lambda)^2-4})_+$ is purely imaginary with positive imaginary part in interior of the spectrum, $(\Delta_\lambda(\lambda)/\sqrt{\Delta(\lambda)^2-4})_-$ is purely imaginary with negative imaginary part in interior of the spectrum, and $(\Delta_\lambda(\lambda)/\sqrt{\Delta(\lambda)^2-4})_\pm$ are real and agree in the spectral gaps.
Using the fact that $\rho(\lambda) = \pm 1$ if and only if $\lambda$ is a periodic/antiperiodic eigenvalue, we see that $\text{Im} \log(\rho(\lambda))_+ = i n \pi$ and $\text{Im} \log(\rho(\lambda))_- = -i n \pi$ for $\lambda \in [\lambda_{2n-1},\lambda_{2n}]$.
Combining the fact the $|\rho_\pm(\lambda)|=1$ in the spectrum, with the reality and single valuedness of the logarithmic derivative of $\rho$ in the gaps, tells us $\text{Re}\log(\rho(\lambda))_+ = \text{Re}\log(\rho(\lambda))_-$ for $\lambda \in \mathbb{R}^+$ and they both vanish for $\lambda \in \sigma(L)$.
These facts give us the control on $\log(\rho(\lambda))$ necessary to establish the jump of $\Xi$ on the gaps.
The $(-1)^{k + m(\lambda)}$ appears in this jump due to $B(\lambda)$.

From the integration formula (\ref{eq:logrhointegraldef}), it is also clear that $\log(\rho(\lambda))_+ + \log(\rho(\lambda))_- = 0$ for $\lambda \in \sigma(L)$.
Since the jump of $B(\lambda) \Psi(x,\lambda)$ is off diagonal for $\lambda \in \sigma(L)$, $\rho_+(\lambda)$ and $\rho_-(\lambda)$ appear in the jump matrix for $\Xi(x,\lambda)$ for $\lambda \in \sigma(L)$ in terms of the combination ${\rho_+(\lambda)^{\frac{x}{T}} \rho_-(\lambda)^{\frac{x}{T}} = e^{(\log(\rho(\lambda))_+ + \log(\rho(\lambda))_-)\frac{x}{T}} = 1}$.
For this reason, boundary values of $\rho_\pm(\lambda)$ do not appear in the jump matrix for $\Xi(x,\lambda)$ for $\lambda \in \sigma(L)$.
The jump of $\Xi(x,\lambda)$ for $\lambda \in \sigma(L)$ therefore matches the jump for $\Phi$.

\section{Cauchy Problem for the Periodic KdV Equation}

To make computations in this section simpler we will no longer work with potentials $u$ in $L^\infty(\mathbb{R})$. Instead we will assume enough regularity to take all the derivatives required for the following arguments.
The KdV equation
\begin{equation}
u_t - 6 u u_x + u_{xxx} = 0
\end{equation}
is equivalent to the Lax equation { \cite{Lax}}
\begin{equation} \label{eq:Lax}
L_t = [A,L]
\end{equation}
where
\begin{equation}
L = -\frac{\partial^2}{\partial x^2} + u(x,t)
\end{equation}
and
\begin{equation}
A = -4 \frac{\partial^3}{\partial x^3} + 6 u(x,t) \frac{\partial}{\partial x} + 3 u_x(x,t).
\end{equation}
We assume the the initial condition $u_0(x) = u(x,0)$ is smooth and periodic, $u_0(x+T) = u_0(x)$.
We can use well known results for the KdV to guarantee existence and uniqueness of smooth periodic solutions to the KdV equation for all time for any smooth periodic initial condition \cite{MckTr76, MckTr78}.
It follows that $u(\cdot,t)$ is smooth and periodic as well, with $u(x+T,t) = u(x,t)$ for all times $t$.
Moreover, the spectrum and periodic/antiperiodic spectra of $L$ are preserved under the flow.
In particular, the Floquet discriminant and thus Floquet multipliers are constant for all time.

However, the Dirichlet eigenvalues are not constant, so if we consider Riemann--Hilbert problem \ref{rhp:Phi} for the potentials $u(\cdot,t_1)$ and $u(\cdot,t_2)$ the band ends $E_j$ will stay the same, but the  Dirichlet eigenvalues could vary.
Therefore $\mu_{n_k}$ and $\sigma_{n_k}$ pick up a $t$ dependence, and we write $\mu_{n_k}(t)$ and $\sigma_{n_k}(t)$ for the Dirichlet divisor and signature for the potential $u(t)$ at time $t$.
The functions $f^0(\lambda)$, $f^\pm(\lambda)$ and $B(\lambda)$ are as defined in definitions \ref{defn:regularize}, and we will write $f^0(t,\lambda)$, $f^\pm(t,\lambda)$ and $B(t,\lambda)$.
In particular the function $B(t,\lambda)$ appearing in the asymptotic description of $\Phi$ will evolve in time.

While the quantities $\psi^{\pm}(x, t, \lambda)$ do not simultaneously solve both differential equations in the Lax pair, in this section we use the Lax pair to deform them in such a way that the new quantities, $\breve{\psi}^{\pm}(x,t,\lambda)$ do solve both differential equations.  Moreover, the poles of $\breve{\psi}^{\pm}(x,t,\lambda)$ are also independent of time.  And, finally, these are explicitly related to Riemann--Hilbert problem \ref{rhp:KdV}, in which the dependence on $x$ and $t$ is completely explicit.   In other words, in this section we establish a Riemann-Hilbert characterization of the KdV equation for smooth periodic infinite gap initial conditions.  The processes described in this paragraph was previously established for periodic finite gap initial conditions using the theory of Riemann Surfaces in \cite{Dubrovin75,ItsMatv75a,ItsMatv75b,Noetal84}.

\begin{prop} The time dependent Dirichlet eigenvalues evolve according to Dubrovin's equation
\begin{align}\mu_{nt}(t) &= -(4 \mu_n(t) + 2 u(0,t)) \frac{y_{2x}(T,t,\mu_n(t)) - y_1(T,t,\mu_n(t))}{y_{2\lambda}(\mu_n(t))} \\ \label{eq:xdubrovin} &=-\sigma_n(t)(4 \mu_n(t) + 2 u(0,t)) \frac{\rho(\mu_n(t)) - \rho(\mu_n(t))^{-1}}{y_{2\lambda}(\mu_n(t))}.\end{align} \end{prop}

This equation was derived in the finite gap case by Dubrovin in \cite{Dubrovin75}.
The proof given here, which also applies to the infinite gap case, was originally given by Its and Matveev and appears in Russian in \cite{ItsMatv75b}.

\begin{proof}

The equation
\begin{equation}
L_t y_2(x,t,\lambda) + L y_{2t}(x,t,\lambda) = \lambda y_{2t}(x,t,\lambda).
\end{equation}
is derived by differentiating both sides of $L y_2(x,t,\lambda) = \lambda y_2(x,t,\lambda)$ by $t$.
Combining the Lax pair (\ref{eq:Lax}) equation and Hill's equation implies
\begin{equation} L_t y_2(x,t,\lambda) = A L y_2(x,t,\lambda) - L A y_2(x,t,\lambda) = \lambda A y_2(x,t,\lambda) - L A y_2(x,t,\lambda). \end{equation}
Therefore, $y_{2t}(x,t,\lambda) - A y_2(x,t,\lambda)$ solves Hill's equation
\begin{equation} L(y_{2t}(x,t,\lambda) - A y_2(x,t,\lambda)) = \lambda (y_{2t}(x,t,\lambda) - A y_2(x,t,\lambda)),  \end{equation}
so for fixed $t$ and $\lambda$ there exist constants $c_1,c_2$ such that
\begin{equation}y_{2t}(x,t,\lambda) - A y_2(x,t,\lambda) = c_1 y_1(x,t,\lambda) + c_2 y_2(x,t,\lambda).  \end{equation}
From the normalizations of $y_1$ and $y_2$ at $x = 0$ we find that $y_{2t}(0,t,\lambda) = y_{2xt}(0,t,\lambda) = 0$ and
\begin{equation} c_1 = - A y_2(0,t,\lambda), c_2 = - \frac{\partial}{\partial x} A y_2(0,t,\lambda). \end{equation}
From Hill's equation it follows that
\begin{equation} -y_{2xxx} (x,t,\lambda) = \lambda y_{2x}(x,t,\lambda) - u(x,t) y_{2x}(x,t,\lambda) - u_x(x,t) y_2(x,t,\lambda)   \end{equation}
and therefore
\begin{equation} A y_2(x,t,\lambda) = 4 \lambda y_{2x}(x,t,\lambda) + 2 u(x,t) y_{2x} (x,t,\lambda) - u_x(x,t,\lambda) y_2 (x,t,\lambda)  \end{equation}
and
\begin{align*} \numberthis \frac{\partial}{\partial x} A y_2(x,t,\lambda) =&  4 \lambda y_{2xx}(x,t,\lambda) + 2 u(x,t) y_{2xx} (x,t,\lambda) \\ & - u_{xx}(x,t) y_2(x,t,\lambda) + u_x(x,t) y_{2x}(x,t,\lambda). \end{align*}
Evaluation at $x = 0$ gives $c_1 = -(4 \lambda + 2 u(0,t))$ and $c_2 = - u_x(0,t)$.
We have therefore derived the following formula for $y_{2t}(x,t,\lambda)$
\begin{align} y_{2t}(x,t,\lambda) & = -(4 \lambda + 2 u(0,t))y_1(x,t,\lambda) + (A - u_x(0,t)) y_2(x,t,\lambda) \\
& = -(4 \lambda + 2 u(0,t))y_1(x,t,\lambda) + (4 \lambda + 2 u(x,t,\lambda)) y_{2x}(x,t,\lambda) \\
& \ \ \ \ \ \ \ \ \ -(u_x(x,t) + u_x(0,t)) y_2(x,t,\lambda).  \end{align}
Evaluating the above at $x = T$ and $\lambda = \mu_n(t)$ gives
\begin{equation} y_{2t}(T,t,\mu_n(t)) = (4 \mu_n(t) + 2 u(0,t))(y_{2x}(T,t,\mu_n(t)) - y_1(T,t,\mu_n(t))).  \end{equation}
From the expansion
\begin{equation} y_2(T,t,\lambda) = \prod_{m = 1}^\infty \frac{T^2}{m^2 \pi^2} (\mu_m(t) - \lambda) \end{equation}
we find on one hand that
\begin{equation} y_{2t}(T,t,\lambda) = \sum_{n = 1}^\infty \frac{T^2}{n^2 \pi^2}  \mu_{nt}(t) \prod_{m \ne n} \frac{T^2}{m^2 \pi^2} (\mu_m(t) - \lambda)   \end{equation}
and on the other that
\begin{equation} y_{2 \lambda} (T,t,\lambda) = -\sum_{n = 1}^\infty \frac{T^2}{n^2 \pi^2} \prod_{m \ne n} \frac{T^2}{m^2 \pi^2} (\mu_m(t) - \lambda).  \end{equation}
Evaluating the two previous expressions at $\lambda = \mu_n(t)$ gives $y_{2t}(T,t,\mu_n(t)) = -\mu_{nt}(t) y_{2 \lambda}(T,t,\mu_n)$.
\end{proof}

\begin{prop} \label{eq:proplaxsystem}
The time dependent Bloch--Floquet solutions $\psi^\pm(x,t,\lambda)$ solve
\begin{equation}\label{eqtoeval0}
\psi^\pm_t(x,t,\lambda) + \alpha^\pm(t,\lambda) \psi^\pm(x,t,\lambda) = A \psi^\pm(x,t,\lambda), 
\end{equation}
where the functions $\alpha^\pm(t,\lambda)$ are given by
\begin{equation}
\alpha^\pm(t,\lambda) = (4 \lambda + 2u(0,t)) \frac{\rho(\lambda)^{\pm 1} - y_1(T,t,\lambda)}{y_2(T,t,\lambda)} - u_x(0,t),
\end{equation}
and for each $s > 0$, as $\lambda \to \infty$, $\lambda \in \Omega_{s}$, we have
\begin{equation} \label{eq:assalpha} \alpha^\pm(t,\lambda) = \pm 4 i \sqrt{\lambda}^3 + O(\sqrt{\lambda}^{-1}) \ . \end{equation}
Moreover, if $\sigma_n(t) = 1$ then $\alpha^+(t,\lambda)$ has a simple pole at $\mu_n(t)$ with residue $-\mu_{nt}(t)$ and if $\sigma_n(t) = -1$ then $\alpha^-(t,\lambda)$ has a simple pole at $\mu_n(t)$ with residue $-\mu_{nt}(t)$. 

\end{prop}


\begin{proof}

This proof proceeds in a manner similar to the proof of the previous proposition.
For each time $t$ we can produce normalized Bloch--Floquet solutions $\psi^+(x,t,\lambda)$ and $\psi^-(x,t,\lambda)$ that factor as
\begin{equation} \psi^+(x,t,\lambda) = p^+(x,t,\lambda) \rho(\lambda)^{x T^{-1}}, \psi^-(x,t,\lambda) = p^-(x,t,\lambda) \rho(\lambda)^{-x T^{-1}}. \end{equation}
Differentiating $(L - \lambda) \psi^\pm(x,t) = 0$ by $t$, and using the Lax equation and Hill's equation implies 
\begin{equation}
L_t \psi^+ + L \psi_t^+ - \lambda \psi_t^+ = [A,L] \psi^+ + L \psi_t^+ - \lambda \psi^+_t = (L - \lambda)(\psi_t^+ - A \psi^+) = 0.
\end{equation}
Therefore $\psi_t^+ - A \psi^+$ solves Hill's equation for all time so
\begin{equation} \label{eq:1234567}
p_t^+ \rho(\lambda)^{x T^{-1}} = A (p^+ \rho(\lambda)^{x T^{-1}}) - \alpha^+(t,\lambda) p^+ \rho(\lambda)^{x T^{-1}} - \beta^+(t,\lambda) p^- \rho(\lambda)^{-x T^{-1}}
\end{equation}
for some $x$ independent functions $\alpha^+(t,\lambda)$ and $\beta^+(t,\lambda)$.
For $\lambda \in \mathbb{C}\setminus \sigma(L)$ every term in (\ref{eq:1234567}) decays exponentially as $x \to \infty$ except $\beta^+(t,\lambda) \psi^-\rho(\lambda)^{-x T^{-1}}$, so $\beta^+(\lambda) = 0$.
Therefore, $\psi^+$ solves
\begin{equation}
\psi^+_t + \alpha^+(t,\lambda) \psi^+ = A \psi^+ 
\end{equation}
for some $\alpha^+(t,\lambda)$.
An analogous argument implies $\psi^-$ solves
\begin{equation}
\psi^-_t + \alpha^-(t,\lambda) \psi^- = A \psi^- 
\end{equation}
for some $\alpha^-(t,\lambda)$.

The functions $A \psi^\pm$ are given by
\begin{equation}  A \psi^\pm (x,t,\lambda) = (4 \lambda + 2 u(x,t)) \psi_{x}^\pm (x,t,\lambda) - u_x(x,t,\lambda) \psi^\pm (x,t,\lambda).  \end{equation}
Since $\psi^\pm(0,t,\lambda) = 1$ for all $t$, $\psi_t^\pm(0,t,\lambda) = 0$.
Also, recall that
\begin{equation} \psi^\pm_x(0,t,\lambda) = \frac{\rho(\lambda)^{\pm1} - y_1(T,t,\lambda)}{y_2(T,t,\lambda)}. \end{equation}
Therefore, evaluation of (\ref{eqtoeval0}) at $x = 0$ gives the following formula
\begin{equation}
\alpha^\pm(t,\lambda) = (4 \lambda + 2u(0,t)) \frac{\rho(\lambda)^{\pm 1} - y_1(T,t,\lambda)}{y_2(T,t,\lambda)} - u_x(0,t).
\end{equation}
Plugging in the asymptotic description from (\ref{eq:propx02}) into the preceding formula immediately gives (\ref{eq:assalpha}).
\end{proof}

Let $e^\pm (t,\lambda)$ be solutions to
\begin{equation} e_t^\pm (t,\lambda)  = \alpha^\pm(t,\lambda) e^\pm(t,\lambda) \end{equation}
with $e^\pm(0,\lambda) = 1$,
or equivalently let $e^\pm(t,\lambda)$ be given by
\begin{equation} e^\pm(t,\lambda) = \exp \left( \int_0^t \alpha^\pm(\tau,\lambda) d \tau \right).\end{equation}
The solutions
\begin{equation} \label{eq:psitobrevepsi} \breve \psi^\pm(x,t,\lambda) = \psi^\pm(x,t,\lambda) e^\pm(t,\lambda) \end{equation}
to Hill's equation satisfy the system
\begin{equation} \label{eqcompat}
\breve \psi_t^+ = A \breve \psi^+, \breve \psi_t^- = A \breve \psi^-,
\end{equation}
for which (\ref{eq:Lax}) is the compatibility condition.
The idea here is that $e^\pm(t,\lambda)$ should cancel the time dependent singularities of $\psi^{\pm}(x,t,\lambda)$ and preserve the initial singularities for all $t$.

\begin{prop} \label{proper:e}
The functions $e^\pm(t,\lambda)$ satisfy the following properties.
\begin{enumerate}
\item $e^\pm(t,\lambda)$ are meromorphic functions in $\mathbb{C} \setminus \sigma(L)$.
\item The boundary values of $e^\pm$ satisfy $e_+^\pm(t,E) = e_-^\mp(t,E)$ for $E \in \sigma(L)$.
\item $e^+(t,\lambda)$ has simple poles on $\mu_{n_k}(0)$ when $\sigma_{n_k}(0) = 1$ and simple zeros on $\mu_{n_k}(t) = 1$ when $\sigma_{n_k}(t) = 1$. $e^-(t,\lambda)$ has simple poles on $\mu_{n_k}(0)$ when $\sigma_{n_k}(0) = -1$ and simple zeros on $\mu_{n_k}(t)$ when $\sigma_{n_k}(t) = -1$. $e^\pm(t,\lambda)$ both have square root singularities at $\mu_{n_k}(0)$ when $\sigma_{n_k}(0) = 0$ and square root zeros at $\mu_{n_k}(t)$ when $\sigma_{n_k}(t) = 0$.
\item For fixed $t$, $e^\pm$ has asymptotic descriptions $e^\pm(t,\lambda) = e^{\pm4 i \sqrt{\lambda}^3 t}(1 + O(\sqrt{\lambda}^{-1}))$ as $\lambda \to \infty$ with $\lambda \in \Omega_s$ for some $0 < s < \tfrac{\pi}{8}$.
\item There exist positive constants $c$, and $M$ such that $|e^\pm (t,\lambda)| \le M e^{c|\lambda|^2} $ for all $\lambda \in \mathcal{D}$.
\end{enumerate}
\end{prop}

\begin{proof}
We begin by noting that $e^{\pm}$ are holomorphic for $\lambda \in \mathbb{C} \setminus \mathbb{R}^+$ because $\alpha^{\pm}$ are holomorphic for $\lambda \in \mathbb{C}\setminus \mathbb{R}^+$ for all $t$.
We will see that $e^{\pm}$ extend to meromorphic functions for $\lambda \in \mathbb{C} \setminus \sigma(L)$ in proving property 3.
Property 2 follows from considering the boundary behavior of $\rho(\lambda)$.

To establish property 3, we begin by supposing without loss of generality that at $t = 0$, ${\mu_n \in (\lambda_{2n-1},\lambda_{2n})}$ and $\sigma_n = 1$.
Then, for sufficiently small time, $\mu_n(t) \in (\lambda_{2n-1},\lambda_{2n})$, and for $\lambda$ near $\mu_n(t)$ we may write using proposition \ref{eq:proplaxsystem}
\begin{equation} e^+(t,\lambda) = \exp\left(-\int_0^t \frac{\mu_{n\tau}(\tau) d\tau}{\lambda - \mu_n(\tau)}\right) e^+_{reg}(t,\lambda)\end{equation}
where $e^+_{reg}(t,\lambda)$ is holomorphic and nonzero in some open set containing the image of $\mu_n(t)$ for $t$ is some small time window. We can change variables from $t$ to $\mu = \mu_n(t)$ in the above integral to get
\begin{equation}\label{eq:230}-\int_0^t \frac{\mu_{n\tau}(\tau) d\tau}{\lambda - \mu_n(\tau)} = \int_{\mu_n(0)}^{\mu_n(t)} \frac{d \mu}{\mu-\lambda} = \log(\mu_n(t)-\lambda) - \log(\mu_n(0)-\lambda)\end{equation}
and therefore
\begin{equation} e^+(t,\lambda) = \frac{\lambda-\mu_n(t)}{\lambda-\mu_n(0)} e^+_{reg}(t,\lambda). \label{eq:204forrevision}\end{equation}
An analogous argument applies to $e^-$ when $\sigma_n = -1$.

Note that in (\ref{eq:230}), we used $\mu_{n}(t)$ as a variable change in the integral, which assumes that $\frac{\partial}{\partial t} \mu_{n}$ is nonzero.  In case this derivative vanishes (possible only at isolated values of $t$), we break integral into two pieces, and the formula (\ref{eq:230}) still holds.

This proves the $\sigma_{n_k} = \pm 1$ cases of the property $3$.
This calculation was local, and applies so long as $\mu_n(t)$ remains in the open interval $(\lambda_{2n-1},\lambda_{2n})$.
However, the Dirichlet eigenvalues $\mu_n(t)$ can be equal to the endpoints $\lambda_{2n-1}$ or $\lambda_{2n}$ of the gaps.
In order to understand what happens in this situation we will apply a time translation so that $\mu_n(0) = \lambda_{2n}$.

We introduce a local coordinate that takes into account the branch cut.
Without loss of generality, we assume that $\mu_n(t)$ is approaching $\lambda_{2n}$, and that $\mu_n(t)$ is a pole of $\alpha^+$.
The local coordinate $w$ is defined using two copies of the complex $\lambda$ plane (we only concern ourselves with a neighborhood of $\lambda_{2n}$ in each plane).
We define
\begin{equation} w(\lambda) = \begin{cases}\rho(\lambda) - \rho(\lambda)^{-1} & \lambda \in \text{sheet 1}, \lambda \text{ near } \lambda_{2n} \\ \rho(\lambda)^{-1} - \rho(\lambda) & \lambda \in \text{sheet 2}, \lambda \text{ near } \lambda_{2n} \end{cases}. \end{equation}
This is an analytic invertible transformation from a neighborhood of $0$ in the $w$ plane to the (local) 2-sheeted surface defined by gluing together the two sheets in the usual way.

In the $w$ plane, then the single function
\begin{equation} \alpha (t,w) = \begin{cases} \alpha^+(t,\lambda(w)), & \lambda(w) \in \text{sheet 1} \\ \alpha^-(t,\lambda(w)), & \lambda(w) \in \text{sheet 2}, \end{cases} \end{equation}
is a meromorphic function in a neighborhood of $w = 0$, with a single simple pole at $w(\mu_n)$.
Then the function
\begin{equation} e (t,w) = \begin{cases} e^+(t,\lambda(w)), & \lambda(w) \in \text{sheet 1} \\ e^-(t,\lambda(w)), & \lambda(w) \in \text{sheet 2}, \end{cases} \end{equation}
is related to $\alpha$ by
\begin{equation} e(t,w) = \exp\left( \int_0^t \alpha(\tau,w) d \tau \right). \end{equation}

The derivatives $w_\lambda(\lambda)$ can be computed as
\begin{equation} w_\lambda(\lambda)|_{\lambda \in \text{sheet 1}} = \frac{\rho_\lambda(\lambda)}{\rho(\lambda)} (\rho(\lambda) + \rho(\lambda)^{-1}) = \frac{\rho_\lambda(\lambda)}{\rho(\lambda)} \Delta(\lambda), \quad  w_\lambda(\lambda)|_{\lambda \in \text{sheet 2}} = -\frac{\rho_\lambda(\lambda)}{\rho(\lambda)} \Delta(\lambda) \end{equation}
Therefore,
\begin{equation} \label{eq:derw1form} \frac{d}{dt} w(\mu_n(t)) = \left. \sigma_n(t) \frac{\rho_\lambda(\lambda)}{\rho(\lambda)} \Delta(\lambda) \right|_{\lambda = \lambda(w(\mu_n(t)))} \mu_{nt}(t).\end{equation}
Using the the computation of the logarithmic derivative of $\rho$ (\ref{eq:logderrhoform}) and the $t$ derivative of $\mu_n$ (\ref{eq:xdubrovin}) in (\ref{eq:derw1form}) gives
\begin{equation} \frac{d}{dt} w(\mu_n(t)) = \left. \frac{\Delta(\lambda) \Delta_{\lambda} (\lambda)}{y_{2\lambda}(T,t,\lambda)} \right|_{\lambda = \lambda(w(\mu_n(t)))}(4 \lambda(w(\mu_n(t))) + 2 u(0,t)).\end{equation}
A straightforward calculation shows that
\begin{equation} \alpha(w) = \frac{-\frac{d}{dt} w(\mu_n(t))}{w - w(\mu_n(t))} + \text{reg near }\lambda_{2n}. \end{equation}
Now we compute the local behavior of $e(w)$ for $w$ near 0 as
\begin{equation} e(w) = \exp\left(-\int_0^t \frac{\frac{d}{dt'} w(\mu_n(t'))}{w - w(\mu_n(t'))} dt'\right) e_{reg}(w) = \frac{w - w(\mu_n(t))}{w - w(\mu_n(0))} e_{reg}(w),\end{equation}
where $e_{reg}$ is analytic and nonzero near $w = 0$.  (Again note that isolated values of $t$ where $\frac{d}{dt'} w(\mu_n(t'))$ vanishes are handled by splitting up the integral.)  This completes the proof of property 3.

To prove 4 we begin by recalling that as  $\lambda \to \infty$ for $\lambda \in \Omega_s$ the functions $\alpha^\pm(t,\lambda)/(\pm 4 i \sqrt{\lambda}^3)$ are continuous functions of $\lambda$ converging to 1.
The dominated convergence theorem gives
\begin{equation} \int_0^t \frac{\alpha(\tau,\lambda)}{4 i \sqrt{\lambda}^3} d \tau = \int_0^t \frac{\pm 4 i \sqrt{\lambda}^3 + O(\sqrt{\lambda}^{-1})}{4 i \sqrt{\lambda}^3} d \tau =   \frac{\pm 4 i \sqrt{\lambda}^3 t + O(\sqrt{\lambda}^{-1})}{4 i \sqrt{\lambda}^3}\end{equation}
and thus
\begin{equation} \int_0^t \alpha(\tau,\lambda) d \tau = \pm 4 i \sqrt{\lambda}^3 t + O(\sqrt{\lambda}^{-1})\end{equation}
as $\lambda \to \infty$ for $\lambda \in \Omega_s$.
Therefore $\breve \psi$ has the asymptotic description
\begin{equation} \label{eq:assbpsi} \breve \psi^\pm(x,t,\lambda) = e^{\pm(i \sqrt{\lambda} x + 4 i \sqrt{\lambda}^3 t)} (1 + O(\sqrt{\lambda}^{-1})) \end{equation}
as $\lambda \to \infty$ for $\lambda \in \Omega_s$.

The bound given in part 5 of the proposition follow from subalgebraic bounds on $\alpha_\pm$ in $\mathcal{D}$.
A subalgebraic bound of the form $|\alpha^\pm(t,\lambda)| \le \max\{C \sqrt{\lambda}^3,C'\}$ clearly holds for $\lambda \in \Omega_s \setminus \mathcal{D}$ from the asymptotic descriptions of $\alpha^\pm$.
The difficulty is extending the bound into the removed sector.
To make notation easier we consider $z = \sqrt{\lambda}$ with our choice of branch cut, taking $z$ in the upper half-plane.
We denote the image of $\mathcal{D}$ under the mapping to the $z$ upper half plane by $\mathcal{D}'$.
The domain $\mathcal{D}'$ consists of the upper half plane with excised half-domes (images of half-discs) centered on the real line whose heights are  bounded above by $R'/n$.

We break the domain $\mathcal{D}'$ into two domains.
One of the domains we choose as ${\mathcal{E} = \{z \in \mathbb{C}^+ : \text{Im}(z) > \max\{R',\log(2)/T\}}$.
If we define a function $w$ by $w(x,z) = e^{izx}y_2(x,z^2)$ then $w$ solves
\begin{equation} w(x,z) = \frac{e^{2izx}-1}{2 i z} + \int_0^x \frac{e^{2iz(x-t)}-1}{2 i z}u(t) w(z,t)dt.\end{equation}
By solving via Neumann series we find that
\begin{equation} \left| w(x,z) - \frac{e^{2izx}-1}{2 i z} \right| \le \frac{C}{|z|^2}. \end{equation}
It then follows that
\begin{equation} y_2(x,z^2) = \frac{\sin(zx)}{z} + O(|e^{-izx}| |z|^{-2}). \end{equation}
We have the lower bound
\begin{equation} |\sin(zT)/z| \ge  \frac{1}{|z|} \left(\frac{|e^{-izT}|}{2} - \frac{1}{4} \right) \ge \frac{|e^{-izT}|}{4|z|}. \end{equation}
Therefore, for $|z|$ large enough
\begin{equation} |y_2(T,z^2)| \ge C \frac{|e^{-iTz}|}{|z|}. \end{equation}
The bound 
\begin{equation} |\rho^{\pm 1} (z^2)-y_1(T,z^2)| \le C |e^{-iTz}| \end{equation} for $|z|$ large enough follows from the fact that $y_1(T,\lambda)$ has growth order $\tfrac{1}{2}$ and (\ref{eq:rhoass}).
Combining the above, we obtain
\begin{equation} |\alpha^\pm(t,z^2)| \le \max \{C |z|^3,C'\}\end{equation}
for $z \in \mathcal{E}$.

The other domain which we will label $\mathcal{F}$ consists of the rest of $\mathcal{D}'$.
For $z$ large enough, none of the excised discs will overlap. At this point the non-straight pieces of boundary are deformed semicircles $C_n$ labeled by $n$ with the closest point to $\sqrt{\mu_n}$ a distance away from $\sqrt{\mu_n}$ bounded below by $R'/n$ for some constant $R'$, and furthest point point from $\sqrt{\mu_n}$ a distance away from $\sqrt{\mu_n}$ bounded above by $R''/n$ for some constant $R''$.
As $z \to \infty$ the following asymptotic behavior is valid \cite{MckTr78}
\begin{equation} y_2(T,z^2) = \frac{\sin(zT)}{z} - \frac{\cos(zT)}{2z^2} Q + O(z^{-3}),\end{equation}
and therefore
\begin{equation} \frac{\partial}{\partial z} y_2(T,z^2) = \frac{1}{z}(T \cos(zT) + O(z^{-1})). \end{equation}
The values $z = \sqrt{\mu_n}$ are a distance $O(n^{-1}) \sim \sqrt{\mu_n}^{-1}$ away from the the values of $z$ for which $\cos(zT) = \pm 1$, so for large $n$
\begin{equation} \frac{\partial}{\partial z} y_2(T,z^2)|_{z = \sqrt{\mu_n}} = \frac{T}{\sqrt{\mu_n}} + O (\mu_n^{-1}). \end{equation}
Therefore, on $C_n$ the functions $y_2(T,z^2)$ are approximated by
\begin{equation} y_2(T,z^2) = \left(\frac{T}{\sqrt{\mu_n}} + O(\mu_n^{-1}) \right)(z - \sqrt{\mu_n}) + O((z - \sqrt{\mu_n})^2) \ge \frac{2 R' T}{\pi n^2} \end{equation}
for $n$ sufficiently large.
Since $z \sim \frac{n \pi}{T}$ for $z \in C_n$ for $n$ sufficiently large, we then have ${y_2(T,z^2)^{-1} = O(z^2)}$ for $z \in C_n$ for $n$ sufficiently large. 
For sufficiently large $z$ we also have $y_2(T,z^2)$ monotone increasing in $\text{Im}(z)$ when $\text{Re}(z)$ is held constant.
Therefore $|y_2(T,z^2)^{-1}| \le \max\{C |z|^2,C'\}$ for $z \in \mathcal{F}$.
Combining and exponentiating the bounds in $\mathcal{E}$ and $\mathcal{F}$ give property 5.
\end{proof}

\begin{defn}
Let $\breve V : \mathbb{R}^+ \setminus \{E_j\}_{j = 0}^{2 \mathcal{G}} \to SL(2,\mathbb{C})$ be given by
\begin{equation} \label{eq:KdVjump} \breve V (\lambda) := \begin{cases} (-1)^{k + m(\lambda) - 1} \begin{pmatrix} 0 & i \frac{f^+(\lambda)}{f^-(\lambda)} \\ i \frac{f^-(\lambda)}{f^+(\lambda)} & 0 \end{pmatrix} & \lambda \in (E_{2k-2},E_{2k-1})  \\ (-1)^{k + m(\lambda)} e^{2 i \sigma_3 \sqrt{\lambda} x+ 8 i \sigma_3 \sqrt{\lambda}^3 t} & \lambda \in (E_{2k-1},E_{2k}) \end{cases}, \end{equation}
where recall $m(\lambda)$ is given by (\ref{eq:mcounting}) as
\begin{equation}  m(\lambda) := |\{ k \in \mathbb{N} : \mu_{n_k} \le \lambda, \sigma_{n_k} = 0 \}|. \end{equation}
\end{defn}

We define $\breve \Phi$ via
\begin{equation} \label{eq:defbPhi} \breve \Phi (x,t,\lambda) := \begin{pmatrix} \breve \psi^-(x,t,\lambda) & \breve \psi^+(x,t,\lambda) \\ \breve \psi_x^-(x,t,\lambda) & \breve \psi_x^+(x,t,\lambda) \end{pmatrix}  B(\lambda) e^{i \sigma_3 \sqrt{\lambda} x + 4 i \sigma_3 \sqrt{\lambda}^3 t } \end{equation}
We can now prove the following theorem giving the solution to the Cauchy problem to the KdV equation in terms of a Riemann--Hilbert problem like the Riemann--Hilbert problem given in section 5 for infinite gap Hill's operators, but with a time dependance added to the jump.
The function $\breve \Phi$ is a particular solution to the Riemann--Hilbert problem giving the solution to the KdV equation.
\begin{thm} \label{thm:rhpKdV}
Suppose that $u(x,t)$ is the solution to the KdV equation
\begin{equation}u_t - 6 u u_x + u_{xxx} = 0\end{equation}
with smooth initial data $u(x,0) = u_0(x)$ and let $\Sigma(u_0)$ be the spectral data for corresponding Hill's operator $-\frac{\partial^2}{\partial x^2} + u_0(x)$ (recall the the spectral data for a Hill's operator is defined in (\ref{eq:specdata})).
There exists a solution $\breve \Phi$ to the following Riemann--Hilbert problem, constructed via (\ref{eq:defbPhi}) above.  The first row of the solution is uniquely determined by this Riemann--Hilbert problem.
\begin{rhp} \label{rhp:KdV}
For fixed $x,t \in \mathbb{R}$ find a $2 \times 2$ matrix valued function $\breve \Phi(x,t,\lambda)$ such that:
\begin{enumerate}

\item $\breve \Phi(x,t,\lambda)$ is a holomorphic function of $\lambda$ for $\lambda \in \mathbb{C} \setminus \mathbb{R}^+$.

\item $\breve \Phi_{\pm} (x,t,\lambda)$ are continuous functions of $\lambda \in \mathbb{R}^+ \setminus \{E_j\}_{j = 0}^{2 \mathcal{G}}$ that have at worst quartic root singularities on $\{E_k\}_{k = 0}^{2 \mathcal{G}}$.

\item $\breve \Phi_{\pm} (x,t,\lambda)$ satisfy the jump relation $\breve \Phi_+(x,t,\lambda) = \breve \Phi_-(x,t,\lambda)  \breve V (x,t,\lambda)$.

\item $\breve \Phi (x,t,\lambda)$ has an asymptotic description of the form
\begin{equation} \label{eq:assphibrebe} \breve \Phi(x,t,\lambda) =  \begin{pmatrix} 1 & 1 \\ -i \sqrt{\lambda} & i \sqrt{\lambda} \end{pmatrix} \left( I + O\left(\sqrt{\lambda}^{-1} \right) \right) B(0,\lambda) \end{equation}
as $\lambda \to \infty$ with $\lambda \in \Omega_s$ for some $0 < s < \tfrac{\pi}{8}$.

\item There exist positive constants $c$, and $M$ such that $|\breve \phi_{ij} (x,t,\lambda)| \le M e^{c|\lambda|^2} $ for all $\lambda \in \mathcal{D}$.
\end{enumerate}
\end{rhp}
Let $x,t \in \mathbb{R}$ be fixed, then $\tilde \Phi(x,t,\lambda)$ solves Riemann--Hilbert problem \ref{rhp:KdV} if and only if
\begin{equation} \label{eq:solspacePhi2} \tilde \Phi (x,t,\lambda) = \begin{pmatrix} 1 & 0 \\ \alpha & 1 \end{pmatrix} \breve \Phi (x,t, \lambda) \end{equation}
for some $\alpha$ that is constant in $\lambda$.
The value $u(x,t)$ of the solution $u$ to the KdV equation with initial data $u_0$ evaluated at $(x,t)$ can be recovered from any solution $\breve \Phi(x,t,\lambda)$ to this Riemann--Hilbert problem as
\begin{equation} \label{eq:KdVrecon} u(x,t) = 2 i \frac{\partial}{\partial x} \lim_{\lambda \to \infty} \sqrt{\lambda} \left( 1 -  b_{11}(0,\lambda)^{-1} \breve\phi_{11} (x,t,\lambda)  \right). \end{equation}
\end{thm}

The existence of a solution $\breve \Phi(x,t,\lambda)$ follows from the existence of a solution to the periodic KdV equation, considering (\ref{eq:defbPhi}), and combining the construction of a solution to Riemann--Hilbert problem \ref{rhp:Phi}, along with with properties \ref{proper:e} of $e^\pm(t,\lambda)$ and (\ref{eq:204forrevision}).
Solutions to the time dependent version of Rieman--Hilbert problem $\ref{rhp:Phi}$ are related to solutions of $\ref{rhp:KdV}$ by the following relation
\begin{equation}
\breve \Phi(x,t,\lambda) = \Phi(x,t,\lambda) B(t,\lambda)^{-1} \begin{pmatrix} e^-(t,\lambda) & 0 \\ 0 & e^+(t,\lambda) \end{pmatrix}  B(0,\lambda) e^{4 i \sigma_3 \sqrt{\lambda}^3 t}. 
\end{equation}
The relation between $\tilde \Phi$ and $\breve \Phi$ by (\ref{eq:solspacePhi2}) giving uniqueness of the first row can be proven by applying the steps of the proof of the corresponding result in Theorem \ref{thm:rhpPhi}.
The time dependence does not affect that validity of the steps taken in the proof.
The formula (\ref{eq:KdVrecon}) generating a solution solution to the KdV equation from the solution to Riemann--Hilbert problem \ref{rhp:KdV} follows from the formula (\ref{eq:reconstructionformulapotential}) generating the potential solving the inverse spectral problem from the solution to Riemann--Hilbert problem \ref{rhp:Phi}.
This is because the effect of multiplying $\psi^\pm(x,t,\lambda)$ by $e^\pm(t,\lambda)$ in (\ref{eq:psitobrevepsi}) only leads to an additive term in
\begin{equation} \lim_{\lambda \to \infty} \sqrt{\lambda} \left( 1 -  b_{11}(0,\lambda)^{-1} \breve\phi_{11} (x,t,\lambda) \right) \end{equation}
that is constant in $x$.
If one multiplies out the phase factors appearing in the asymptotic condition (\ref{eq:assbpsi}) and conjugates the jump matrix accordingly one gets the jump condition on the gaps for the above Riemann--Hilbert problem.

\begin{rmk}
We close this section by remarking that a {\it general existence theory} for Riemann-Hilbert problem \ref{rhp:KdV}, providing existence of the time-dependent potential $u(x,t)$ under conditions described in Remark \ref{rmk:quasiperiod}, would be extremely valuable.  As observed in Remark \ref{rmk:quasiperiod}, this would imply that the entries in the first row of a solution to Riemann-Hilbert problem \ref{rhp:KdV} with these more general conditions solve Hill's equation, yielding the existence of a potential $u(x,t)$.  Thus, such an existence result would provide an alternative existence theory for periodic solutions of the KdV equation.  But more importantly, an existence theorem could apply in cases when the data does not correspond to a periodic potential, and one would obtain new results concerning existence of bounded solutions of the KdV equation that are not necessarily periodic.

\end{rmk}

\section{Conditions on Periodicity in Space and Time}
\label{sec:Periodicity}

In the previous sections we have shown that a bounded periodic potential is determined by the spectral data as described in Definition \ref{def:SpecData}, and established that, as $u$ evolves according to the KdV equation, it can be determined through the solution of Riemann--Hilbert problem \ref{rhp:KdV}.    In this section we consider our matrix Riemann--Hilbert problem \ref{rhp:KdV} (under the KdV evolution), which can be stated for any candidate spectral data $\{E_{k}, \mu_{k},\sigma_{k}\}$.  We suppose that it possesses a solution (which by our work must have a unique first row), and that it determines a potential $u(x,t)$, and we determine explicit conditions under which the potential is periodic in space, and also conditions under which the potential is periodic in time.

If the potential is known to be periodic, the classical theory of Hill's equation yields the existence of the Floquet multiplier $\rho(\lambda)$, which satisfies a scalar Riemann-Hilbert problem (conditions 1.(a)-(d) in Theorem \ref{thm:spaceperiod} below). We prove that the converse is true:  the existence of a solution to this scalar Riemann-Hilbert problem implies the potential is periodic.  Moreover, we prove there is an analogous scalar Riemann-Hilbert problem for which the existence of a solution implies temporal periodicity.

For this purpose, we consider a ``candidate spectral data,'' namely a sequence $\{E_{k}\}_{k=0}^{2 \mathcal{G}+1}$, a sequence $\{\mu_{n_k} \}_{k= 1}^{\mathcal{G}}$ and a sequence $\{\sigma_{n_k}\}_{k=1}^{\mathcal{G}}$ ($\mathcal{G}$ could be finite or infinite), satisfying
\begin{enumerate}
\item $E_{0} < E_{1} < \cdots $,
\item $\mu_{n_k} \in [E_{2k-1},E_{2k}]$,
\item $R = \max_{j} |E_{2j}-E_{2j-1}| < \infty$, 
\item for each $k$, $\sigma_{n_k} \in \{ -1,0,1\}$.
\item In case $\mathcal{G}$ is infinite:  there exists $N, C$ such that  $E_{n} > C n^{2}$ for all $n > N$.
\item In case $\mathcal{G}$ is infinite: there exists $K$ such that the discs $D_{n_k}$ of radius $R$ centered at ${(E_{2k}+E_{2k-1})/2}$ are disjoint for $k \ge K$.
\end{enumerate}

As explained in remark \ref{rmk:1}, with this candidate spectral data, all the quantities needed to define the Riemann-Hilbert problem \ref{rhp:KdV} are well defined, and we may pose the Riemann-Hilbert problem.
 We assume that there is a solution to this problem, and that it determines a potential $u(x,t)$.  
\begin{thm}
\label{thm:spaceperiod}
Suppose that there exists a solution to Riemann--Hilbert problem \ref{rhp:KdV} defined from the candidate spectral data above, determining the potential $u(x,t)$, via (\ref{eq:KdVrecon}).  
\begin{enumerate}
\item If there exists a function $r_1(\lambda)$ such that
\begin{enumerate} \item $r_1$ is holomorphic in $\mathbb{C} \setminus \sigma(L)$ with continuous boundary values $r_{1\pm}$ on $\sigma(L)$ from above and below,
\item $r_1$ satisfies the jump relation $r_{1+}(\lambda) = r_{1-}(\lambda)^{-1}$ for $\lambda \in \sigma(L)$,
\item $r_1$ satisfies the asymptotic condition
$r_1(\lambda) = e^{i \sqrt{\lambda} T_1}(1 + O(\sqrt{\lambda}^{-1}))$
for $\lambda \in \Omega_s$ for some $T_1 > 0$ and $0 < s <\frac{\pi}{8}$.
\item $r_1 \in \mathcal{A}_2(\mathcal{D})$,
\end{enumerate}
then $u(x+T_1,t) = u(x,t)$.
\item If there exists a function $r_2(\lambda)$ such that
\begin{enumerate} \item $r_2$ is holomorphic in $\mathbb{C} \setminus \sigma(L)$ with continuous boundary values $r_{2\pm}$ on $\sigma(L)$ from above and below,
\item $r_2$ satisfies the jump relation $r_{2+}(\lambda) = r_{2-}(\lambda)^{-1}$ for $\lambda \in \sigma(L)$,
\item $r_2$ satisfies the asymptotic condition
$r_2(\lambda) = e^{4i \sqrt{\lambda}^3 T_2}(1 + O(\sqrt{\lambda}^{-1}))$
for $\lambda \in \Omega_s$ for some $T_2 > 0$ and $0<s<\tfrac{\pi}{8}$,
\item $r_2 \in \mathcal{A}_2(\mathcal{D})$,
\end{enumerate}
then $u(x,t+T_2) = u(x,t)$.
\end{enumerate}
\end{thm}

Note that when $r_1(\lambda)$ exists it is equal to the Floquet multiplier $\rho(\lambda)$.

\begin{proof}
Suppose the $r_1(\lambda)$ exists and consider the function
\begin{equation} \tilde \Phi(x,t,\lambda) = \Phi(x+T_1,t,\lambda) r_1(\lambda)^{\sigma_3} e^{-i\sqrt{\lambda} \sigma_3 T_1}. \end{equation}
We would like to show $\tilde \Phi$ solves the same Riemann--Hilbert problem as $\Phi$.
Properties 1 and 2 are clear.
On the spectral band $[E_{2k-2},E_{2k-1}]$ the jump relation for $\tilde \Phi(x,t,\lambda)$ is computed as
\begin{align}
\tilde \Phi_+ (x,t,\lambda) & =  \Phi_+(x+T_1,t,\lambda) r_{1+}(\lambda)^{\sigma_3} e^{-i\sqrt{\lambda}_+ \sigma_3 T_1} \\
& = \Phi_-(x+T_1,t,\lambda)  (-1)^{k + m(\lambda) - 1} \begin{pmatrix} 0 & i \frac{f^+(\lambda)}{f^-(\lambda)} \\ i \frac{f^-(\lambda)}{f^+(\lambda)} & 0 \end{pmatrix} r_{1-}(\lambda)^{-\sigma_3} e^{-i\sqrt{\lambda} \sigma_3 T_1} \\
& = \tilde \Phi_-(x,t,\lambda) (-1)^{k + m(\lambda) - 1} \begin{pmatrix} 0 & i \frac{f^+(\lambda)}{f^-(\lambda)} \\ i \frac{f^-(\lambda)}{f^+(\lambda)} & 0 \end{pmatrix}
\end{align}
and on the gap $[E_{2k-1},E_{2k}]$ the jump is computed as
\begin{align}
\tilde \Phi_+(x,t,\lambda) & = \Phi_+(x+T_1,t,\lambda) r_1(\lambda)^{\sigma_3} e^{-i\sqrt{\lambda}_+ \sigma_3 T_1} \\
& = \Phi_-(x+T_1,t,\lambda) (-1)^{k + m(\lambda)} e^{2 i \sigma_3 \sqrt{\lambda} (x + T_1)+ 8 i \sigma_3 \sqrt{\lambda}^3 t} r_{1}(\lambda)^{\sigma_3} e^{-i\sqrt{\lambda} \sigma_3 T_1} \\
& = \tilde \Phi_-(x,t,\lambda) (-1)^{k + m(\lambda)} e^{2 i \sigma_3 \sqrt{\lambda} x + 8 i \sigma_3 \sqrt{\lambda}^3 t}.
\end{align}
Therefore $\tilde \Phi$ satisfies property 3.
Property 4 follows from the fact that
\begin{equation} r_1(\lambda)^{\sigma_3} e^{-i\sqrt{\lambda} \sigma_3 T_1} = 1 + O(\sqrt{\lambda}^{-1}). \end{equation}
Property 5 follows from the fact that $\mathcal{A}_2(\mathcal{D})$ is an algebra.

Theorem \ref{thm:rhpKdV} then tells us that for all $x$,$t$,
\begin{equation}  \tilde \Phi (x,t,\lambda) = \begin{pmatrix} 1 & 0 \\ \alpha & 1 \end{pmatrix} \breve \Phi (x,t, \lambda) \end{equation}
for some constant $\alpha$.
The first row $(\phi_{11}(x,t,\lambda),\phi_{12}(x,t,\lambda))$ of $\Phi$ is equal to the first row $(\tilde \phi_{11}(x,t,\lambda),\tilde \phi_{12}(x,t,\lambda))$ of $\tilde \Phi$.
In particular,
\begin{equation} \phi_{11}(x,t,\lambda) = \phi_{11}(x+T_1,t,\lambda) r_1(\lambda) e^{-i\sqrt{\lambda} T_1}. \end{equation}
Then (\ref{eq:KdVrecon}) implies $u(x+T_1,t) = u(x,t)$.

Now suppose that $r_2(\lambda)$ exists, and determine $\tilde \Phi$ by
\begin{equation} \tilde \Phi(x,t,\lambda) = \Phi(x,t+T_2,\lambda) r_2(\lambda)^{\sigma_3} e^{-4i\sqrt{\lambda}^3 \sigma_3 T_2}. \end{equation}
We would again like to show $\tilde \Phi$ solves the same Riemann--Hilbert problem as $\Phi$.
Properties 1 and 2 are clear.
On the spectral band $[E_{2k-2},E_{2k-1}]$ the jump relation for $\tilde \Phi(x,t,\lambda)$ is computed as
\begin{align}
\tilde \Phi_+ (x,t,\lambda) & =  \Phi_+(x,t+T_2,\lambda) r_{2+}(\lambda)^{\sigma_3} e^{-4i\sqrt{\lambda}_+^3 \sigma_3 T_2} \\
& = \Phi_-(x,t+T_2,\lambda)  (-1)^{k + m(\lambda) - 1} \begin{pmatrix} 0 & i \frac{f^+(\lambda)}{f^-(\lambda)} \\ i \frac{f^-(\lambda)}{f^+(\lambda)} & 0 \end{pmatrix} r_{2-}(\lambda)^{-\sigma_3} e^{-4i\sqrt{\lambda}^3 \sigma_3 T_2} \\
& = \tilde \Phi_-(x,t,\lambda) (-1)^{k + m(\lambda) - 1} \begin{pmatrix} 0 & i \frac{f^+(\lambda)}{f^-(\lambda)} \\ i \frac{f^-(\lambda)}{f^+(\lambda)} & 0 \end{pmatrix}
\end{align}
and on the gap $[E_{2k-1},E_{2k}]$ the jump is computed as
\begin{align}
\tilde \Phi_+(x,t,\lambda) & = \Phi_+(x+T_1,t,\lambda) r_2(\lambda)^{\sigma_3} e^{-4i\sqrt{\lambda}_+^3 \sigma_3 T_2} \\
& = \Phi_-(x+T_1,t,\lambda) (-1)^{k + m(\lambda)} e^{2 i \sigma_3 \sqrt{\lambda} x + 8 i \sigma_3 \sqrt{\lambda}^3( t+ T_2)} r_{2}(\lambda)^{\sigma_3} e^{-4i\sqrt{\lambda}^3 \sigma_3 T_2} \\
& = \tilde \Phi_-(x,t,\lambda) (-1)^{k + m(\lambda)} e^{2 i \sigma_3 \sqrt{\lambda} (x + T_1)+ 8 i \sigma_3 \sqrt{\lambda}^3 t}.
\end{align}
Therefore $\tilde \Phi$ satisfies property 3.
Property 4 follows from the fact that
\begin{equation} r_2(\lambda)^{\sigma_3} e^{-4i\sqrt{\lambda}^3 \sigma_3 T_2} = 1 + O(\sqrt{\lambda}^{-1}). \end{equation}
Property 5 follows from the fact that $\mathcal{A}_2(\mathcal{D})$ is an algebra.

Theorem \ref{thm:rhpKdV} then tells us that for all $x$,$t$,
\begin{equation}  \tilde \Phi (x,t,\lambda) = \begin{pmatrix} 1 & 0 \\ \alpha & 1 \end{pmatrix} \breve \Phi (x,t, \lambda) \end{equation}
for some constant $\alpha$.
Then
\begin{equation} \phi_{11}(x,t,\lambda) = \phi_{11}(x,t+T_2,\lambda) r_2(\lambda) e^{-4i\sqrt{\lambda}^3 T_2}, \end{equation}
so (\ref{eq:KdVrecon}) implies $u(x,t+T_2) = u(x,t)$.
\end{proof}

\section{Baker--Akhiezer Functions}

Up to this point, we have made no use of the underlying spectral curve.
However, the approach we have taken has an interpretation in terms of Baker--Akheizer functions.
We will now provide a corollary to the theorem that gives a geometric interpretation to solutions of Riemann--Hilbert problem \ref{rhp:KdV}.
The idea here is to connect the Riemann--Hilbert problem theory discussed in this paper back to what has been done before in terms of the theory of Riemann--Surfaces.
Let $\Sigma \subset \mathbb{C} \times \mathbb{C}$ be the the curve defined by
\begin{equation} \label{eq:polyeq} w^2 = P(\lambda) = \lambda \prod_{k = 1}^{ \mathcal{G}} \frac{T^4}{n_k^4 \pi^4}(\lambda - E_{2k-1})(\lambda - E_{2k}). \end{equation}
This curves is diffeomorphic via a holomorphic map to the desingularization of the curve defined by $w^2 = \Delta(\lambda)^2-4$ by 2 point blowups at the degenerate Dirichlet eigenvalues.
We choose not to compactify $\Sigma$ as a topological space because if $\mathcal{G} = \infty$ then the compactification of $\Sigma$ is not smooth at $\infty$ (this is because of an accumulation of infinitely many `holes' at $\infty$).
The projection $\pi((\lambda,w)) \to \lambda$ onto the $\lambda$ plane has two inverses under composition,
\begin{equation}
\pi_+^{-1} (\lambda) = (\lambda,\sqrt{P(\lambda)}),\end{equation}
\begin{equation} \pi_-^{-1}(\lambda) = (\lambda,-\sqrt{P(\lambda)}),
\end{equation}
which agree only at the branch points.
The images of $\pi_+^{-1}$ and $\pi_-^{-1}$ we will write as $\Sigma_+$ and $\Sigma_-$ so that $\Sigma$ is a union of $\Sigma_+$ and $\Sigma_-$. 
\begin{cor} \label{cor:BAfunction}
For every $x,t$ there is a unique meromorphic function $\breve \psi(x,t,p)$ on $\Sigma$ with only simple poles at $p_k = (\mu_{n_k},\sigma_{n_k} P(\mu_{n_k}))$ for $k = 1,2,\dots, \mathcal{G}$ such that
\begin{equation} \label{ex1} \breve \psi(x,t,\pi_+^{-1}(\lambda)) = e^{i \sqrt{\lambda} x + 4 i \sqrt{\lambda}^3 t}(1 + O(\sqrt{\lambda}^{-1})), \end{equation}
\begin{equation*}\breve \psi(x,t,\pi_-^{-1}(\lambda)) = e^{-i \sqrt{\lambda} x- 4 i \sqrt{\lambda}^3 t }(1 + O(\sqrt{\lambda}^{-1})) \end{equation*}
as $\lambda \to \infty$ with $\lambda \in \Omega_s$, and
\begin{equation} \label{ex2}
\breve \psi(x,t,\pi_+^{-1}(\cdot)) \in \mathcal{A}_2(\mathcal{D}),
\end{equation}
\begin{equation*}
\breve \psi(x,t,\pi_-^{-1}(\cdot)) \in \mathcal{A}_2(\mathcal{D}).
\end{equation*}
The function $\breve \psi$ is known as the Baker--Akhiezer function for the KdV equation.

\end{cor}

Note that while the pole condition on the $\Sigma$ can be defined in a coordinate independent manner, the asymptotic condition (\ref{ex1}) and bound (\ref{ex2}) are expressed in terms of the complex $\lambda$ coordinate on $\Sigma$.

\begin{proof}[Sketch of Proof]
If $\psi$ is defined by
\begin{equation} \label{ex3} \breve \psi(x,t,p) = \begin{cases} \breve \psi^+(x,t,\pi(p))  & p \in \Sigma^+ \\ \breve \psi^-(x,t,\pi(p)) & p \in \Sigma^-  \end{cases}
\end{equation}
then:
\begin{enumerate}
\item The poles of either $\breve \psi^+$ or $\breve \psi^-$ at the Dirichlet eigenvalues in the interior gaps, and the square root singularities of both $\breve \psi^+$ and $\breve \psi^-$ on the edges of the non-degenerate gaps map to poles on $p_k = (\mu_{n_k},\sigma_{n_k} P(\mu_{n_k}))$.
\item The asymptotic expansions of $\breve \psi^\pm$ obtained by multiplying the asymptotic expansion from lemma \ref{lma:BFass} by the asymptotic expansion of $e^\pm$ immediately imply the asymptotic condition (\ref{ex1}).
\item The containments (\ref{ex2}) follow immediately from (\ref{eq:psibound}).
\end{enumerate}
To show that $\psi$ is the only such function, suppose there is a second such function $\tilde\psi(x,t,p)$.
It is easy to show that if
\begin{equation} \breve \Psi^{(2)}(x,t,\lambda) := \begin{pmatrix} \frac{1}{2}  (\tilde \psi(x,t,\pi_-^{-1}(\lambda)) + \breve \psi^-(x,t,\lambda)) &  \frac{1}{2} (\tilde \psi(x,t,\pi_+^{-1}(\lambda)) + \breve\psi^+(x,t,\lambda))\\ ({\breve \psi^-})'(x,t,\lambda) & ({\breve \psi^+})'(x,t,\lambda) \end{pmatrix}  \end{equation}
then
\begin{equation} \breve{\Phi}^{(2)}(x,t,\lambda) := \breve{\Psi}^{(2)}(x,t,\lambda) B(\lambda) e^{i \sigma_3 \sqrt{\lambda} x + 4 i \sigma_3 \sqrt{\lambda}^3 t}
\end{equation}
must solve Riemann--Hilbert problem \ref{rhp:KdV}.
Therefore, by Theorem \ref{thm:rhpKdV} the first row of $\breve{\Phi}^{(2)}(x,t,\lambda)$ is equal to the first row of $\breve\Phi(x,t,\lambda)$.
Setting them equal implies that $\tilde \psi = \breve \psi$.
\end{proof}

\section{Finite Gap Solutions and Riemann Hilbert Problem}

In the finite gap case (\ref{eq:polyeq}) can be replaced by
\begin{equation} w^2 =\lambda \prod_{n = 1}^{2g} (\lambda-E_n) \end{equation}
and produce an equivalent Riemann--surface $\Sigma $.
The $\frac{T^4}{n_k^4 \pi^4}$ was only necessary to guarantee convergence in the infinite gap case. \\

Before moving on we review some elementary algebraic geometry following \cite{Dubrovin81}.
We introduce a homology basis $a_i,b_j$ for $j = 1,\dots,g$ satisfying $a_i \circ b_j = \delta_{ij}$, $a_i \circ a_j=0$ and $b_i\circ b_j=0$, where $\circ$ indicates minimal crossing number in the homology class of $a_i$ and $b_j$.
We then introduce the basis of Abelian differentials of the first kind on $\Sigma$ normalized such that
\begin{equation} \int_{a_j} \omega_i = 2 \pi i \delta_{ij}. \end{equation}
The Abel map ${A}:\Sigma \to \mathbb{C}^g$ with base point $p_0$ is given by
\begin{equation} {A}(p) = \int_{p_0}^p {\omega}. \end{equation}
The Abel map extends to a mapping of divisors by adding together the Abel maps of the points in the divisor.

The Riemann matrix $\tau$ for $\Sigma$ is
\begin{equation} \tau_{ij} = \int_{b_j} \omega_i. \end{equation}
The Riemann theta function corresponding to Riemann 
matrix $\tau$ is
\begin{equation}
\theta({z},\tau) = \sum_{{m} \in \mathbb{Z}^g} \exp\left(\left<  {m}, {z} \right> + \frac{1}{2} \left< {m}, \tau {m} \right> \right).
\end{equation}
The Riemann matrix is symmetric with negative definite real part.
We also introduce the Abelian differentials of the second kind $\omega^{(j)}$ for $j = 1,2,3,\dots$ uniquely determined by the principle part
\begin{equation} \omega^{(j)} \sim d \sqrt{\lambda}^j \end{equation}
as $\lambda \to \infty$, where $d$ is the exterior derivative, and the condition
\begin{equation} \int_{a_k} \omega^{(j)} = 0. \end{equation}

We now consider how the algebraic geometry so far discussed relates to finite gap potentials.
Let us consider the function
\begin{equation} \label{eq:formBAfun} { \breve \psi(x,t,p) = \exp \left( i\int_{p_0}^p  (\omega^{(1)} x + 4  \omega^{(3)} t) \right) \frac{\theta({A}(p) - {A}(\mathcal{P}) + i {\Omega}^{(1)} x + 4 i {\Omega}^{(3)} t  - {K},\tau)\theta(- {A}(\mathcal{P}) - {K},\tau)}{\theta({A}(p) - {A}(\mathcal{P}) - {K},\tau) \theta( - {A}(\mathcal{P}) + i {\Omega}^{(1)} x + 4 i {\Omega}^{(3)} t  - {K},\tau)},}\end{equation}
where
\begin{equation} \Omega_j^{(\ell)} = \int_{b_j} \omega^{(\ell)},\end{equation}
and
\begin{equation} K_j = \frac{2 \pi i + \tau_{jj}}{2} - \frac{1}{2 \pi i} \sum_{\ell \ne j} \int_{a_\ell} \left( \omega_{\ell}(p) \int_{p_0}^p \omega_j \right)\end{equation}
is the vector of Riemann constants, and $\mathcal{P}$ is the divisor of poles consisting of the direct sum of the points $p_k$ defined in corollary \ref{cor:BAfunction}.
The function (\ref{eq:formBAfun}) satisfies the properties of the finite genus Baker--Akhiezer function \cite{Dubrovin81}, and the values of $\breve \psi$ on $\Sigma^+$ and $\Sigma^-$ can be projected back onto the plane to produce two functions $\breve \psi^+$ and $\breve \psi^-$.
The function $\breve \psi^\pm$ can be used to create a solution $\breve \Phi$ to Riemann--Hilbert problem \ref{rhp:KdV} via (\ref{eq:defbPhi}).
The solution to the KdV equation recovered from the Baker--Akhiezer function is given by the Matveev--Its formula
\begin{equation} \label{eq:MI} u(x,t) = -2 \frac{\partial^2}{\partial x^2} \log \theta( i {\Omega}^{(1)} x + i {\Omega}^{(3)} t - {A}(\mathcal{P}) - {K} ; \tau ) + c \end{equation}
where $c$ is a constant \cite{Dubrovin81}.
At fixed time $t$, the functions $\breve \psi^\pm(x,t)$ are solutions to Hill's equation with potential $u(x,t)$.
The function $c$ can be changed by translating the function values of the potential or equivalently by translating the spectrum itself.

Summarizing the finite gap construction, one starts with a hyper-elliptic Riemann--Surface and produces both a solution to Hill's equation (\ref{eq:formBAfun}), and the potential in Hill's equation (\ref{eq:MI}).
Continuing, one may also use (\ref{eq:formBAfun}) to construct a solution to the Riemann-Hilbert problem \ref{rhp:KdV}, from which the potential (\ref{eq:MI}) may also be obtained as in (\ref{eq:KdVrecon}).

In connection with Section \ref{sec:Periodicity}, we note that here, the potential $u(x,t)$ is {not assumed to be periodic}.   Indeed, if the half phases in the vector ${ \Omega}^{(1)}$ satisfy
\begin{eqnarray}
\label{eq:Commensurate}
\Omega_j^{(1)} = \frac{\ell_j \pi}{T_1}, \mbox{ for some constant $T_1$ and integers $\ell_j$ with $j = 1,2,\dots,g$}
\end{eqnarray}
then the potential $u(x,t)$ so constructed is periodic with period $T_{1}$, and the Floquet multiplier $\rho(\lambda)$ exists, yielding a solution to the scalar Riemann-Hilbert problem of Theorem \ref{thm:spaceperiod}.  However, if (\ref{eq:Commensurate}) does not hold, one says that the phases are not commensurate, and the potential is quasi-periodic.  It is still the case that the Riemann-Hilbert problem possesses a solution, and this solution determines the potential $u(x,t)$, it is just that the potential is not periodic in space.
If the scalar Riemann-Hilbert of Theorem \ref{thm:spaceperiod} were solvable, then the potential would be periodic by Theorem \ref{thm:spaceperiod}.
Therefore, we can conclude that the scalar Riemann-Hilbert problem of Theorem \ref{thm:spaceperiod} is not solvable in the non commensurate case.

If the phases are commensurate (i.e. if (\ref{eq:Commensurate}) holds), then the Floquet multiplier $r_{1}(\lambda) = \rho(\lambda)$ can be constructed explicitly. This is achieved as follows.  

Let us consider the function
\begin{equation}
f(x,t,\lambda) =  x \int_{p_0}^\lambda \omega^{(1)}  + 4 t\int_{p_0}^{\lambda} \omega^{(3)}
\end{equation}
where we identify $\lambda$ with a corresponding point in $\Sigma^+ \subset \Sigma$.
The base point $p_0$ can be set so that
\begin{equation} f_+(x,t,\lambda) + f_-(x,t,\lambda) = 0\end{equation}
for $\lambda \in \sigma(L)$, and
\begin{equation} f_+(x,t,\lambda) - f_-(x,t,\lambda) = \int_{b_j} \omega^{(1)} x + 4 \omega^{(3)} t =  \Omega_j^{(1)} x + 4 \Omega_j^{(3)} t \end{equation}
for $\lambda \in [E_{2j-1},E_{2j}]$.
The first can be seen by drawing contours $\gamma^\pm$ from $\infty$ to $\lambda$ in $\mathbb{C}^+$ and $\mathbb{C}^-$ and hit the opposite sides of the branch cut.
If we switch the orientation $\gamma^-$ and switch to the sheet $\Sigma^-$ we get the same result upon integration.
This new contour in conjunction with $\gamma^+$ can be put together to produce a closed loop.

Now if (\ref{eq:Commensurate}) holds, then using $T_{1}$ and $f$ we can construct the function
\begin{equation} \rho(\lambda) = e^{i f(T_1,0,\lambda)}\end{equation}
that is holomorphic in $\mathbb{C} \setminus \sigma(L)$ and $\rho_+(\lambda)  = \rho_-(\lambda)^{-1}$ and $\log(\rho(\lambda)) = i T_1 \sqrt{\lambda}$.
The properties on $\rho$ are precisely the properties of the Floquet discriminant $\rho(\lambda) = r_{1}(\lambda)$ enumerated in Theorem \ref{thm:spaceperiod}.

Alternatively, consider the case that 
\begin{eqnarray}
\label{eq:TimeCommensurate}
4 \Omega_j^{(3)} = \frac{\ell_j \pi}{T_2}, \mbox{ for some constant $T_1$ and integers $\ell_j$ with $j = 1,2,\dots,g$}\ .
\end{eqnarray}
In this case, we can construct the function
\begin{equation} r_2(\lambda) = e^{i f(0,T_2,\lambda)}. \end{equation}
This function is holomorphic in $\mathbb{C} \setminus \sigma(L)$ and $r_{2+}(\lambda)  = r_{2-}(\lambda)^{-1}$ and $\log(r_2(\lambda)) = 4i T_2 \sqrt{\lambda}^3$.
These properties of $r_2(\lambda)$ are the properties of the function $r_{2}$ appearing in Theorem \ref{thm:spaceperiod}, and we learn that the potential $u(x,t)$ appearing in (\ref{eq:MI}) is periodic in $t$ with period $T_{2}$, which is consistent.

\section*{Funding}

KM was supported in part by the National Science Foundation under grant DMS-1733967. PN was supported in part by the National Science Foundation under grant DMS-1715323.

\section*{Acknowledgment} The authors would like to thank the reviewers for their detailed reports; their recommendations improved this paper. Much of the research was conducted when PN was a PhD student at the University of Arizona.

\end{document}